\documentclass[10pt]{article}
\usepackage{amsthm}
\usepackage{amsmath}
\usepackage{amssymb}
\usepackage{makeidx}

\theoremstyle{definition}
\newtheorem{definition}{Definition}[section]

\newtheorem{example}[definition]{Example}

\newtheorem{remark}[definition]{Remark}

\theoremstyle{plain}
\newtheorem{theorem}[definition]{Theorem}
\newtheorem{lemma}[definition]{Lemma}
\newtheorem{proposition}[definition]{Proposition}
\newtheorem{corollary}[definition]{Corollary}

\newcommand{\BA}{\mathbb{BA}}

\newcommand{\PBZ}{PBZ$^{\ast }$}

\def\N{{\mathbb N}}

\def\V{{\mathbb V}}

\def\C{{\mathbb C}}

\def\S{{\mathrm S}}

\def\1{\textcircled{1}}
\def\2{\textcircled{2}}
\newcommand{\WCL}{\mathbb{WCL}}
\newcommand{\WDCL}{\mathbb{WDCL}}
\newcommand{\WDL}{\mathbb{WDL}}
\newcommand{\dotcup}{\stackrel{\scriptscriptstyle \bullet }{\cup }}

\begin{document}

\title{On Nontrivial Weak Dicomplementations and the\\ Lattice Congruences that Preserve Them}
\author{Leonard KWUIDA$^1$ and Claudia MURE\c{S}AN$^2$\\ 
{\small {\sc $^1$Bern University of Applied Sciences, $^2$University of Bucharest}}\\ 
{\small $^1$leonard.kwuida@bfh.ch; $^2$cmuresan@fmi.unibuc.ro, $^2$c.muresan@yahoo.com}}
\date{\today }
\maketitle

\begin{abstract} We study the existence of nontrivial and of representable (dual) weak complementations, along with the lattice congruences that preserve them, in different constructions of bounded lattices, then use this study to determine the finite (dual) weakly complemented lattices with the largest numbers of congruences, along with the structures of their congruence lattices. It turns out that, if $n\geq 7$ is a natural number, then the four largest numbers of congruences of the $n$--element (dual) weakly complemented lattices are: $2^{n-2}+1$, $2^{n-3}+1$, $5\cdot 2^{n-6}+1$ and $2^{n-4}+1$. For smaller numbers of elements, several intermediate numbers of congruences appear between the elements of this sequence. After determining these numbers, along with the structures of the (dual) weakly complemented lattices having these numbers of congruences, we derive a similar result for weakly dicomplemented lattices.


\end{abstract}

\section{Introduction}
\label{introduction}

{\bf Weakly dicomplemented lattices} arise as abstractions of concept algebras, introduced by Rudolf Wille when modelling negation on concept lattices \cite{Wi00}; they are bounded lattices endowed with two unary operations, called {\em weak complementation} and {\em dual weak complementation}, together forming the {\em weak dicomplementation}, which generalize Boolean algebras; in fact if $(L,\wedge,\vee,\overline{\cdot },0,1)$ is a Boolean algebra, then $(L,\wedge,\vee,\overline{\cdot },\overline{\cdot },0,1)$ is a weakly dicomplemented lattice. Their bounded lattice reducts endowed with the (dual) weak complementation are called {\bf (dual) weakly complemented lattices}. Any bounded lattice can be endowed with the {\em trivial weak dicomplementation}, formed of the weak complementation that sends $1$ to $0$ and all other elements to $1$ and the dual weak complementation that sends $0$ to $1$ and all other elements to $0$. If $X$ is a set and $c$ a closure operator on $\mathcal{P}(X)$, then the set of closed subsets of $X$ forms a lattice and the operation $cA\mapsto c(X\setminus cA)$ is a weak complementation on the lattice of $c$-closed subsets of $X$. Dually, if $k$ is a kernel operator on $\mathcal{P}(X)$, then $kB\mapsto k(X\setminus kB)$ is a dual weak complementation. Additional examples can be found in \cite{KM10}.

We take a purely lattice--theoretical approach to the study of these algebras and investigate the existence of nontrivial weak dicomplementations on ordinal and horizontal sums of bounded lattices, as well as (co)atomic bounded lattices with different numbers of (co)atoms and determine the lattice congruences that preserve those (dual) weak complementations. This allows us to determine all weak dicomplementations that can be defined on bounded lattices with certain lattice structures. Since the notion of representability is important in the study of these algebras, we also determine which of those (dual) weak complementations are representable.

After this preliminary investigation, we are able to determine the structures of the finite (dual) weakly complemented lattices, as well as weakly dicomplemented lattices, that have the largest numbers of congruences out of the algebras of the same kind with the same numbers of elements, along with the structures of their congruence lattices; we do this for weakly complemented lattices in our main theorem: Theorem \ref{cgwclfin}, for dual weakly complemented lattices in Corollary \ref{cgwdclfin} and for weakly dicomplemented lattices in Corollary \ref{cgwdlfin}. This problem, which is also related to that of the representability of lattices as congruence lattices of various kinds of algebras, has been investigated for lattices in \cite{free} and later in \cite{gcze,eucfifin}, for semilattices in \cite{gcz} and for bounded involution lattices, pseudo--Kleene algebras and antiortholattices in \cite{eucgbiklaol}; its counterpart for infinite algebras has been studied in \cite{eucard,eunoucard} and later in \cite{gccm,eucginfbi}. 

\section{Preliminaries}
\label{preliminaries}

We will designate all algebras by their underlying sets. By {\em trivial algebra} we mean one--element algebra. Recall that a variety $\V $ of similar algebras is said to be {\em semi--degenerate} iff no nontrivial member of $\V $ has trivial subalgebras. We denote by $\N $ the set of the natural numbers and by $\N ^*=\N \setminus \{0\}$. $\dotcup $ will be the disjoint union. For any set $M$, $|M|$ denotes the cardinality of $M$ and ${\cal P}(M)$ denotes the power set of $M$; ${\rm Part}(M)$ and $({\rm Eq}(M),\cap ,\vee ,\mbox{\bf =}_M,M^2)$ will be the complete lattices of the partitions and the equivalences on $M$, respectively, where ${\rm Eq}(M)$ is ordered by the set inclusion, while the order $\leq $ of ${\rm Part}(M)$ is given by: for any $\pi ,\rho \in {\rm Part}(M)$, $\pi \leq \rho $ iff every class of $\rho $ is a union of classes from $\pi $, and $eq:{\rm Part}(M)\rightarrow {\rm Eq}(M)$ will be the canonical lattice isomorphism. If $\{M_1,\ldots ,M_n\}\in {\rm Part}(M)$ for some $n\in \N ^*$, then $eq(\{M_1,\ldots ,M_n\})$ will be streamlined to $eq(M_1,\ldots ,M_n)$. We will use Gr\" atzer's notation for the lattice operations \cite{gratzer}.

Let $\V $ be a variety of algebras of a similarity type $\tau $, $\C $ a class of algebras with reducts in $\V $ and $A$ and $B$ algebras with reducts in $\V $. Then $\S _{\V }(\C )$ will denote the class of the subalgebras of the $\tau $--reducts of the members of $\C $, and $\S _{\V }(\{A\})$ will simply be denoted $\S _{\V }(A)$. We will abbreviate by $A\cong _{\V }B$ the fact that the $\tau $--reducts of $A$ and $B$ are isomorphic.

${\rm Con}_{\V }(A)$ will be the complete lattice of the congruences of the $\tau $--reduct of $A$, and, for any $n\in \N ^*$ and any constants $\kappa _1,\ldots ,\kappa _n$ from $\tau $, we denote by ${\rm Con}_{\V \kappa _1\ldots \kappa _n}(A)=\{\theta \in {\rm Con}_{\V }(A)\ |\ (\forall \, i\in [1,n])\, (\kappa _i^A/\theta =\{\kappa _i\})\}$, and, by extension, for any $a_1,\ldots ,a_n\in A$, by ${\rm Con}_{\V a_1\ldots a_n}(A)=\{\theta \in {\rm Con}_{\V }(A)\ |\ (\forall \, i\in [1,n])\, (a_i/\theta =\{a_i\})\}$, which is a complete sublattice of ${\rm Con}_{\V }(A)$ and thus a bounded lattice, according to the straightforward consequence \cite[Lemma 2.(iii)]{pbzsums} of \cite[Corollary 2,p.51]{gralgu}.

Recall that $A$ is subdirectly irreducible in $\V $ iff $\mbox{\bf =}_A$ is strictly meet--irreducible in the bounded lattice ${\rm Con}_{\V }(A)$.

For any $X\subseteq A^2$ and any $a,b\in A$, we denote by $Cg_{\V ,L}(X)$ the congruence of the $\tau $--reduct of $A$ generated by $X$ and, for brevity, the principal congruence $Cg_{\V ,L}(\{(a,b)\})$ by $Cg_{\V ,L}(a,b)$.

If $\V $ is the variety of (bounded) lattices, then the index $\V $ will be eliminated from the notations above.

For any poset $(P,\leq )$, ${\rm Min}(P)$ and ${\rm Max}(P)$ will be the set of the minimal elements and that of the maximal elements of $(P,\leq )$, respectively.

For any (bounded) lattice $L$, $\prec $ will denote the cover relation of $L$, $L^d$ will be the dual of $L$ and, if $L$ has a $0$, then the set of the atoms of $L$ will be denoted by ${\rm At}(L)$, while, if $L$ has a $1$, then the set of the coatoms of $L$ will be denoted by ${\rm CoAt}(L)$. For any $a,b\in L$, we denote by $[a,b]_L=[a)_L\cap (b]_L$ the interval of $L$ bounded by $a$ and $b$; we eliminate the index $L$ from this notation if $L$ is $\N $ endowed with the natural order. Also, we denote by $a||b$ the fact that $a$ and $b$ are incomparable. Note that a lattice congruence of $L$ is complete iff all its classes are intervals. We denote by ${\rm Con}_{\rm cplt}(L)$ the bounded sublattice of ${\rm Con}(L)$ consisting of the complete lattice congruences of $L$. ${\rm Ji}(L)$, ${\rm Sji}(L)$, ${\rm Mi}(L)$ and ${\rm Smi}(L)$ will be the sets of the join--irreducible, strictly join--irreducible, meet--irreducible and strictly meet--irreducible elements of $L$, respectively. For any $a\in {\rm Smi}(L)$, we will denote the unique successor of $a$ in $L$ by $a^+$, or by $a^{+L}$ if the lattice $L$ needs to be specified; similarly, for any $b\in {\rm Sji}(L)$, we denote by $b^-$ or $b^{-L}$ the unique predecessor of $b$ in $L$. For all $n\in \N ^*$, we denote by ${\cal C}_n$ the $n$--element chain.

Let $L$ be a lattice with top element $1^L$ and $M$ be a lattice with bottom element $0^M$. Recall that the ordinal sum of $L$ with $M$ is the lattice obtained by stacking $M$ on top of $L$ and glueing the top element of $L$ together with the bottom element of $M$. For the rigorous definition, we consider the equivalence on the disjoint union of $L$ with $M$ that only collapses $1^L$ with $0^M$: $\varepsilon =eq(\{\{1^L,0^M\}\}\cup \{\{x\}\ |\ x\in (L\setminus \{1^L\})\dotcup (M\setminus \{0^M\})\})\in {\rm Eq}(L\dotcup M)$. Since $\varepsilon \cap L^2=\mbox{\bf =}_L\in {\rm Con}(L)$ and $\varepsilon \cap M^2=\mbox{\bf =}_M\in {\rm Con}(M)$, we can identify $L$ with $L/\varepsilon $ and $M$ with $M/\varepsilon $ by identifying $x$ with $x/\varepsilon $ for every $x\in L\dotcup M$. Now we define the {\em ordinal sum} of $L$ with $M$ to be the lattice $L\oplus M=(L\oplus M,\leq ^{L\oplus M})$, where $L\oplus M=(L\dotcup M)/\varepsilon $, which becomes $L\cup M$ with the previous identification, and $\leq ^{L\oplus M}=\leq ^L\cup \leq ^M\cup \{(x,y)\ |\ x\in L,y\in M\}$. Of course, $L\oplus M$ becomes a bounded lattice iff $L$ and $M$ are bounded lattices. For any $\alpha \in {\rm Eq}(L)$ and any $\beta \in {\rm Eq}(M)$, we denote by $\alpha \oplus \beta =eq((L/\alpha \setminus \{1^L/\alpha \})\cup (M/\beta \setminus \{0^M/\beta \})\cup \{\{1^L/\alpha \cup 0^M/\beta \}\})\in {\rm Eq}(L\oplus M)$; clearly, $\alpha \oplus \beta \in {\rm Con}(L\oplus M)$ iff $\alpha \in {\rm Con}(L)$ and $\beta \in {\rm Con}(M)$, and the map $(\alpha ,\beta )\mapsto \alpha \oplus \beta $ is a lattice isomorphism from ${\rm Con}(L\times M)\cong {\rm Con}(L)\times {\rm Con}(M)$ to ${\rm Con}(L\oplus M)$. Clearly, the operation $\oplus $ on bounded lattices is associative, and so is the operation $\oplus $ on the congruences of such lattices or the equivalences of their underlying sets.

Now let $L$ and $M$ be nontrivial bounded lattices. Recall that the horizontal sum of $L$ with $M$ is the nontrivial bounded lattice obtained by glueing the bottom elements of $L$ and $M$ together, glueing their top elements together and letting all other elements of $L$ be incomparable to every other element of $M$. For the rigorous definition, we consider the equivalence on the disjoint union of $L$ with $M$ that only collapses the bottom element of $L$ with that of $M$ and the top element of $L$ with that of $M$: $\varepsilon =eq(\{\{0^L,0^M\},\{1^L,1^M\}\}\cup \{\{x\}\ |\ x\in (L\setminus \{0^L,1^L\})\dotcup (M\setminus \{0^M,1^M\})\})\in {\rm Eq}(L\dotcup M)$. Since $\varepsilon \cap L^2=\mbox{\bf =}_L\in {\rm Con}(L)$ and $\varepsilon \cap M^2=\mbox{\bf =}_M\in {\rm Con}(M)$, we can identify $L$ with $L/\varepsilon $ and $M$ with $M/\varepsilon $ by identifying $x$ with $x/\varepsilon $ for every $x\in L\dotcup M$. Now we define the {\em horizontal sum} of $L$ with $M$ to be the nontrivial bounded lattice $L\boxplus M=(L\boxplus M,\leq ^{L\boxplus M},0^{L\boxplus M},1^{L\boxplus M})$, where $L\boxplus M=(L\dotcup M)/\varepsilon =L\cup M$ in view of the previous identification, $\leq ^{L\boxplus M}=\leq ^L\cup \leq ^M$, $0^{L\boxplus M}=0^L=0^M$ and $1^{L\boxplus M}=1^L=1^M$. For any $\alpha \in {\rm Eq}(L)\setminus \{L^2\}$ and any $\beta \in {\rm Eq}(M)\setminus \{M^2\}$, we denote by $\alpha \boxplus \beta =eq((L/\alpha \setminus \{0/\alpha ,1/\alpha \})\cup (M/\beta \setminus \{0/\beta ,1/\beta \})\cup \{\{0/\alpha \cup 0/\beta ,1/\alpha \cup 1/\beta \}\})\in {\rm Eq}(L\boxplus M)\setminus \{(L\boxplus M)^2\}$.

Note that the horizontal sum of nontrivial bounded lattices is commutative and associative, it has ${\cal C}_2$ as a neutral element and it can be generalized to arbitrary families of nontrivial bounded lattices. The operation $\boxplus $ on proper equivalences of those bounded lattices is commutative and associative, as well. The five--element modular non--distributive lattice is ${\cal M}_3={\cal C}_3\boxplus {\cal C}_3\boxplus {\cal C}_3$ and the five--element non--modular lattice is ${\cal N}_5={\cal C}_3\boxplus {\cal C}_4$. For any nonzero cardinality $\kappa $, the modular lattice ${\cal M}_{\kappa }$ of length $3$ with $\kappa $ atoms is the horizontal sum of $\kappa $ copies of the three--element chain.

\section{The Algebras We Are Working With}
\label{thealg}

\begin{definition} Let $(L,\wedge ,\vee ,0,1)$ be a bounded lattice and $\cdot ^{\Delta }$, $\cdot ^{\nabla }$ be unary operations on $L$. 

The algebra $(L,\wedge ,\vee ,\cdot ^{\Delta },0,1)$ is called a {\em weakly complemented lattice} iff the unary operation $\cdot ^{\Delta }$ is order--reversing and, for all $x,y\in L$, $x^{\Delta \Delta }\leq x$ and $(x\wedge y)\vee (x\wedge y^{\Delta })=x$. In this case, the operation $\cdot ^{\Delta }$ is called {\em weak complementation} on the bounded lattice $L$.

The algebra $(L,\wedge ,\vee ,\cdot ^{\nabla },0,1)$ is called a {\em dual weakly complemented lattice} iff the unary operation $\cdot ^{\nabla }$ is order--reversing and, for all $x,y\in L$, $x\leq x^{\nabla \nabla }$ and $(x\vee y)\wedge (x\vee y^{\nabla })=x$. In this case, the operation $\cdot ^{\nabla }$ is called {\em dual weak complementation} on $L$.

The algebra $(L,\wedge ,\vee ,\cdot ^{\Delta },\cdot ^{\nabla },0,1)$ is called a {\em weakly dicomplemented lattice} iff $(L,\wedge ,\vee ,\cdot ^{\Delta },0,1)$ is a weakly complemented lattice and $(L,\wedge ,\vee ,\cdot ^{\nabla },0,1)$ is a dual weakly complemented lattice. In this case, the pair $(\cdot ^{\Delta },\cdot ^{\nabla })$ is called {\em weak dicomplementation} on $L$.\end{definition}

If $L$ is a bounded lattice, $\cdot ^{\Delta }$ is a weak complementation and $\cdot ^{\nabla }$ is a dual weak complementation on $L$, then the abbreviated notations $(A,\cdot ^{\Delta })$, $(A,\cdot ^{\nabla })$ and $(A,\cdot ^{\Delta },\cdot ^{\nabla })$ will designate the weakly complemented lattice $(A,\wedge ,\vee ,\cdot ^{\Delta },0,1)$, the dual weakly complemented lattice $(A,\wedge ,\vee ,\cdot ^{\nabla },0,1)$, and the weakly dicomplemented lattice $(A,\wedge ,\vee ,\cdot ^{\Delta },\cdot ^{\nabla },0,1)$, respectively.

We denote by $\BA $, $\WCL $, $\WDCL$ and $\WDL $ the varieties of Boolean algebras, weakly complemented lattices, dual weakly complemented lattices and weakly dicomplemented lattices, respectively.

It is immediate that any $L\in \WCL $ satisfies the identities: $0^{\Delta }\approx 1$, $1^{\Delta }\approx 0$, $x\vee x^{\Delta }\approx 1$ and $(x\wedge y)^{\Delta }\approx x^{\Delta }\vee y^{\Delta }$, as well as the quasiequations: $x^{\Delta }\approx 0\rightarrow x\approx 1$, $x\wedge y\approx 0\rightarrow x^{\Delta }\geq y$ and $x^{\Delta }\leq y\rightarrow y^{\Delta }\leq x$.

Dually, any $L\in \WDCL $ satisfies the identities: $0^{\nabla }\approx 1$, $1^{\nabla }\approx 0$, $x\wedge x^{\nabla }\approx 0$ and $(x\vee y)^{\nabla }\approx x^{\nabla }\wedge y^{\nabla }$, as well as the quasiequations: $x^{\nabla }\approx 1\rightarrow x\approx 0$, $x\vee y\approx 1\rightarrow x^{\nabla }\leq y$ and $x^{\nabla }\geq y\rightarrow y^{\nabla }\geq x$.

Additionally, any $L\in \WDL $ satisfies: $x^{\nabla }\leq x^{\Delta }$.

Clearly, $\BA \subseteq \WCL \cap \WDCL $, because the Boolean complementation of any Boolean algebra $A$ is a weak complementation, as well as a dual weak complementation on $A$, hence $\BA $ can be considered as a subvariety of $\WDL $ with an extended signature, by endowing each Boolean algebra with a second unary operation equalling its Boolean complementation. Moreover, from the above it is easy to notice that, in a weakly dicomplemented lattice $L$, the weak complementation coincides with the weak dicomplementation iff $L$ is a Boolean algebra and each of these operations coincides with the Boolean complementation of $L$. Hence $\BA $ with the extended signature is exactly the subvariety of $\WDL $ axiomatized by $x^{\Delta }\approx x^{\nabla }$.

If $A$ is a Boolean algebra, then we will always consider the weak complementation and dual weak complementation on $A$ that equal its Boolean complementation, unless mentioned otherwise. Note that, with these operations, we have ${\rm Con}_{\WDL }(A)={\rm Con}_{\WCL }(A)={\rm Con}_{\WDCL }(A)={\rm Con}_{\BA }(A)={\rm Con}(A)$.

Let us notice that any bounded lattice $L$ can be organized as a weakly complemented lattice by endowing it with the {\em trivial weak complementation}: $x^{\Delta }=1$ for all $x\in L\setminus \{1\}$, and it can be organized as a dual weakly complemented lattice by endowing it with the {\em trivial dual weak complementation}: $x^{\nabla }=0$ for all $x\in L\setminus \{0\}$, hence it can be organized as a weakly dicomplemented lattice by endowing it with the {\em trivial weak dicomplementation}: $(\cdot ^{\Delta },\cdot ^{\nabla })$, where $\cdot ^{\Delta }$ is the trivial weak complementation and $\cdot ^{\nabla }$ is the trivial dual weak complementation on $L$.

Since $\WDL \vDash \{x\wedge y\approx 0\rightarrow x^{\Delta }\geq y,x\vee y\approx 1\rightarrow x^{\nabla }\leq y\}$, it clearly follows that, for any cardinality $\kappa \geq 3$, the bounded lattice ${\cal M}_{\kappa }$ can only be endowed with the trivial weak dicomplementation.

Of course, for any $L\in \WDL $, if we consider its reducts from $\WCL $ and $\WDCL $, then ${\rm Con}_{\WDL }(L)={\rm Con}_{\WCL }(L)\cap {\rm Con}_{\WDCL }(L)$. It is routine to prove that a lattice congruence of a bounded lattice $L$ preserves the trivial weak complementation on $L$ iff its $1$--class is a singleton and, dually, it preserves the trivial dual weak complementation on $L$ iff its $0$--class is a singleton, therefore, if we endow $L$ with the trivial weak complementation and the trivial dual weak complementation, then ${\rm Con}_{\WCL }(L)={\rm Con}_1(L)\cup \{L^2\}$ and ${\rm Con}_{\WDCL }(L)={\rm Con}_0(L)\cup \{L^2\}$, so that ${\rm Con}_{\WDL }(L)={\rm Con}_{01}(L)\cup \{L^2\}$.

Clearly, the trivial weak complementation is the (pointwise) largest weak complementation on $L$, while the trivial dual weak complementation is the (pointwise) smallest dual weak complementation on $L$. If $(\cdot ^{\Delta 1},\cdot ^{\nabla 1})$ and $(\cdot ^{\Delta 2},\cdot ^{\nabla 2})$ are weak dicomplementations on a bounded lattice $L$, then we say that $(\cdot ^{\Delta 1},\cdot ^{\nabla 1})$ is {\em smaller} than $(\cdot ^{\Delta 2},\cdot ^{\nabla 2})$ iff $\cdot ^{\Delta 1}$ is pointwise smaller than $\cdot ^{\Delta 2}$ and $\cdot ^{\nabla 1}$ is pointwise larger than $\cdot ^{\nabla 2}$. According to this definition, the trivial weak dicomplementation is the largest weak dicomplementation on any bounded lattice.

As mentioned in Section \ref{introduction}, the basic example of a weakly dicomplemented lattice is the canonical concept algebra associated to a context. A {\em context} is a triple $(G,M,I)$, where $G$ and $M$ are sets and $I\subseteq G\times M$ is a binary relation; the elements of $G$ are called {\em objects}, and  elements of $M$ are called {\em attributes}. For every $A\subseteq G$ and every $B\subseteq M$, we denote by: $\begin{cases}
A^{\prime }=\{m\in M\ |\ (\forall \, a\in A)\, (aIm)\},\\ 
B^{\prime }=\{g\in G\ |\ (\forall \, b\in B)\, (gIb)\}.
\end{cases}$. The operations $\cdot ^{\prime }:{\cal P}(G)\rightarrow {\cal P}(M)$ and $\cdot ^{\prime }:{\cal P}(M)\rightarrow {\cal P}(G)$ are called {\em derivation of objects} and {\em of attributes}, respectively. The {\em concept algebra associated to the context $(G,M,I)$} is the weakly dicomplemented lattice $({\cal B}(G,M,I),\wedge ,\vee ,\cdot ^{\Delta },\cdot ^{\nabla },0,1)$, where: $({\cal B}(G,M,I)=\{(A,B)\ |\ A\subseteq G,B\subseteq M,A^{\prime \prime }=A,B^{\prime \prime }=B\},\wedge ,\vee )$ is a lattice ordered by: $\leq =\{((A,B),(C,D))\in {\cal B}(G,M,I)^2\ |\ A\subseteq C,B\supseteq D\}$, with first element $0=(\emptyset ^{\prime \prime },M)$ and last element $1=(G,G^{\prime })$, endowed with the weak complementation defined by $(A,B)^{\Delta }=((G\setminus A)^{\prime \prime },(G\setminus A)^{\prime })$ and the dual weak complementation defined by $(A,B)^{\nabla }=((M\setminus B)^{\prime },(M\setminus B)^{\prime \prime })$ for all $(A,B)\in {\cal B}(G,M,I)$.

Whenever $J$ is a join--dense subset and $M$ is a meet--dense subset of a complete lattice $L$, we have $L\cong {\cal B}(J,M,\leq )$, because the map $\varphi _{L,J,M}:L\rightarrow {\cal B}(J,M,\leq )$, defined by $\varphi _{L,J,M}(x)=(J\cap (x]_L,M\cap [x)_L)$ for all $x\in L$, is a lattice isomorphism. In this case, $L$ can be endowed with the weak dicomplementation $(\cdot ^{\Delta J},\cdot ^{\nabla M})$ defined by $x^{\Delta J}=\bigvee (J\setminus (x]_L)$ and $x^{\nabla M}=\bigwedge (M\setminus [x)_L)$ for all $x\in L$. Note that the subsets $J\cup \{0\}$ and $J\setminus \{0\}$ of $L$ are also join--dense, the subsets $M\cup \{1\}$ and $M\setminus \{1\}$ of $L$ are also meet--dense, $\cdot ^{\Delta J}=\cdot ^{\Delta (J\cup \{0\})}=\cdot ^{\Delta (J\setminus \{0\})}$, $\cdot ^{\nabla M}=\cdot ^{\nabla (M\cup \{1\})}=\cdot ^{\nabla (M\setminus \{1\})}$ and, for every $H\in \{J,J\cup \{0\},J\setminus \{0\}\}$ and every $N\in \{M,M\cup \{1\},M\setminus \{1\}\}$, $\varphi _{L,J,M}=\varphi _{L,H,N}$, thus, furthermore, since this map is a weakly dicomplemented lattice isomorphism, the canonical weakly dicomplemented lattices ${\cal B}(J,M,\leq )$ and ${\cal B}(H,N,\leq )$ coincide. We say that a weak complementation $\cdot ^{\Delta }$, respectively of a dual weak complementation $\cdot ^{\nabla }$ on $L$ is {\em representable} iff $\cdot ^{\Delta }=\cdot ^{\Delta J}$ for some join--dense subset $J$ of $L$, respectively $\cdot ^{\nabla }=\cdot ^{\nabla M}$ for some meet--dense subset $M$ of $L$; we say that a weak dicomplementation $(\cdot ^{\Delta },\cdot ^{\nabla })$  on $L$ is {\em representable} iff $\cdot ^{\Delta }$ and $\cdot ^{\nabla }$ are representable. Clearly, the trivial weak dicomplementation on $L$ is representable, since it equals $(\cdot ^{\Delta L},\cdot ^{\nabla L})$.

In particular, if $L$ is a complete dually algebraic lattice, then ${\rm Sji}(L)$ is join--dense in $L$, thus so is ${\rm Ji}(L)$, hence $L\cong {\cal B}({\rm Sji}(L),L,\leq )\cong {\cal B}({\rm Ji}(L),L,\leq )$, so that $L$ can be endowed with the weak complementation $\cdot ^{\Delta {\rm Ji}(L)}$, as well as the weak complementation $\cdot ^{\Delta {\rm Sji}(L)}$. Dually, if $L$ is an algebraic lattice, then ${\rm Smi}(L)$ is meet--dense in $L$, thus so is ${\rm Mi}(L)$, hence $L\cong {\cal B}(L,{\rm Smi}(L),\leq )\cong {\cal B}(L,{\rm Mi}(L),\leq )$, which can be endowed with the dual weak complementations $\cdot ^{\nabla {\rm Mi}(L)}$ and $\cdot ^{\nabla {\rm Smi}(L)}$.

If a lattice $L$ is complete, algebraic and dually algebraic, then ${\rm Sji}(L)$ is join--dense in $L$ and ${\rm Smi}(L)$ is meet--dense in $L$, therefore $L\cong {\cal B}({\rm Sji}(L),{\rm Smi}(L),\leq )$, that can be endowed with the weak dicomplementation $(\cdot ^{\Delta {\rm Sji}(L)},\cdot ^{\nabla {\rm Smi}(L)})$, which, according to \cite[p.$236$]{bglk}, is the smallest weak dicomplementation on $L$. Consequently, $L$ has nontrivial weak complementations iff $\cdot ^{\Delta {\rm Sji}(L)}$ is nontrivial, and $L$ has nontrivial dual weak complementations iff $\cdot ^{\nabla {\rm Smi}(L)}$ is nontrivial.

In particular, if $L$ is a finite lattice, then $L\cong {\cal B}({\rm Ji}(L),{\rm Mi}(L),\leq )$ and the smallest weak dicomplementation on $L$ is $(\cdot ^{\Delta {\rm Ji}(L)},\cdot ^{\nabla {\rm Mi}(L)})$, so that $L$ has nontrivial weak complementations iff $\cdot ^{\Delta {\rm Ji}(L)}$ is nontrivial, and $L$ has nontrivial dual weak complementations iff $\cdot ^{\nabla {\rm Mi}(L)}$ is nontrivial.

Now let $L$ be a bounded lattice and $J$, $M$ subsets of $L$. We consider the following condition:$$\neg sg\Delta (L,J,M):\quad (\exists \, m,n\in M)\, (J\subseteq (m]_L\cup (n]_L)$$Clearly, if $\bigvee J$ exists in $L$, in particular if $J$ is finite, then condition $\neg sg\Delta (L,J,M)$ implies: $\neg sg\Delta (L,\{\bigvee J\},M):\quad (\exists \, m,n\in M)\, (\bigvee J\in (m\vee n]_L)$. Hence, if, furthermore, $1\notin M$ and $\bigvee J=1$, in particular if $J$ is join--dense in $L$, then $\neg sg\Delta (L,J,M)$ is equivalent to:$$(\exists \, m,n\in M)\, (m\neq n\mbox{ and }J\subseteq (m]_L\cup (n]_L)$$

If $L$ is a complete algebraic and dually algebraic lattice, then, by the above, the smallest weak dicomplementation on $L$ is of the form $(\cdot ^{\Delta J},\cdot ^{\nabla M})$ for the join--dense subset $J={\rm Sji}(L)$ and the meet--dense subset $M={\rm Smi}(L)$ of $L$, therefore:\begin{itemize}
\item $L$ has nontrivial weak complementations iff condition $\neg sg\Delta (L,{\rm Sji}(L),{\rm Smi}(L))$ is satisfied.\end{itemize}

In particular, if $L$ is a finite lattice, then:\begin{itemize}
\item $L$ has nontrivial weak complementations iff condition $\neg sg\Delta (L,{\rm Ji}(L),{\rm Mi}(L)\setminus \{1\})$ is satisfied.\end{itemize}

And, of course, dually for the dual weak complementations. All the following results on weak complementations involving condition $\neg sg\Delta (L,J,M)$ can be dualized, using condition $\neg sg\Delta (L^d,J,M)$ in results on dual weak complementations.

\section{Weak Dicomplementations on Ordinal Sums}
\label{ordsum}

Throughout this section, $K$, $L$ and $M$ will be nontrivial bounded lattices and we will consider the ordinal sums of bounded lattices $A=L\oplus M$, with $L\cap M=\{c\}$, and $B=L\oplus K\oplus M$. We also consider weak dicomplementations on $A$ and $B$, which we denote in the same way: $(\cdot ^{\Delta },\cdot ^{\nabla })$.

Since the lattice $M$ is nontrivial and $A$ satisfies $x\vee x^{\Delta }\approx 1$, it follows that $a^{\Delta }=1$ for all $a\in L$. We have $1^{\Delta }=0$ and, clearly, $\cdot ^{\Delta }\mid _{M\setminus \{1\}}$ is the restriction to $M\setminus \{1\}$ of a weak complementation on $M$, that we will denote by $\cdot ^{\Delta \mbox{\textcircled{M}}}$.

Dually, $\cdot ^{\nabla }$ must be defined by: $b^{\nabla }=0$ for all $b\in M$, $0^{\nabla }=1$ and $\cdot ^{\nabla }\mid _{L\setminus \{0\}}=\cdot ^{\nabla \mbox{\textcircled{L}}}\mid _{L\setminus \{0\}}$ for some dual weak complementation $\cdot ^{\nabla \mbox{\textcircled{L}}}$ on $L$.

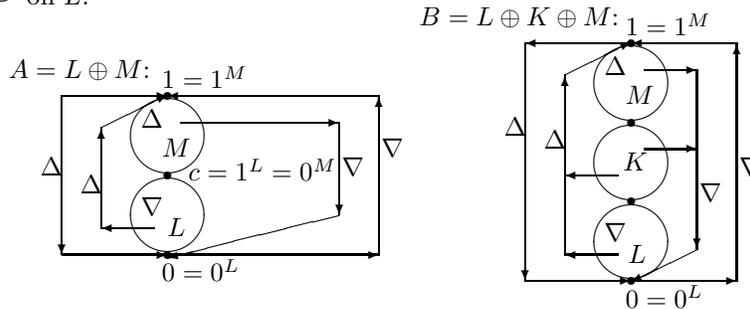
\begin{center}\begin{tabular}{cc}\begin{picture}(40,85)(0,0)
\put(-40,67){$A=L\oplus M$:}
\put(20,0){\circle*{3}}
\put(20,30){\circle*{3}}
\put(20,60){\circle*{3}}
\put(20,15){\circle{29}}
\put(20,45){\circle{29}}
\put(20,0){\vector(1,0){80}}
\put(100,0){\vector(0,1){60}}
\put(100,60){\vector(-1,0){80}}
\put(101,37){$\nabla $}
\put(86,30){$\nabla $}
\put(-20,60){\vector(1,0){40}}
\put(-20,60){\vector(0,-1){60}}
\put(-28,30){$\Delta $}
\put(-20,0){\vector(1,0){40}}
\put(20,7){$L$}
\put(10,15){$\nabla $}
\put(10,48){$\Delta $}
\put(25,50){\vector(1,0){60}}
\put(85,50){\vector(0,-10){35}}
\put(85,15){\vector(-4,-1){65}}
\put(18,37){$M$}
\put(15,10){\vector(-1,0){20}}
\put(-5,10){\vector(0,1){38}}
\put(-5,48){\vector(2,1){25}}
\put(-13,23){$\Delta $}
\put(18,63){$1=1^M$}
\put(28,28){$c=1^L=0^M$}
\put(18,-10){$0=0^L$}
\end{picture}
&\hspace*{120pt}
\begin{picture}(40,85)(0,0)
\put(-60,87){$B=L\oplus K\oplus M$:}
\put(20,-10){\circle*{3}}
\put(20,20){\circle*{3}}
\put(20,50){\circle*{3}}
\put(20,80){\circle*{3}}
\put(20,5){\circle{29}}
\put(20,35){\circle{29}}
\put(20,65){\circle{29}}
\put(10,5){$\nabla $}
\put(10,68){$\Delta $}
\put(19,-3){$L$}
\put(17,32){$K$}
\put(18,57){$M$}
\put(18,83){$1=1^M$}
\put(18,-20){$0=0^L$}
\put(15,0){\vector(-1,0){20}}
\put(15,30){\vector(-1,0){20}}
\put(-5,0){\vector(0,1){68}}
\put(-5,68){\vector(2,1){25}}
\put(-13,40){$\Delta $}
\put(-20,-10){\vector(1,0){40}}
\put(20,80){\vector(-1,0){40}}
\put(-20,80){\vector(0,-1){90}}
\put(-28,45){$\Delta $}
\put(20,-10){\vector(1,0){40}}
\put(60,80){\vector(-1,0){40}}
\put(60,-10){\vector(0,1){90}}
\put(61,30){$\nabla $}
\put(46,20){$\nabla $}
\put(25,70){\vector(1,0){20}}
\put(25,40){\vector(1,0){20}}
\put(45,70){\vector(0,-1){68}}
\put(45,2){\vector(-2,-1){25}}
\end{picture}\end{tabular}\end{center}\vspace*{10pt}

Consequently, since the ordinal sum of bounded lattices is associative, $\cdot ^{\Delta }$ must be defined on $B$ by: $a^{\Delta }=1$ for all $a\in K\cup L$, $1^{\Delta }=0$ and $\cdot ^{\Delta }\mid _{M\setminus \{1\}}$ is the restriction to $M\setminus \{1\}$ of a weak complementation on $M$, while $\cdot ^{\nabla }$ must be defined on $B$ by: $b^{\nabla }=0$ for all $b\in L\cup M$, $0^{\nabla }=1$ and $\cdot ^{\nabla }\mid _{K\setminus \{0\}}$ is the restriction to $K\setminus \{0\}$ of a weak dicomplementation on $K$.

Therefore: $A$ can only be endowed with the trivial weak complementation iff $B$ can only be endowed with the trivial weak complementation iff $M$ can only be endowed with the trivial weak complementation.

Dually: $A$ can only be endowed with the trivial dual weak complementation iff $B$ can only be endowed with the trivial dual weak complementation iff $L$ can only be endowed with the trivial dual weak complementation.

Consequently: $A$ can only be endowed with the trivial weak dicomplementation iff $B$ can only be endowed with the trivial weak dicomplementation iff $L$ can only be endowed with the trivial dual weak complementation and $M$ can only be endowed with the trivial weak complementation.

In particular, $L\oplus {\cal C}_2$ can only be endowed with the trivial weak complementation, while ${\cal C}_2\oplus M$ can only be endowed with the trivial dual weak complementation, thus ${\cal C}_2\oplus L\oplus {\cal C}_2$ can only be endowed with the trivial weak dicomplementation.

\begin{proposition}\begin{itemize}
\item ${\rm Con}_{\WCL }(A)=\{\alpha \oplus \beta \ |\ \alpha \in {\rm Con}(L),\beta \in {\rm Con}_{\WCL 1}(M)\}\cup \{\nabla _A\}\cong ({\rm Con}(L)\times $\linebreak ${\rm Con}_{\WCL 1}(M))\oplus {\cal C}_2$.
\item ${\rm Con}_{\WDCL }(A)=\{\alpha \oplus \beta \ |\ \alpha \in {\rm Con}_{\WDCL 0}(L),\beta \in {\rm Con}(M)\}\cup \{\nabla _A\}\cong ({\rm Con}_{\WDCL 0}(L)\times {\rm Con}(M))\oplus {\cal C}_2$.
\end{itemize}\label{cgordsum}\end{proposition}

\begin{proof} ${\rm Con}(A)=\{\alpha \oplus \beta \ |\ \alpha \in {\rm Con}(L),\beta \in {\rm Con}(M)\}$. Let $\alpha \in {\rm Con}(L)$ and $\beta \in {\rm Con}(M)$.

Then $\alpha \oplus \beta \in {\rm Con}_{\WCL }(A)$ iff, for all $x,y\in A$, whenever $x(\alpha \oplus \beta )y$, it follows that $x^{\Delta }(\alpha \oplus \beta )y^{\Delta }$.

If $x,y\in L$, then $x^{\Delta }=1=y^{\Delta }$, so the preservation of $\cdot ^{\Delta }$ is trivially satisfied by $\alpha \oplus \beta $ in this case.

If $x,y\in M\setminus \{1\}$, then $x^{\Delta }(\alpha \oplus \beta )y^{\Delta }$ iff $x^{\Delta \mbox{\textcircled{M}}}\beta y^{\Delta \mbox{\textcircled{M}}}$, therefore, if the $1$--class of $\beta $ is a singleton, then $\alpha \oplus \beta $ preserves the $\cdot ^{\Delta }$ iff $\beta $ preserves the $\cdot ^{\Delta \mbox{\textcircled{M}}}$.

If $x=1$ and $y\in M\setminus \{1\}$ are such that $x(\alpha \oplus \beta )y$, case in which the $1$--class of $\beta $ is not a singleton, then $x^{\Delta }(\alpha \oplus \beta )y^{\Delta }$ iff $0=1^{\Delta }(\alpha \oplus \beta )y^{\Delta \mbox{\textcircled{M}}}\in M$, which is equivalent to $\alpha =L^2$ and $c\beta y^{\Delta \mbox{\textcircled{M}}}$, the latter of which holds when $\beta $ preserves the $\cdot ^{\Delta \mbox{\textcircled{M}}}$.

If $x=1$ and $y\in K$, then $x=1(\alpha \oplus \beta )y$ iff $\beta =M^2$ and $c\alpha y$, and, if $\alpha \oplus \beta $ preserves the $\cdot ^{\Delta }$, then $0=1^{\Delta }(\alpha \oplus \beta )y^{\Delta }=1$, hence $\alpha \oplus \beta =A^2$, thus we also have $\alpha =L^2$.

If $x\in M\setminus \{1\}$ and $y\in K$ are such that $x(\alpha \oplus \beta )y$, so that $x\beta c$, then $x^{\Delta \mbox{\textcircled{M}}}=x^{\Delta }(\alpha \oplus \beta )y^{\Delta }=1$ holds when $\beta $ preserves the $\cdot ^{\Delta \mbox{\textcircled{M}}}$.

Since the ordinal sums of congruences of the forms above clearly preserve the $\cdot ^{\Delta }$, we have the equivalence: $\alpha \oplus \beta $ preserves the $\cdot ^{\Delta }$ iff $\beta $ preserves the $\cdot ^{\Delta \mbox{\textcircled{M}}}$ and, whenever the $1$--class of $\beta $ is not a singleton, we have $\alpha =L^2$. 

Dually, $\alpha \oplus \beta $ preserves the $\cdot ^{\nabla }$ iff $\alpha $ preserves the $\cdot ^{\nabla \mbox{\textcircled{L}}}$ and, whenever the $0$--class of $\alpha $ is not a singleton, we have $\beta =M^2$.\end{proof}

\begin{corollary}\begin{itemize}
\item ${\rm Con}_{\WCL }(B)=\{\alpha \oplus \theta \oplus \beta \ |\ \alpha \in {\rm Con}(L),\theta \in {\rm Con}(K),\beta \in {\rm Con}_{\WCL 1}(M)\}\cup \{\nabla _A\}\cong ({\rm Con}(L)\times {\rm Con}(K)\times {\rm Con}_{\WCL 1}(M))\oplus {\cal C}_2$.
\item ${\rm Con}_{\WDCL }(B)=\{\alpha \oplus \theta \oplus \beta \ |\ \alpha \in {\rm Con}_{\WDCL 0}(L),\theta \in {\rm Con}(K),\beta \in {\rm Con}(M)\}\cup \{\nabla _A\}\cong ({\rm Con}_{\WDCL 0}(L)\times {\rm Con}(K)\times {\rm Con}(M))\oplus {\cal C}_2$.\end{itemize}\end{corollary}

Note, from the observations in Sections \ref{thealg} and \ref{ordsum} that, with the notations in Section \ref{ordsum}, if $\cdot ^{\Delta \mbox{\textcircled{M}}}$ is trivial, then ${\rm Con}_{\WCL 1}(M)={\rm Con}_1(M)={\rm Con}_{\WCL }(M)\setminus \{M^2\}$ and, similarly, if $\cdot ^{\nabla \mbox{\textcircled{L}}}$ is trivial, then ${\rm Con}_{\WDCL 0}(L)={\rm Con}_0(L)={\rm Con}_{\WDCL }(L)\setminus \{L^2\}$.

\begin{corollary}\begin{enumerate}
\item\label{cgordsumwdl1} ${\rm Con}_{\WDL }(A)=\{\alpha \oplus \beta \ |\ \alpha \in {\rm Con}_{\WCL 0}(L),\beta \in {\rm Con}_{\WDCL 1}(M)\}\cup \{A^2\}\cong ({\rm Con}_{\WCL 0}(L)\times {\rm Con}_{\WDCL 1}(M))\oplus {\cal C}_2$;
\item\label{cgordsumwdl2} ${\rm Con}_{\WDL }(B)=\{\alpha \oplus \theta \oplus \beta \ |\ \alpha \in {\rm Con}_{\WCL 0}(K),\theta \in {\rm Con}(L),\beta \in {\rm Con}_{\WDCL 1}(M)\}\cup \{B^2\}\cong ({\rm Con}_{\WCL 0}(L)\times {\rm Con}(L)\times {\rm Con}_{\WDCL 1}(M))\oplus {\cal C}_2$.\end{enumerate}\label{cgordsumwdl}\end{corollary}

\begin{proof} (\ref{cgordsumwdl1}) ${\rm Con}_{\WDL }(A)={\rm Con}_{\WCL }(A)\cap {\rm Con}_{\WDCL }(A)$, so the statement follows by Proposition \ref{cgordsum}.

\noindent (\ref{cgordsumwdl2}) By (\ref{cgordsumwdl1}) and Proposition \ref{cgordsum} applied for determining ${\rm Con}_{\WCL }(K\oplus L)$, then ${\rm Con}_{\WDCL }(L\oplus M)$.\end{proof}

\begin{corollary}\begin{itemize}
\item $A=L\oplus M$ is subdirectly irreducible in $\WCL $ iff $L$ is subdirectly irreducible as a bounded lattice and ${\rm Con}_{\WCL 1}(M)=\{\mbox{\bf =}_M\}$.
\item $A=L\oplus M$ is subdirectly irreducible in $\WDCL $ iff ${\rm Con}_{\WDCL 0}(L)=\{\mbox{\bf =}_L\}$ and $M$ is subdirectly irreducible as a bounded lattice.
\item $A=L\oplus M$ is subdirectly irreducible in $\WDL $ iff at least one of the lattices ${\rm Con}_{\WDCL 0}(L)$ and ${\rm Con}_{\WCL 1}(M)$ is trivial and the other one has the bottom element strictly meet--irreducible.
\item $B=L\oplus K\oplus M$ is neither subdirectly irreducible in $\WCL $, nor in $\WDCL $.
\item $B=L\oplus K\oplus M$ is subdirectly irreducible in $\WDL $ iff $K$ is subdirectly irreducible as a bounded lattice, ${\rm Con}_{\WDCL 0}(L)=\{\mbox{\bf =}_L\}$ and ${\rm Con}_{\WCL 1}(M)=\{\mbox{\bf =}_M\}$.\end{itemize}\end{corollary}

\section{Weak Dicomplementations on Atomic or Coatomic Lattices}
\label{coatomic}

Since $\WCL \vDash x\vee x^{\Delta }\approx 1$ and $\WDCL \vDash x\wedge x^{\nabla }\approx 0$, it follows that bounded lattices with the $1$ join--irreducible can only be endowed with the trivial weak complementation, while bounded lattices with the $0$ meet--irreducible can only be endowed with the trivial dual weak complementation, hence bounded lattices with the $0$ meet--irreducible and the $1$ join--irreducible can only be endowed with the trivial weak dicomplementation.

In particular, any bounded chain can only be endowed with the trivial weak dicomplementation.

Since weak complementations and dual weak complementations are order--reversing:\begin{itemize}
\item a weak complementation $\cdot ^{\Delta }$ on a coatomic bounded lattice $L$ is nontrivial iff $a^{\Delta }<1$ for some $a\in {\rm CoAt}(L)$, which implies that $a^{\Delta }\leq b$ for some $b\in {\rm CoAt}(L)\setminus \{a\}$ since $L$ is coatomic and we must have $a\vee a^{\Delta }=1$;
\item dually, a dual weak complementation $\cdot ^{\nabla }$ on an atomic bounded lattice $L$ is nontrivial iff $a^{\nabla }>0$ for some $a\in {\rm At}(L)$, which implies that $a^{\nabla }\geq b$ for some $b\in {\rm At}(L)\setminus \{a\}$ since $L$ is atomic and we must have $a\wedge a^{\nabla }=0$.\end{itemize}

If $L$ is a bounded lattice and $a,b\in L\setminus \{0,1\}$ such that $a\neq b$, let us denote by $\cdot ^{\Delta a,b}:L\rightarrow L$ the operation defined by: $1^{\Delta a,b}=0$, $a^{\Delta a,b}=b$, $b^{\Delta a,b}=a$ and $x^{\Delta a,b}=1$ for all $x\in L\setminus \{1,a,b\}$. Of course, since $a,b\in L\setminus \{0,1\}$:\begin{itemize}
\item whenever $\cdot ^{\Delta a,b}$ is a weak complementation on $L$, it is a nontrivial weak complementation;
\item and, since weak complementations are order--reversing, whenever $\cdot ^{\Delta a,b}$ is a weak complementation on $L$, $a$ and $b$ are coatoms of $L$.\end{itemize}

With this notation, we have:

\begin{proposition}\begin{enumerate}
\item\label{2coat1} If $L$ is a coatomic bounded lattice with exactly two coatoms $a$ and $b$, then $\cdot ^{\Delta a,b}$ is a representable weak complementation on $L$, thus $L$ has nontrivial representable weak complementations.
\item\label{2coat2} If $L$ is a bounded distributive lattice $L$ with at least two coatoms and $a$ and $b$ are distinct coatoms of $L$, then $\cdot ^{\Delta a,b}$ is a weak complementation on $L$, thus $L$ has nontrivial weak complementations.\end{enumerate}

Dually for atoms and dual weak complementations.
\label{2coat}\end{proposition}

\begin{proof} Let $L$ be a bounded lattice.

\noindent (\ref{2coat1}) If $L$ is coatomic and ${\rm CoAt}(L)=\{a,b\}$ with $a\neq b$, then it is routine to prove that $\cdot ^{\Delta a,b}$ is a weak complementation on $L$. Clearly, $(a]_L\cup (b]_L=L\setminus \{1\}$ is a join--dense subset of $L$ and, with the notation in Section \ref{thealg}, the weak complementation associated to this join--dense subset is $\cdot ^{\Delta ((a]_L\cup (b]_L)}=\cdot ^{\Delta a,b}$.

\noindent (\ref{2coat2}) Now assume that $L$ is distributive and ${\rm CoAt}(L)\supseteq \{a,b\}$ with $a\neq b$. Then the operation $\cdot ^{\Delta a,b}$ reverses the lattice order of $L$ and satisfies $x^{\Delta a,b\Delta a,b}\leq x$ and $x\vee x^{\Delta a,b}=1$ for all $x\in L$, therefore it is a weak complementation on $L$ since $L$ is distributive.\end{proof}

\begin{proposition} If $L$ is a bounded lattice, $\theta $ is a lattice congruence of $L$ and $a,b\in L\setminus \{0,1\}$ with $a\neq b$ are such that $\cdot ^{\Delta a,b}$ is a weak complementation on $L$, then $\theta \in {\rm Con}_{\WCL }(L,\cdot ^{\Delta a,b})$ iff one of the following is satisfied:\begin{itemize}
\item $\theta \in {\rm Con}_{ab1}(L)\cup \{L^2\}$;
\item $a/\theta =\{a,1\}$ and $(0,b)\in \theta $;
\item $b/\theta =\{b,1\}$ and $(0,a)\in \theta $.\end{itemize}\end{proposition}

\begin{proof} Routine.\end{proof}

Clearly, any direct product of at least two nontrivial bounded lattices can be endowed with the product weak dicomplementation, which is nontrivial. Note that, for instance, the weak complementation on ${\cal C}_2\times {\cal C}_3$ equal to the product of the trivial weak complementations of the chains ${\cal C}_2$ and ${\cal C}_3$ is neither trivial, nor equal to $\cdot ^{\Delta _a,b}$, where $a$ and $b$ are the two coatoms of ${\cal C}_2\times {\cal C}_3$, which proves that $\cdot ^{\Delta a,b}$ is not necessarily the unique nontrivial weak complementation on a coatomic bounded lattice $L$ with ${\rm CoAt}(L)=\{a,b\}$, even when $L$ is finite and distributive.

Of course, in terms of the congruence lattices, if $\V $ is any of the varieties $\WCL $, $\WDCL $ and $\WDL $ and $L$, $M$ are members of $\V $, then, since $\V $ is congruence--distributive and thus $L\times M$ has no skew congruences, we have ${\rm Con}_{\V }(L\times M)=\{\alpha \times \beta \ |\ \alpha \in {\rm Con}_{\V }(L),\beta \in {\rm Con}_{\V }(M)\}\cong {\rm Con}_{\V }(L)\times {\rm Con}_{\V }(M)$, so, if $L$ and $M$ are endowed with the trivial weak dicomplementations, then: ${\rm Con}_{\WCL }(L\times M)\cong ({\rm Con}_1(L)\oplus {\cal C}_2)\times ({\rm Con}_1(M)\oplus {\cal C}_2)$, ${\rm Con}_{\WDCL }(L\times M)\cong ({\rm Con}_0(L)\oplus {\cal C}_2)\times ({\rm Con}_0(M)\oplus {\cal C}_2)$ and ${\rm Con}_{\WDL }(L\times M)\cong ({\rm Con}_{01}(L)\oplus {\cal C}_2)\times ({\rm Con}_{01}(M)\oplus {\cal C}_2)$.

Let us also notice here that, if $L$ is endowed with the trivial weak dicomplementation, then ${\rm Con}_{\V }(L)$ has a single coatom, thus it is directly irreducible. Consequently, if ${\rm Con}_{\WCL }(L)$, ${\rm Con}_{\WDCL }(L)$, respectively ${\rm Con}_{\WDL }(L)$ is directly reducible, then the weak complementation, respectively the dual weak complementation, respectively the weak dicomplementation of $L$ is nontrivial.

For any nontrivial bounded lattices $L$ and $M$:\begin{itemize}
\item ${\rm At}(L\times M)=\{(a,0^M),(0^L,b)\ |\ a\in {\rm At}(L),b\in {\rm At}(M)\}$ and dually for coatoms, hence $|{\rm At}(L\times M)|=|{\rm At}(L)|+|{\rm At}(M)|\geq 2$ and $|{\rm CoAt}(L\times M)|=|{\rm CoAt}(L)|+|{\rm CoAt}(M)|\geq 2$;
\item if $|L|>2$ and $|M|>2$, then ${\rm At}(L\boxplus M)={\rm At}(L)\cup {\rm At}(M)={\rm At}(L)\dotcup {\rm At}(M)$ and ${\rm CoAt}(L\boxplus M)={\rm CoAt}(L)\cup {\rm CoAt}(M)={\rm CoAt}(L)\dotcup {\rm CoAt}(M)$, hence $|{\rm At}(L\boxplus M)|=|{\rm At}(L)|+|{\rm At}(M)|$ and $|{\rm CoAt}(L\boxplus M)|=|{\rm CoAt}(L)|+|{\rm CoAt}(M)|$.\end{itemize}

\begin{example} Clearly, there exist coatomic bounded lattices with nontrivial weak complementations having any number of coatoms greater than $2$, as well as atomic bounded lattices with nontrivial dual weak complementations having any number of atoms greater than $2$, Boolean algebras being the simplest example. The remarks above on direct products provide us with the possibility to construct atomic and coatomic bounded lattices with nontrivial weak dicomplementations having any (equal or distinct) numbers of atoms and coatoms greater than $2$.

On the other hand, there exist bounded lattices with no nontrivial weak complementations having any number of coatoms other than $2$, as well as bounded lattices with no nontrivial dual weak complementations having any number of atoms other than $2$, and bounded lattices with no nontrivial weak dicomplementations having any (equal or distinct) numbers of atoms and coatoms other than $2$, as shown by the remarks at the beginning of Section \ref{ordsum}, according to which, for any bounded lattice $K$ and any cardinal numbers $\kappa ,\lambda \notin \{0,2\}$, since the modular lattice ${\cal M}_{\lambda }$ of length $3$ with $\lambda $ atoms (which are also its coatoms) can obviously only be endowed with the trivial weak dicomplementation (see also Corollary \ref{triplehsum} below), it follows that the bounded lattice $K\oplus {\cal M}_{\lambda }$ has no nontrivial weak complementation, ${\cal M}_{\kappa }\oplus K$ has no nontrivial dual weak complementation, while ${\cal M}_{\kappa }\oplus K\oplus {\cal M}_{\lambda }$ has no nontrivial weak dicomplementation.

See also Corollary \ref{triplehsum} below.\end{example}

\begin{remark} For any coatomic bounded lattice $L$, any $J\subseteq L$ and any $M\subseteq L\setminus \{1\}$ such that ${\rm CoAt}(L)\subseteq M$, condition $\neg sg\Delta (L,J,M)$ is clearly equivalent to $\neg sg\Delta (L,J,{\rm CoAt}(L))$.\label{nontrivcoat}\end{remark}

\begin{example} Remark \ref{nontrivcoat} provides us with an easy construction one can apply to coatomic complete algebraic and dually algebraic lattices $L$ having at least three distinct coatoms in order to transform them into bounded lattices without nontrivial weak complementations having the same number of coatoms: for at least three distinct coatoms $a,b,c$ of $L$, choose elements $p,q,r$ of $L$ such that $p<a,q<b,r<c$, and replace each of the intervals $[p,a]_L,[q,b]_L,[r,c]_L$ with its horizontal sum with a complete algebraic and dually algebraic lattice having at least one strictly join irreducible other than its top element. The resulting bounded lattice $M$ will have the same coatoms as $L$, it will be complete, algebraic and dually algebraic, and it will clearly fail condition $\neg sg\Delta (M,{\rm Sji}(M),{\rm CoAt}(M)={\rm CoAt}(L))$.

In particular, the construction above applied to a finite lattice with at least three coatoms, considering horizontal sums with finite lattices with join--irreducibles other than their lattice bounds, in particular with finite chains of lengths at least three, produces finite lattices with the same number of coatoms and without nontrivial weak complementations.

Here is the previous construction applied to ${\cal C}_2\oplus {\cal C}_2^3$, which, by the remarks in Section \ref{ordsum}, has nontrivial weak complementations since ${\cal C}_2^3$ does, with the intervals given by the filters generated by each of its coatoms, turned into their horizontal sums with the three--element chain; the resulting lattice $L$ has no nontrivial weak complementation, since it clearly fails condition $\neg sg\Delta (L,{\rm Ji}(L),{\rm CoAt}(L))$:

\begin{center}\begin{picture}(40,85)(0,0)
\put(-40,70){$L$:}
\put(20,0){\circle*{3}}
\put(20,20){\circle*{3}}
\put(20,0){\line(0,1){40}}
\put(20,40){\circle*{3}}
\put(0,40){\circle*{3}}
\put(-20,40){\circle*{3}}
\put(40,40){\circle*{3}}
\put(60,40){\circle*{3}}
\put(20,40){\circle*{3}}
\put(20,60){\circle*{3}}
\put(0,60){\circle*{3}}
\put(40,60){\circle*{3}}
\put(20,80){\circle*{3}}
\put(20,80){\line(0,-1){20}}
\put(20,20){\line(1,1){20}}
\put(20,20){\line(-1,1){20}}
\put(20,40){\line(1,1){20}}
\put(20,40){\line(-1,1){20}}
\put(20,80){\line(1,-1){40}}
\put(20,80){\line(-1,-1){40}}
\put(20,60){\line(1,-1){20}}
\put(20,60){\line(-1,-1){20}}
\put(0,40){\line(0,1){20}}
\put(40,40){\line(0,1){20}}
\put(20,0){\line(1,1){40}}
\put(20,0){\line(-1,1){40}}
\put(20,0){\line(2,3){20}}
\put(40,30){\circle*{3}}
\put(20,60){\line(2,-3){20}}
\put(18,-9){$0$}
\put(18,83){$1$}
\put(23,17){$a$}
\put(23,37){$v$}
\put(43,38){$w$}
\put(-8,37){$u$}
\put(23,59){$c$}
\put(-7,58){$b$}
\put(43,58){$d$}
\put(62,37){$k$}
\put(42,29){$h$}
\put(-26,38){$j$}
\end{picture}\end{center}\end{example}

\begin{example} By {\rm \cite{Kw04}}, the only weak complementations on the direct product of chains ${\cal C}_2\times {\cal C}_3$, with the elements denoted as in the leftmost Hasse diagram below, are:\begin{itemize}
\item the trivial weak complementation $\cdot ^{\Delta {\cal C}_2\times {\cal C}_3}$, w.r.t. which ${\rm Con}_{\WCL }({\cal C}_2\times {\cal C}_3,\cdot ^{\Delta {\cal C}_2\times {\cal C}_3})={\rm Con}_1({\cal C}_2\times {\cal C}_3)\cup \{({\cal C}_2\times {\cal C}_3)^2\}=\{\mbox{\bf =}_{{\cal C}_2\times {\cal C}_3},\gamma ,({\cal C}_2\times {\cal C}_3)^2\}\cong {\cal C}_3$, where $\gamma =eq(\{0,v\},\{u,a\},\{b\},\{1\})$, thus ${\rm Con}_{\WCL 1}({\cal C}_2\times {\cal C}_3,\cdot ^{\Delta {\cal C}_2\times {\cal C}_3})=\{\mbox{\bf =}_{{\cal C}_2\times {\cal C}_3},\gamma \}\cong {\cal C}_2$;
\item the direct product $\cdot ^{\Delta {\cal C}_2\times \Delta {\cal C}_3}$ of the trivial weak complementations $\cdot ^{\Delta {\cal C}_2}$ and $\cdot ^{\Delta {\cal C}_3}$ on the chains ${\cal C}_2$ and ${\cal C}_3$, respectively, defined by $1^{\Delta {\cal C}_2\times \Delta {\cal C}_3}=0$, $a^{\Delta {\cal C}_2\times \Delta {\cal C}_3}=u^{\Delta {\cal C}_2\times \Delta {\cal C}_3}=b$, $b^{\Delta {\cal C}_2\times \Delta {\cal C}_3}=u$ and $v^{\Delta {\cal C}_2\times \Delta {\cal C}_3}=0^{\Delta {\cal C}_2\times \Delta {\cal C}_3}=1$, w.r.t. which ${\rm Con}_{\WCL }({\cal C}_2\times {\cal C}_3,\cdot ^{\Delta {\cal C}_2\times \Delta {\cal C}_3})\cong {\rm Con}_{\WCL }({\cal C}_2,\cdot ^{\Delta {\cal C}_2})\times {\rm Con}_{\WCL }({\cal C}_3,\cdot ^{\Delta {\cal C}_3})=({\rm Con}_1({\cal C}_2)\linebreak \cup \{{\cal C}_2^2\})\times ({\rm Con}_1({\cal C}_3)\cup \{{\cal C}_3^2\})\cong ({\cal C}_1\oplus {\cal C}_2)\times ({\cal C}_2\oplus {\cal C}_2)\cong {\cal C}_2\times {\cal C}_3$, thus ${\rm Con}_{\WCL 1}({\cal C}_2\times {\cal C}_3,\cdot ^{\Delta {\cal C}_2\times \Delta {\cal C}_3})\cong {\rm Con}_1({\cal C}_2)\times {\rm Con}_1({\cal C}_3)\cong {\cal C}_1\times {\cal C}_2\cong {\cal C}_2$;
\item with the notation above, $\cdot ^{\Delta a,b}$;
\item the weak complementation that we denote by $\cdot ^{\Delta a,b,b}$, defined by $0^{\Delta a,b,b}=v^{\Delta a,b,b}=1$, $u^{\Delta a,b,b}=a^{\Delta a,b,b}=b$, $b^{\Delta a,b,b}=a$ and $1^{\Delta a,b,b}=0$.\end{itemize}

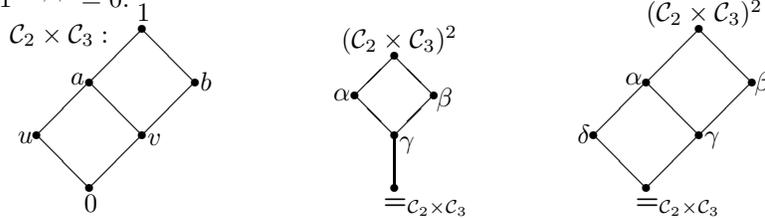
\begin{center}\begin{tabular}{ccc}\begin{picture}(60,60)(0,0)
\put(20,0){\circle*{3}}
\put(20,40){\circle*{3}}
\put(0,20){\circle*{3}}
\put(40,20){\circle*{3}}
\put(60,40){\circle*{3}}
\put(40,60){\circle*{3}}
\put(20,0){\line(-1,1){20}}
\put(20,0){\line(1,1){40}}
\put(0,20){\line(1,1){40}}
\put(40,20){\line(-1,1){20}}
\put(60,40){\line(-1,1){20}}
\put(18,-9){$0$}
\put(38,63){$1$}
\put(13,39){$a$}
\put(62,37){$b$}
\put(-7,17){$u$}
\put(42,16){$v$}
\put(-10,55){${\cal C}_2\times {\cal C}_3:$}
\end{picture}
&\hspace*{40pt}
\begin{picture}(40,60)(0,0)
\put(20,0){\circle*{3}}
\put(20,20){\circle*{3}}
\put(5,35){\circle*{3}}
\put(35,35){\circle*{3}}
\put(20,50){\circle*{3}}
\put(-3,33){$\alpha $}
\put(36,31){$\beta $}
\put(22,15){$\gamma $}
\put(20,0){\line(0,1){20}}
\put(20,20){\line(1,1){15}}
\put(20,20){\line(-1,1){15}}
\put(20,50){\line(1,-1){15}}
\put(20,50){\line(-1,-1){15}}
\put(16,-8){$\mbox{\bf =}_{{\cal C}_2\times {\cal C}_3}$}
\put(0,53){$({\cal C}_2\times {\cal C}_3)^2$}
\end{picture}
&\hspace*{40pt}
\begin{picture}(60,60)(0,0)
\put(20,0){\circle*{3}}
\put(20,40){\circle*{3}}
\put(0,20){\circle*{3}}
\put(40,20){\circle*{3}}
\put(60,40){\circle*{3}}
\put(40,60){\circle*{3}}
\put(20,0){\line(-1,1){20}}
\put(20,0){\line(1,1){40}}
\put(0,20){\line(1,1){40}}
\put(40,20){\line(-1,1){20}}
\put(60,40){\line(-1,1){20}}
\put(16,-8){$\mbox{\bf =}_{{\cal C}_2\times {\cal C}_3}$}
\put(20,63){$({\cal C}_2\times {\cal C}_3)^2$}
\put(12,39){$\alpha $}
\put(61,37){$\beta $}
\put(-6,17){$\delta $}
\put(42,17){$\gamma $}
\end{picture}\end{tabular}\end{center}\vspace*{3pt}

W.r.t. the latter two weak complementations:\begin{itemize}
\item ${\rm Con}_{\WCL }({\cal C}_2\times {\cal C}_3,\cdot ^{\Delta a,b})={\rm Con}_{\WCL }({\cal C}_2\times {\cal C}_3,\cdot ^{\Delta a,b,b})=\{\mbox{\bf =}_{{\cal C}_2\times {\cal C}_3},\alpha ,\beta ,\gamma ,({\cal C}_2\times {\cal C}_3)^2\}\cong {\cal C}_2\oplus {\cal C}_2^2$, as in the rightmost diagram above, where $\alpha =eq(\{0,u,v,a\},\{b,1\})$, $\beta =eq(\{0,v,b\},\{u,a,1\})$ and $\gamma =\alpha \cap \beta $ is as above, thus ${\rm Con}_{\WCL 1}({\cal C}_2\times {\cal C}_3,\cdot ^{\Delta a,b})={\rm Con}_{\WCL 1}({\cal C}_2\times {\cal C}_3,\cdot ^{\Delta a,b,b})=\{\mbox{\bf =}_{{\cal C}_2\times {\cal C}_3},\gamma \}\cong {\cal C}_2$.\end{itemize}

Note also that ${\rm Con}_{\WCL }({\cal C}_2\times {\cal C}_3,\cdot ^{\Delta {\cal C}_2\times {\cal C}_3})=\{\mbox{\bf =}_{{\cal C}_2\times {\cal C}_3},\alpha ,\beta ,\gamma ,\delta ,({\cal C}_2\times {\cal C}_3)^2\}$, where $\delta =eq(\{0,u\},\{v,a\},\{b,1\})$.

For any bounded lattice $M$, the weak complementation on $M\oplus ({\cal C}_2\times {\cal C}_3)$ that restricts to $\cdot ^{\Delta {\cal C}_2\times \Delta {\cal C}_3}$, $\cdot ^{\Delta a,b}$, respectively $\cdot ^{\Delta a,b,b}$ on the upper copy of ${\cal C}_2\times {\cal C}_3$ will be denoted by $\cdot ^{\Delta {\cal C}_2\times \Delta {\cal C}_3}$, $\cdot ^{\Delta a,b}$, respectively $\cdot ^{\Delta a,b,b}$, as well; see Section \ref{ordsum}.

By the above and Proposition \ref{cgordsum}, if $M$ is nontrivial, then:\begin{itemize}
\item ${\rm Con}_{\WCL }(M\oplus ({\cal C}_2\times {\cal C}_3),\cdot ^{\Delta M\oplus ({\cal C}_2\times {\cal C}_3)})\cong ({\rm Con}(M)\times {\rm Con}_{\WCL 1}({\cal C}_2\times {\cal C}_3,\cdot ^{\Delta {\cal C}_2\times {\cal C}_3}))\oplus {\cal C}_2=({\rm Con}(M)\times {\rm Con}_1({\cal C}_2\times {\cal C}_3)\oplus {\cal C}_2\cong ({\rm Con}(M)\times {\cal C}_2)\oplus {\cal C}_2$;
\item ${\rm Con}_{\WCL }(M\oplus ({\cal C}_2\times {\cal C}_3),\cdot ^{\Delta {\cal C}_2\times \Delta {\cal C}_3})\cong ({\rm Con}(M)\times {\rm Con}_{\WCL 1}({\cal C}_2\times {\cal C}_3,\cdot ^{\Delta {\cal C}_2\times \Delta {\cal C}_3})\oplus {\cal C}_2\cong ({\rm Con}(M)\times {\rm Con}_{\WCL 1}({\cal C}_2,\linebreak \cdot ^{\Delta {\cal C}_2})\times {\rm Con}_{\WCL 1}({\cal C}_3,\cdot ^{\Delta {\cal C}_3}))\oplus {\cal C}_2=({\rm Con}(M)\times {\rm Con}_1({\cal C}_2)\times {\rm Con}_1({\cal C}_3))\oplus {\cal C}_2\cong ({\rm Con}(M)\times {\cal C}_1\times {\cal C}_2)\oplus {\cal C}_2\cong ({\rm Con}(M)\times {\cal C}_2)\oplus {\cal C}_2$;
\item ${\rm Con}_{\WCL }(M\oplus ({\cal C}_2\times {\cal C}_3),\cdot ^{\Delta a,b})={\rm Con}_{\WCL }(M\oplus ({\cal C}_2\times {\cal C}_3),\cdot ^{\Delta a,b,b})\cong ({\rm Con}(M)\times {\rm Con}_{\WCL 1}({\cal C}_2\times {\cal C}_3,\cdot ^{\Delta a,b}))\oplus {\cal C}_2=({\rm Con}(M)\times {\rm Con}_{\WCL 1}({\cal C}_2\times {\cal C}_3,\cdot ^{\Delta a,b,b}))\oplus {\cal C}_2\cong ({\rm Con}(M)\times {\cal C}_2)\oplus {\cal C}_2$;\end{itemize}

\noindent so, for any weak complementation $\cdot ^{\Delta }$ on $M\oplus ({\cal C}_2\times {\cal C}_3)$, ${\rm Con}_{\WCL }(M\oplus ({\cal C}_2\times {\cal C}_3),\cdot ^{\Delta })\cong ({\rm Con}(M)\times {\cal C}_2)\oplus {\cal C}_2$. For instance, if $k\in \N \setminus \{0,1\}$, then, for any weak complementation $\cdot ^{\Delta }$ on ${\cal C}_k\oplus ({\cal C}_2\times {\cal C}_3)$, ${\rm Con}_{\WCL }({\cal C}_k\oplus ({\cal C}_2\times {\cal C}_3),\cdot ^{\Delta })\cong ({\cal C}_2^{k-1}\times {\cal C}_2)\oplus {\cal C}_2\cong {\cal C}_2^n\oplus {\cal C}_2$, so $|{\rm Con}_{\WCL }({\cal C}_k\oplus ({\cal C}_2\times {\cal C}_3),\cdot ^{\Delta })|=2^k+1=2^{|{\cal C}_k\oplus ({\cal C}_2\times {\cal C}_3)|-5}+1$.

Note, also, that, for any $h,k\in \N ^*$ with $k\geq 2$, ${\cal C}_h\oplus ({\cal C}_2\times {\cal C}_3)\oplus {\cal C}_k$ can only be endowed with the trivial weak complementation, w.r.t. which it has exactly $2^{h-1}\cdot 2^3\cdot 2^{k-2}+1=2^{h
+k}+1=2^{|{\cal C}_h\oplus ({\cal C}_2\times {\cal C}_3)\oplus {\cal C}_k|-4}+1$ congruences.

Dually, ${\cal C}_2\times {\cal C}_3$ has four dual weak complementations:\begin{itemize}
\item the trivial one, w.r.t. which ${\rm Con}_{\WDCL }({\cal C}_2\times {\cal C}_3,\cdot ^{\nabla {\cal C}_2\times {\cal C}_3})=\{\mbox{\bf =}_{{\cal C}_2\times {\cal C}_3},eq(\{0\},\{u\},\{v,b\},\{a,1\}),({\cal C}_2\times {\cal C}_3)^2\}\cong {\cal C}_3$;
\item the direct product $\cdot ^{\nabla {\cal C}_2\times \nabla {\cal C}_3}$ of the trivial dual weak complementations on ${\cal C}_2$ and ${\cal C}_3$, w.r.t. which ${\rm Con}_{\WDCL }({\cal C}_2\times {\cal C}_3,\cdot ^{\nabla {\cal C}_2\times \nabla {\cal C}_3})=\{\mbox{\bf =}_{{\cal C}_2\times {\cal C}_3},eq(\{0,u\},\{v,a,b,1\}),\beta ,eq(\{0\},\{u\},\{v,b\},\{a,1\}),\delta ,({\cal C}_2\times {\cal C}_3)^2\}\cong {\cal C}_2\times {\cal C}_3$;
\item $\cdot ^{\nabla u,v}$ and $\cdot ^{\nabla u,u,v}$, defined by: $0^{\nabla u,v}=1$, $u^{\nabla u,v}=v$, $v^{\nabla u,v}=u$ and $a^{\nabla u,v}=b^{\nabla u,v}=1^{\nabla u,v}=0$, respectively: $0^{\nabla u,u,v}=1$, $a^{\nabla u,u,v}=u^{\nabla u,u,v}=v$, $v^{\nabla u,u,v}=u$ and $b^{\nabla u,u,v}=1^{\nabla u,u,v}=0$, w.r.t. which ${\rm Con}_{\WDCL }({\cal C}_2\times {\cal C}_3,\cdot ^{\nabla u,v})={\rm Con}_{\WDCL }({\cal C}_2\times {\cal C}_3,\cdot ^{\nabla u,u,v})=\{\mbox{\bf =}_{{\cal C}_2\times {\cal C}_3},eq(\{0,u\},\{v,a,b,1\}),\beta ,eq(\{0\},\{u\},\{v,b\},\{a,1\}),({\cal C}_2\times {\cal C}_3)^2\}\cong {\cal C}_2\oplus {\cal C}_2^2$.\end{itemize}

Hence, for each $\cdot ^{\Delta }\in \{\cdot ^{\Delta a,b},\cdot ^{\Delta a,b,b}\}$ and each $\cdot ^{\nabla }\in \{\cdot ^{\nabla u,v},\cdot ^{\nabla u,u,v}\}$:\begin{itemize}
\item ${\rm Con}_{\WDL }({\cal C}_2\times {\cal C}_3,\cdot ^{\Delta {\cal C}_2\times {\cal C}_3},\cdot ^{\nabla {\cal C}_2\times {\cal C}_3})={\rm Con}_{\WDL }({\cal C}_2\times {\cal C}_3,\cdot ^{\Delta {\cal C}_2\times {\cal C}_3},\cdot ^{\nabla })={\rm Con}_{\WDL }({\cal C}_2\times {\cal C}_3,\cdot ^{\Delta },\cdot ^{\nabla {\cal C}_2\times {\cal C}_3})={\rm Con}_{\WDL }({\cal C}_2\times {\cal C}_3,\cdot ^{\Delta {\cal C}_2\times \Delta {\cal C}_3},\cdot ^{\nabla {\cal C}_2\times {\cal C}_3})={\rm Con}_{\WDL }({\cal C}_2\times {\cal C}_3,\cdot ^{\Delta {\cal C}_2\times {\cal C}_3},\cdot ^{\nabla {\cal C}_2\times \nabla {\cal C}_3})=\{\mbox{\bf =}_{{\cal C}_2\times {\cal C}_3},({\cal C}_2\times {\cal C}_3)^2\}\cong {\cal C}_2$;
\item ${\rm Con}_{\WDL }({\cal C}_2\times {\cal C}_3,\cdot ^{\Delta {\cal C}_2\times \Delta {\cal C}_3},\cdot ^{\nabla {\cal C}_2\times \nabla {\cal C}_3})=\{\mbox{\bf =}_{{\cal C}_2\times {\cal C}_3},\beta ,\delta ,({\cal C}_2\times {\cal C}_3)^2\}\cong {\cal C}_2^2$;
\item ${\rm Con}_{\WDL }({\cal C}_2\times {\cal C}_3,\cdot ^{\Delta },\cdot ^{\nabla })={\rm Con}_{\WDL }({\cal C}_2\times {\cal C}_3,\cdot ^{\Delta {\cal C}_2\times \Delta {\cal C}_3},\cdot ^{\nabla })={\rm Con}_{\WDL }({\cal C}_2\times {\cal C}_3,\cdot ^{\Delta },\cdot ^{\nabla {\cal C}_2\times \nabla {\cal C}_3})=\{\mbox{\bf =}_{{\cal C}_2\times {\cal C}_3},\beta ,\linebreak ({\cal C}_2\times {\cal C}_3)^2\}\cong {\cal C}_3$.\end{itemize}

And similarly for ordinal sums of chains with ${\cal C}_2\times {\cal C}_3$ endowed with dual weak complementations or weak dicomplemenations.\label{C2xC3}\end{example}

\section{Weak Dicomplementations on Horizontal Sums}
\label{hsums}

Throughout this section, $L$ and $M$ will be bounded lattices with $|L|>2$ and $|M|>2$, $A=L\boxplus M$ and, unless mentioned otherwise, $(\cdot ^{\Delta \mbox{\textcircled{L}}},\cdot ^{\nabla \mbox{\textcircled{L}}})$, $(\cdot ^{\Delta \mbox{\textcircled{M}}},\cdot ^{\nabla \mbox{\textcircled{M}}})$ and $(\cdot ^{\Delta },\cdot ^{\nabla })$ will be arbitrary weak dicomplementations on $L$, $M$ and $A$, respectively.

\begin{center}\begin{picture}(40,45)(0,0)
\put(20,0){\circle*{3}}
\put(20,40){\circle*{3}}
\put(20,20){\oval(60,40)}
\put(20,20){\oval(30,40)}
\put(-5,17){$L$}
\put(37,17){$M$}
\put(-5,43){$1=1^L=1^M$}
\put(-5,-11){$0=0^L=0^M$}
\put(-80,40){$A=L\boxplus M:$}\end{picture}\end{center}\vspace*{3pt}

Of course, $L$ and $M$ are bounded sublattices of $A$.

Note that, for all $x\in L\setminus \{0\}$ and all $y\in M\setminus \{0\}$, $[x)_A=[x)_L$ and $[y)_A=[y)_M$, while, for all $x\in L\setminus \{1\}$ and all $y\in M\setminus \{1\}$, $(x]_A=(x]_L$ and $(y]_A=(y]_M$.

Note, also, that $1\notin {\rm Ji}(A)$, so ${\rm Ji}(A)=({\rm Ji}(L)\cup {\rm Ji}(M))\setminus \{1\}$, and $0\notin {\rm Mi}(A)$, so ${\rm Mi}(A)=({\rm Mi}(L)\cup {\rm Mi}(M))\setminus \{0\}$.

\begin{lemma}\begin{itemize}
\item For all $x\in L\setminus \{1\}$ and all $y\in M\setminus \{1\}$, we have $x\leq y^{\Delta }$ and $y\leq x^{\Delta }$, in particular $y^{\Delta }\in L$ and $x^{\Delta }\in M$.

\item For all $x\in L\setminus \{0\}$ and all $y\in M\setminus \{0\}$, we have $x\geq y^{\nabla }$ and $y\geq x^{\nabla }$, in particular $y^{\nabla }\in L$ and $x^{\nabla }\in M$.\end{itemize}\label{lhsum}\end{lemma}

\begin{proof} For all $x\in L\setminus \{1\}$ and all $y\in M\setminus \{1\}$, we have: $0\neq x=(x\wedge y)\vee (x\wedge y^{\Delta })=0\vee (x\wedge y^{\Delta })=x\wedge y^{\Delta }$, hence $y^{\Delta }\geq x$. Since $|L|>2$, there exists a $u\in L\setminus \{0,1\}$. Then $y^{\Delta }\in [u)_A=[u)_L\subset L$.

Analogously for $x^{\Delta }$, and dually for the dual weak complementation.\end{proof}

\begin{proposition}\begin{itemize}
\item $L\in \S _{\WCL }(A)$ iff $\cdot ^{\Delta \mbox{\textcircled{L}}}$ is trivial and $\cdot ^{\Delta }\mid _L=\cdot ^{\Delta \mbox{\textcircled{L}}}$. Similarly for $M$.
\item $L\in \S _{\WDCL }(A)$ iff $\cdot ^{\nabla \mbox{\textcircled{L}}}$ is trivial and $\cdot ^{\nabla }\mid _L=\cdot ^{\nabla \mbox{\textcircled{L}}}$. Similarly for $M$.
\item $L\in \S _{\WDL }(A)$ iff $(\cdot ^{\Delta \mbox{\textcircled{L}}},\cdot ^{\nabla \mbox{\textcircled{L}}})$ is trivial, $\cdot ^{\Delta }\mid _L=\cdot ^{\Delta \mbox{\textcircled{L}}}$ and $\cdot ^{\nabla }\mid _L=\cdot ^{\nabla \mbox{\textcircled{L}}}$. Similarly for $M$.\end{itemize}\label{hsumwdls}\end{proposition}

\begin{proof} By Lemma \ref{lhsum}, if $L\in \S _{\WCL }(A)$, then, for any $x\in L\setminus \{0\}$, we have $x^{\Delta \mbox{\textcircled{L}}}=x^{\Delta }\in M\cap L=\{0,1\}$, hence $\cdot ^{\Delta \mbox{\textcircled{L}}}$ is trivial and $\cdot ^{\Delta }\mid _L=\cdot ^{\Delta \mbox{\textcircled{L}}}$. The converse is trivial.

Similarly for $M$, and dually for the dual weak complementation. Hence the statement on the dicomplementation.\end{proof}

\begin{lemma}\begin{itemize}
\item If $1$ is not strictly join--irreducible in $L$, then $\cdot ^{\Delta }\mid _M$ is the trivial weak complementation on $M$.
\item If $1$ is not strictly join--irreducible in $M$, then $\cdot ^{\Delta }\mid _L$ is the trivial weak complementation on $L$.\end{itemize}

Dually for $0$ and the dual weak complementation.\label{hsumnonsji}\end{lemma}

\begin{proof} Assume that $1\notin {\rm Sji}(L)$ and let $y\in M\setminus \{1\}$.

{\bf Case 1:} $1\notin {\rm Ji}(L)$.

By Lemma \ref{lhsum}, if $1$ is join--reducible in $L$, so that $1=a\vee b$ for some $a,b\in L\setminus \{1\}$, then $y^{\Delta }\geq a$ and $y^{\Delta }\geq b$, thus $y^{\Delta }=1$, so that $\cdot ^{\Delta }\mid _M$ is the trivial weak complementation on $M$.

{\bf Case 2:} $1\in {\rm Ji}(L)\setminus {\rm Sji}(L)$ (so that $L$ is not finite).

If $1$ is join--irreducible, but not strictly join--irreducible in $L$, then $1$ has no predecessors in $L$. Assume by absurdum that $y^{\Delta }<1$. Then, since $y^{\Delta }\nprec 1$, it follows that there exists some $z\in L$ such that $y^{\Delta }<z<1$, so $z\in L\setminus \{1\}$ and $y^{\Delta }\ngeq z$, which is a contradiction to Lemma \ref{lhsum}. Hence $\cdot ^{\Delta }\mid _M$ is the trivial weak complementation on $M$.

Similarly for the case when $1\notin {\rm Sji}(M)$.\end{proof}

\begin{remark} Recall the notations at the end of Section \ref{thealg} and note that, for any bounded lattice $K$, the weak complementation on $K$ is $\cdot ^{\Delta K}$ and the dual weak complementation on $K$ is $\cdot ^{\nabla K}$.\end{remark}

With the notations in Sections \ref{thealg} and \ref{coatomic}, we have:

\begin{theorem} $A=L\boxplus M$ has nontrivial weak complementations iff $1$ is strictly join--irreducible in each of the lattices $L$ and $M$, case in which $A$ has only these two weak complementations, both of which are representable: the trivial weak complementation $\cdot ^{\Delta A}$ and the nontrivial weak complementation $\cdot ^{\Delta 1^{-L},1^{-M}}=\cdot ^{\Delta (1^{-L}]_A\cup (1^{-M}]_A}$, which satisfy $\cdot ^{\Delta A}<\cdot ^{\Delta 1^{-L},1^{-M}}$. In the particular case when $A$ is complete and dually algebraic, $\cdot ^{\Delta 1^{-L},1^{-M}}=\cdot ^{\Delta {\rm Sji}(A)}$. \label{mainth}\end{theorem}

\begin{proof} If $1\in {\rm Sji}(L)\cap {\rm Sji}(M)$, then $A=L\boxplus M$ is a coatomic lattice with exactly two coatoms, $1^{-L}\in L$ and $1^{-M}\in M$, thus, by Proposition \ref{2coat}, (\ref{2coat1}), it has the nontrivial weak complementation $\cdot ^{\Delta 1^{-L},1^{-M}}$.

Now assume that $A$ has a nontrivial weak complementation $\cdot ^{\Delta }$ and let $x\in A$ such that $x^{\Delta }\neq 1$. Then $x\neq 0$ and w.l.g. we may assume that $x\in L$. Then, by Lemma \ref{hsumnonsji}, $1\in {\rm Sji}(M)$ and $x^{\Delta }=1^{-M}$, the unique coatom of $M$. Hence, for all $y\in L\setminus \{1\}$, we have $y=(y\wedge x)\vee (y\wedge x^{\Delta })=(y\wedge x)\vee (y\wedge 1^{-M})=(y\wedge x)\vee 0=y\wedge x$, thus $y\leq x$. Therefore $x=\max (L\setminus \{1\})$, so $1\in {\rm Sji}(L)$ and $x=1^{-L}$, the unique coatom of $L$.

If $(1^{-L})^{\Delta }=1^{-M}$ and $(1^{-M})^{\Delta }=1$, then $(1^{-L})^{\Delta \Delta }=1\nleq 1^{-L}$, which contradicts the definition of a weak complementation. We get a similar contradiction if we assume that $(1^{-L})^{\Delta }=1$ and $(1^{-M})^{\Delta }=1$.

Consequently, $\cdot ^{\Delta }=\cdot ^{\Delta 1^{-L},1^{-M}}$.

Either by the property of weak complementations on complete dually algebraic lattices at the end of Section \ref{thealg} or directly from the definition of such a lattice, along with the fact that ${\rm Sji}(A)=({\rm Sji}(L)\cup {\rm Sji}(M))\setminus \{1\}$ and the fact that $A$ is complete and dually algebraic iff both $L$ and $M$ are complete and dually algebraic, we get the last statement in the enunciation.\end{proof}

\begin{corollary} If $K$ is a bounded lattice with $|K|>2$, then $K\boxplus L\boxplus M$ can only be endowed with the trivial weak dicomplementation.\label{triplehsum}\end{corollary}

\begin{proof} By Theorem \ref{mainth} and the fact that $1\notin {\rm Ji}(K\boxplus L)\cup {\rm Ji}(K\boxplus M)\cup {\rm Ji}(L\boxplus M)$ and $0\notin {\rm Mi}(K\boxplus L)\cup {\rm Mi}(K\boxplus M)\cup {\rm Mi}(L\boxplus M)$.\end{proof}

\begin{remark} Recall from \cite{eunoucard} that:\begin{itemize}
\item ${\rm Con}_{01}(A)={\rm Con}_{01}(L\boxplus M)=\{\alpha \boxplus \beta \ |\ \alpha \in {\rm Con}_{01}(L),\beta \in {\rm Con}_{01}(M)\}\cong {\rm Con}_{01}(L)\times {\rm Con}_{01}(M)$ and
\item ${\rm Con}_{01}(A)\cup \{A^2\}\subseteq {\rm Con}(A)\subseteq {\rm Con}_{01}(A)\cup \{eq(L\setminus \{0\},M\setminus \{1\}),eq(L\setminus \{1\},M\setminus \{0\}),A^2\}$, and:

$eq(L\setminus \{0\},M\setminus \{1\})\in {\rm Con}(A)$ iff $0\in {\rm Mi}(L)$ and $1\in {\rm Ji}(M)$,

$eq(L\setminus \{1\},M\setminus \{0\})\in {\rm Con}(A)$ iff $1\in {\rm Ji}(L)$ and $0\in {\rm Mi}(M)$.\end{itemize}

Since $|L|,|M|>2$, we have $eq(L\setminus \{0\},M\setminus \{1\}),eq(L\setminus \{1\},M\setminus \{0\})\notin {\rm Con}_0(A)\cup {\rm Con}_1(A)$, hence ${\rm Con}_0(A)={\rm Con}_1(A)={\rm Con}_{01}(A)$.\label{thelatcg}\end{remark}

\begin{remark} By Proposition \ref{hsumwdls}: $L,M\in \S _{\WCL }(A)$ iff $\cdot ^{\Delta }$, $\cdot ^{\Delta \mbox{\textcircled{L}}}$ and $\cdot ^{\Delta \mbox{\textcircled{M}}}$ are trivial, and similarly for $\WDCL $ and $\WDL $. In this case, $A$ can be considered as the horizontal sums of the algebras $L$ and $M$ from $\WCL $, $\WDCL $ and $\WDL $, respectively, and, by Remarks \ref{cghsumtriv} and \ref{thelatcg}, w.r.t. the trivial weak dicomplementations, ${\rm Con}_{\WDL }(A)={\rm Con}_{\WCL }(A)={\rm Con}_{\WDCL }(A)={\rm Con}_{01}(A)\cup \{A^2\}=\{\alpha \boxplus \beta \ |\ \alpha \in {\rm Con}_{01}(L)={\rm Con}_{\WDL }(L)={\rm Con}_{\WCL 0}(L)={\rm Con}_{\WDCL 1}(L),\beta \in {\rm Con}_{01}(M)={\rm Con}_{\WDL }(M)={\rm Con}_{\WCL 0}(M)={\rm Con}_{\WDCL 1}(M)\}\cup \{A^2\}$, so the horizontal sum cancels congruences in $\WCL $ and $\WDCL $, while keeping congruences in $\WDL $ in place.\end{remark}

\begin{remark} W.r.t. the trivial weak dicomplementation, we have, by Remark \ref{thelatcg}: ${\rm Con}_{\WCL }(A,\cdot ^{\Delta A})=$\linebreak ${\rm Con}_{\WDCL }(A,\cdot ^{\nabla A})={\rm Con}_{\WDL }(A,\cdot ^{\Delta A},\cdot ^{\nabla A})={\rm Con}_{01}(A)\cup \{A^2\}\cong ({\rm Con}_{01}(L)\times {\rm Con}_{01}(M))\oplus {\cal C}_2$.

For example, since ${\cal C}_2^2={\cal C}_3\boxplus {\cal C}_3$ and ${\rm Con}_{01}({\cal C}_2^2)=\{\mbox{\bf =}_{{\cal C}_2^2}\}$, we have: ${\rm Con}_{\WCL }({\cal C}_2^2,\cdot ^{\Delta {\cal C}_2^2})={\rm Con}_{\WDCL }({\cal C}_2^2,\cdot ^{\nabla {\cal C}_2^2})={\rm Con}_{\WDL }({\cal C}_2^2,\cdot ^{\Delta {\cal C}_2^2},\cdot ^{\nabla {\cal C}_2^2})=\{\mbox{\bf =}_{{\cal C}_2^2},({\cal C}_2^2)^2\}$.\label{cghsumtriv}\end{remark}

\begin{proposition} If $A$ has a nontrivial weak complementation, then, if we denote by $\phi =eq(L\setminus \{0\},M\setminus \{1\})$ and $\psi =eq(L\setminus \{1\},M\setminus \{0\})$, we have:\begin{enumerate}
\item\label{cghsumnontriv1} if $A=\{0,1^{-L},1^{-M},1\}\cong {\cal C}_2^2$, or, equivalently, $L=\{0,1^{-L},1\}\cong {\cal C}_3$ and $M=\{0,1^{-M},1\}\cong {\cal C}_3$, then $\cdot ^{\Delta 1^{-L},1^{-M}}$ is the Boolean complementation and:$${\rm Con}_{\WCL }(A,\cdot ^{\Delta 1^{-L},1^{-M}})=\{\mbox{\bf =}_A,\phi ,\psi ,A^2\}\cong {\cal C}_2^2,$$so $(A,\cdot ^{\Delta 1^{-L},1^{-M}})$ is not subdirectly irreducible in $\WCL $;

\item\label{cghsumnontriv2} if $L\cong {\cal C}_3$ and $|M|>3$, then:$${\rm Con}_{\WCL }(A,\cdot ^{\Delta 1^{-L},1^{-M}})=\{\mbox{\bf =}_L\boxplus (\delta \oplus \mbox{\bf =}_{{\cal C}_2})\ |\ \delta \in {\rm Con}_{01}((1^{-M}]_M)\}\cup \{\phi ,A^2\}\cong {\rm Con}_{01}((1^{-M}]_M)\oplus {\cal C}_3,$$in particular $(A,\cdot ^{\Delta 1^{-L},1^{-M}})$ is subdirectly irreducible in $\WCL $ iff $\mbox{\bf =}_{(1^{-M}]_M}\in {\rm Smi}({\rm Con}_{01}((1^{-M}]_M))$ iff $((1^{-M}]_M,\cdot ^{\Delta (1^{-M}]_M},\cdot ^{\nabla (1^{-M}]_M})$ is subdirectly irreducible in $\WDL $;

\item\label{cghsumnontriv4} if $|L|>3$ and $|M|>3$, then:$${\rm Con}_{\WCL }(A,\cdot ^{\Delta 1^{-L},1^{-M}})=\{(\gamma \oplus \mbox{\bf =}_{{\cal C}_2})\boxplus (\delta \oplus \mbox{\bf =}_{{\cal C}_2})\ |\ \gamma \in {\rm Con}_{01}((1^{-L}]_L),\delta \in {\rm Con}_{01}((1^{-M}]_M)\}\cup \{A^2\}$$$$\cong ({\rm Con}_{01}((1^{-L}]_L)\times {\rm Con}_{01}((1^{-M}]_M))\oplus {\cal C}_2,$$in particular $(A,\cdot ^{\Delta 1^{-L},1^{-M}})$ is subdirectly irreducible in $\WCL $ iff ${\rm Con}_{01}((1^{-L}]_L)=\{\mbox{\bf =}_{(1^{-L}]_L}\}$ and $\mbox{\bf =}_{(1^{-M}]_M}\in {\rm Smi}({\rm Con}_{01}((1^{-M}]_M))$ or vice--versa iff ${\rm Con}_{01}((1^{-L}]_L)=\{\mbox{\bf =}_{(1^{-L}]_L}\}$ and $((1^{-M}]_M,\cdot ^{\Delta (1^{-M}]_M},\cdot ^{\nabla (1^{-M}]_M})$ is subdirectly irreducible in $\WDL $ or vice--versa.\end{enumerate}\label{cghsumnontriv}\end{proposition}

\begin{proof} Remark \ref{cghsumtriv} gives us the congruences of the weakly complemented lattice $(A,\cdot ^{\Delta _A})$.

Since $A$ admits other weak complementations except $\cdot ^{\Delta A}$, by Theorem \ref{mainth} and Remark \ref{thelatcg} we have $1\in {\rm Sji}(L)\cap {\rm Sji}(M)$ and thus:\begin{itemize}
\item $\phi \in {\rm Con}(A)$ iff $0\in {\rm Ji}(L)$, while $\psi \in {\rm Con}(A)$ iff $0\in {\rm Ji}(M)$;
\item $L=(1^{-L}]_L\oplus {\cal C}_2$, so ${\rm Con}(L)=\{\gamma \oplus \zeta \ |\ \gamma \in {\rm Con}((1^{-L}]_L),\zeta \in {\rm Con}({\cal C}_2)\}=\{\gamma \oplus \mbox{\bf =}_{{\cal C}_2},\gamma \oplus {\cal C}_2^2\ |\ \gamma \in {\rm Con}((1^{-L}]_L)\}$;
\item $M=(1^{-M}]_M\oplus {\cal C}_2$, so, analogously, ${\rm Con}(M)=\{\delta \oplus \mbox{\bf =}_{{\cal C}_2},\delta \oplus {\cal C}_2^2\ |\ \delta \in {\rm Con}((1^{-M}]_M)\}$.\end{itemize}

Let $\theta \in {\rm Con}(A)\setminus \{A^2\}$, arbitrary, and let us consider the following conditions on the lattice congruence $\theta $:

\begin{flushleft}\begin{tabular}{rl}
$(LM)(\theta )$ & $(\forall \, x\in A\setminus \{1^{-L},1^{-M}\})\, (x\theta 1^{-L}\Rightarrow x^{\Delta 1^{-L},1^{-M}}\theta 1^{-M})$\\ 
$(ML)(\theta )$ & $(\forall \, x\in A\setminus \{1^{-L},1^{-M}\})\, (x\theta 1^{-M}\Rightarrow x^{\Delta 1^{-L},1^{-M}}\theta 1^{-L})$\end{tabular}\end{flushleft}

Since $\theta \neq A^2$, we have $\theta =(\theta \cap L^2)\boxplus (\theta \cap M^2)$, with $\theta \cap L^2\in {\rm Con}(L)\setminus \{L^2\}$ and $\theta \cap M^2\in {\rm Con}(M)\setminus \{M^2\}$ since $L$ and $M$ are sublattices of $A$ and $(0,1)\notin \theta $, so that $(0,1)\notin \theta \cap L^2$ and $(0,1)\notin \theta \cap M^2$. Thus, for any $x\in L\setminus \{0,1\}=L\setminus M$ and any $y\in M\setminus \{0,1\}=M\setminus L$, $(x,y)\notin \alpha \boxplus \beta $, so that $x/\theta =x/(\theta \cap L^2)$ and $y/\theta =y/(\theta \cap M^2)$, in particular, $(1^{-L},1^{-M})\notin \theta $ and hence:$$\theta \in {\rm Con}_{\WCL }(A,\cdot ^{\Delta 1^{-L},1^{-M}})\mbox{ iff it satisfies conditions }(LM)(\theta )\mbox{ and }(ML)(\theta ).$$

Now let $\alpha \in {\rm Con}_{01}(L)$ and $\beta \in {\rm Con}_{01}(M)$, so that $\alpha \boxplus \beta \in {\rm Con}_{01}(A)\subseteq {\rm Con}(A)\setminus \{A^2\}$, $(\alpha \boxplus \beta )\cap L^2=\alpha $ and $(\alpha \boxplus \beta )\cap M^2=\beta $, $x/\alpha \subseteq L\setminus \{0,1\}$ for any $x\in L\setminus \{0,1\}$ and $y/\beta \subseteq M\setminus \{0,1\}$ for any $y\in M\setminus \{0,1\}$. In particular, $1^{-L}/(\alpha \boxplus \beta )=1^{-L}/\alpha \subseteq L\setminus \{0,1\}$ and $1^{-M}/(\alpha \boxplus \beta )=1^{-M}/\beta \subseteq M\setminus \{0,1\}$.

If there exists an $x\in 1^{-L}/(\alpha \boxplus \beta )=1^{-L}/\alpha $, then $x^{\Delta 1^{-L},1^{-M}}/(\alpha \boxplus \beta )=1/(\alpha \boxplus \beta )=\{1\}\neq 1^{-M}/\beta =1^{-M}/(\alpha \boxplus \beta )$, so condition $(LM)(\alpha \boxplus \beta )$ fails. Clearly, if $1^{-L}/\alpha =\{1^{-L}\}$, then condition $(LM)(\alpha \boxplus \beta )$ is satisfied.

Similarly, condition $(ML)(\alpha \boxplus \beta )$ is satisfied iff $1^{-M}/\beta =\{1^{-M}\}$.

Therefore: $\alpha \boxplus \beta \in {\rm Con}_{\WCL }(A,\cdot ^{\Delta 1^{-L},1^{-M}})$ iff both conditions $(LM)(\alpha \boxplus \beta )$ and $(ML)(\alpha \boxplus \beta )$ are satisfied iff $1^{-L}/\alpha =\{1^{-L}\}$ and $1^{-M}/\beta =\{1^{-M}\}$ iff $\alpha =\gamma \oplus \mbox{\bf =}_{{\cal C}_2}$ for some $\gamma \in {\rm Con}_{01}((1^{-L}]_L)$ and $\beta =\delta \oplus \mbox{\bf =}_{{\cal C}_2}$ for some $\delta \in {\rm Con}_{01}((1^{-M}]_M)$, by the above, hence:$${\rm Con}_{\WCL 01}(A,\cdot ^{\Delta 1^{-L},1^{-M}}))=\{(\gamma \oplus \mbox{\bf =}_{{\cal C}_2})\boxplus (\delta \oplus \mbox{\bf =}_{{\cal C}_2})\ |\ \gamma \in {\rm Con}_{01}((1^{-L}]_L),\delta \in {\rm Con}_{01}((1^{-M}]_M)\}$$$$\cong {\rm Con}_{01}((1^{-L}]_L)\times {\rm Con}_{01}((1^{-M}]_M).$$

Whenever $|L|>3$, so that there exists an element $x\in L\setminus \{0,1^{-L},1\}\subset 1^{-L}/\phi =L\setminus \{0\}$, we have $x^{\Delta 1^{-L},1^{-M}}=1\notin 1^{-M}/\phi $, hence condition $(LM)(\phi )$ fails. If $L\cong {\cal C}_3$, so that $L=\{0,1^{-L},1\}$ and thus $1^{-L}/\phi =\{1^{-L},1\}$, then $1^{\Delta 1^{-L},1^{-M}}=0\in 1^{-M}/\phi $, so condition $(LM)(\phi )$ is satisfied. Similarly, condition $(ML)(\psi )$ is satisfied iff $M\cong {\cal C}_3$.

For every $x\in 1^{-M}/\phi \setminus \{1^{-M}\}=M\setminus \{1^{-M},1\}$, we have $x^{\Delta 1^{-L},1^{-M}}=1\in L\setminus \{0\}=1^{-L}/\phi $, thus condition $(ML)(\phi )$ is satisfied. Similarly, condition $(LM)(\psi )$ is satisfied.

Therefore: $\phi \in {\rm Con}_{\WCL }(A,\cdot ^{\Delta 1^{-L},1^{-M}})$ iff conditions $(LM)(\phi )$ and $(ML)(\phi )$ are satisfied iff condition $(LM)(\phi )$ is satisfied iff $L\cong {\cal C}_3$.

Similarly: $\psi \in {\rm Con}_{\WCL }(A,\cdot ^{\Delta 1^{-L},1^{-M}})$ iff conditions $(LM)(\psi )$ and $(ML)(\psi )$ are satisfied iff condition $(ML)(\psi )$ is satisfied iff $M\cong {\cal C}_3$.

Therefore we have the following cases. Note that, when $L\cong {\cal C}_3$, so that $(1^{-L}]_L\cong {\cal C}_2$, we have: ${\rm Con}_{01}(L)=\{\mbox{\bf =}_L\}\cong {\cal C}_1\cong \{\mbox{\bf =}_{(1^{-L}]_L}\}={\rm Con}_{01}((1^{-L}]_L)={\rm Con}_0((1^{-L}]_L)$, and similarly for $M$.

\noindent (\ref{cghsumnontriv1}) If $A=\{0,1^{-L},1^{-M},1\}\cong {\cal C}_2^2$, then $\phi ,\psi \in {\rm Con}_{\WCL }(A,\cdot ^{\Delta 1^{-L},1^{-M}})$ and ${\rm Con}_{01}(A)=\{\mbox{\bf =}_A\}\cong {\cal C}_1$.

\noindent (\ref{cghsumnontriv2}) If $L=\{0,1^{-L},1\}\cong {\cal C}_3$, but $|M|>3$, then $\phi \in {\rm Con}_{\WCL }(A,\cdot ^{\Delta 1^{-L},1^{-M}})$, but $\psi \notin {\rm Con}_{\WCL }(A,\cdot ^{\Delta 1^{-L},1^{-M}})$.

\noindent (\ref{cghsumnontriv4}) If $|L|>3$ and $|M|>3$, then $\phi ,\psi \notin {\rm Con}_{\WCL }(A,\cdot ^{\Delta 1^{-L},1^{-M}})$.

Hence the forms of the congruence lattices in the enunciation, which yield the subdirect irreducibility results upon noticing that, since $|L|>3$ and $|M|>3$, the bounded lattices $(1^{-L}]_L$ and $(1^{-M}]_M$ are nontrivial.\end{proof}

\begin{corollary} $A=L\boxplus M$ has nontrivial dual weak complementations iff $0$ is strictly meet--irreducible in each of the lattices $L$ and $M$, case in which $A$ has exactly two dual weak complementations, both of which are representable: the trivial dual weak complementation $\cdot ^{\nabla A}$ and $\cdot ^{\nabla 0^{+L},0^{+M}}$, with $\cdot ^{\nabla A}<\cdot ^{\nabla 0^{+L},0^{+M}}$, where $\cdot ^{\nabla 0^{+L},0^{+M}}=\cdot ^{\nabla [0^{+L})_A\cup [0^{+M})_A}$, so that $\cdot ^{\nabla 0^{+L},0^{+M}}$ is defined by: $1^{\nabla 0^{+L},0^{+M}}=0$, $(0^{+L})^{\nabla 0^{+L},0^{+M}}=0^{+M}$, $(0^{+M})^{\Delta 1^{-L},1^{-M}}=0^{+L}$ and $x^{\nabla 0^{+L},0^{+M}}=1$ for all $x\in A\setminus \{0,0^{+L},0^{+M}\}$.

In this case, if we also denote by $\phi =eq(L\setminus \{0\},M\setminus \{1\})$ and $\psi =eq(L\setminus \{1\},M\setminus \{0\})$, then we have:\begin{enumerate}
\item\label{cghsumdualnontriv1} if $A=\{0,0^{+L},0^{+M},1\}\cong {\cal C}_2^2$, or, equivalently, $L=\{0,0^{+L},1\}\cong {\cal C}_3$ and $M=\{0,0^{+M},1\}\cong {\cal C}_3$, then $\cdot ^{\nabla 0^{+L},0^{+M}}$ is the Boolean complementation and:$${\rm Con}_{\WCL }(A,\cdot ^{\nabla 0^{+L},0^{+M}})=\{\mbox{\bf =}_A,\phi ,\psi ,A^2\}\cong {\cal C}_2^2,$$in particular $(A,\cdot ^{\nabla 0^{+L},0^{+M}})$ is not subdirectly irreducible in $\WDCL $;

\item\label{cghsumdualnontriv2} if $L\cong {\cal C}_3$ and $|M|>3$, then:$${\rm Con}_{\WCL }(A,\cdot ^{\nabla 0^{+L},0^{+M}})=\{\mbox{\bf =}_L\boxplus (\delta \oplus \mbox{\bf =}_{{\cal C}_2})\ |\ \delta \in {\rm Con}_{01}([0^{+M})_M)\}\cup \{\psi ,A^2\}\cong {\rm Con}_{01}([0^{+M})_M)\oplus {\cal C}_3,$$in particular $(A,\cdot ^{\nabla 0^{+L},0^{+M}})$ is subdirectly irreducible in $\WDCL $ iff $\mbox{\bf =}_{[0^{+M})_M}\in {\rm Smi}({\rm Con}_{01}([0^{+M})_M)$ iff $([0^{+M})_M,\cdot ^{\Delta [0^{+M})_M},\cdot ^{\nabla [0^{+M})_M})$ is subdirectly irreducible in $\WDL $;

\item\label{cghsumdualnontriv4} if $|L|>3$ and $|M|>3$, then:$${\rm Con}_{\WCL }(A,\cdot ^{\nabla 0^{+L},0^{+M}})=\{(\gamma \oplus \mbox{\bf =}_{{\cal C}_2})\boxplus (\delta \oplus \mbox{\bf =}_{{\cal C}_2})\ |\ \gamma \in {\rm Con}_{01}([0^{+L})_L),\delta \in {\rm Con}_{01}([0^{+M})_M)\}\cup \{A^2\}$$$$\cong ({\rm Con}_{01}([0^{+L})_L)\times {\rm Con}_{01}([0^{+M})_M)\oplus {\cal C}_2,$$in particular $(A,\cdot ^{\nabla 0^{+L},0^{+M}})$ is subdirectly irreducible in $\WCL $ iff ${\rm Con}_{01}([0^{+L})_L)=\{\mbox{\bf =}_{[0^{+L})_L}\}$ and $\mbox{\bf =}_{[0^{+M})_M}\in {\rm Smi}({\rm Con}_{01}([0^{+M})_M))$ or vice--versa iff ${\rm Con}_{01}([0^{+L})_L)=\{\mbox{\bf =}_{[0^{+L})_L}\}$ and $([0^{+M})_M,\cdot ^{\Delta [0^{+M})_M},\cdot ^{\nabla [0^{+M})_M})$ is subdirectly irreducible in $\WDL $ or vice--versa.\end{enumerate}

In the particular case when $A$ is complete and algebraic, $\cdot ^{\nabla 0^{+L},0^{+M}}=\cdot ^{\nabla {\rm Smi}(A)}$.\label{cghsumdualnontriv}\end{corollary}

\begin{proof} By duality, from Theorem \ref{mainth} and Proposition \ref{cghsumnontriv}.\end{proof}

\begin{corollary} $A=L\boxplus M$ has nontrivial weak dicomplementations iff at least one of the following conditions holds:\begin{enumerate}
\item\label{cghsumdinontriv1} $1$ is strictly join--irreducible in both $L$ and $M$;
\item\label{cghsumdinontriv2} $0$ is strictly meet--irreducible in both $L$ and $M$.\end{enumerate}

The weak dicomplementations on $A$ are $(\cdot ^{\Delta },\cdot ^{\nabla })$, with $\cdot ^{\Delta }\in \{\cdot ^{\Delta A},\cdot ^{\Delta 1^{-L},1^{-M}}\}$ if condition (\ref{cghsumdinontriv1}) holds and $\cdot ^{\Delta }=\cdot ^{\Delta A}$ otherwise, and $\cdot ^{\nabla }\in \{\cdot ^{\nabla A},\cdot ^{\nabla 0^{+L},0^{+M}}\}$ if condition (\ref{cghsumdinontriv2}) holds and $\cdot ^{\nabla }=\cdot ^{\nabla A}$ otherwise, all of which are representable.

If $A\cong {\cal C}_2^2$, which is equivalent to $L\cong M\cong {\cal C}_3$ and also to $1^{-L}=0^{+L}$ and $1^{-M}=0^{+M}$, then both $\cdot ^{\Delta 1^{-L},1^{-M}}$ and $\cdot ^{\nabla 0^{+L},0^{+M}}$ equal the Boolean complementation, so we have: ${\rm Con}_{\WDL }(A,\cdot ^{\Delta 1^{-L},1^{-M}},\cdot ^{\nabla 0^{+L},0^{+M}})\cong {\cal C}_2^2$, while ${\rm Con}_{\WDL }(A,\cdot ^{\Delta A},\cdot ^{\nabla A})={\rm Con}_{\WDL }(A,\cdot ^{\Delta 1^{-L},1^{-M}},\cdot ^{\nabla A})={\rm Con}_{\WDL }(A,\cdot ^{\Delta A},\cdot ^{\nabla 0^{+L},0^{+M}})=\{\mbox{\bf =}_A,A^2\}\cong {\cal C}_2$.

Otherwise:\begin{itemize}
\item if condition (\ref{cghsumdinontriv1}) holds, then:$${\rm Con}_{\WDL }(A,\cdot ^{\Delta 1^{-L},1^{-M}},\cdot ^{\nabla A})=\{(\gamma \oplus \mbox{\bf =}_{{\cal C}_2})\boxplus (\delta \oplus \mbox{\bf =}_{{\cal C}_2})\ |\ \gamma \in {\rm Con}_{01}((1^{-L}]_L),\delta \in {\rm Con}_{01}((1^{-M}]_M)\}\cup \{A^2\}$$$$\cong ({\rm Con}_{01}((1^{-L}]_L)\times {\rm Con}_{01}((1^{-M}]_M))\oplus {\cal C}_2;$$
\item if condition (\ref{cghsumdinontriv2}) holds, then:$${\rm Con}_{\WDL }(A,\cdot ^{\Delta A},\cdot ^{\nabla 0^{+L},0^{+M}})=\{(\gamma \oplus \mbox{\bf =}_{{\cal C}_2})\boxplus (\delta \oplus \mbox{\bf =}_{{\cal C}_2})\ |\ \gamma \in {\rm Con}_{01}([0^{+L})_L),\delta \in {\rm Con}_{01}([0^{+M})_M)\}\cup \{A^2\}$$$$\cong ({\rm Con}_{01}([0^{+L})_L)\times {\rm Con}_{01}([0^{+M})_M)\oplus {\cal C}_2;$$
\item if both conditions (\ref{cghsumdinontriv1}) and (\ref{cghsumdinontriv2}) hold, then:$${\rm Con}_{\WDL }(A,\cdot ^{\Delta 1^{-L},1^{-M}},\cdot ^{\nabla 0^{+L},0^{+M}})=\{(\mbox{\bf =}_{{\cal C}_2}\oplus \varepsilon \oplus \mbox{\bf =}_{{\cal C}_2})\boxplus (\mbox{\bf =}_{{\cal C}_2}\oplus \xi \oplus \mbox{\bf =}_{{\cal C}_2})\ |\ \varepsilon \in {\rm Con}_{01}([0^{+L},1^{-L}]_L),$$$$\xi \in {\rm Con}_{01}([0^{+M},1^{-M}]_M)\}\cup \{A^2\}\cong ({\rm Con}_{01}([0^{+L},1^{-L}]_L)\times {\rm Con}_{01}([0^{+M},1^{-M}]_M)\oplus {\cal C}_2.$$\end{itemize}\label{cghsumdinontriv}\end{corollary}

\begin{remark} So $A=L\boxplus M$ has at most two weak complementations and at most two dual weak complementations, so either just one, or two, or four weak dicomplementations, in the latter case its weak dicomplementations, being ordered as in the following Hasse diagram:\vspace*{10pt}

\begin{center}\begin{picture}(40,35)(0,0)
\put(20,5){\circle*{3}}
\put(20,5){\line(-1,1){15}}
\put(20,5){\line(1,1){15}}
\put(5,20){\circle*{3}}
\put(35,20){\circle*{3}}
\put(20,35){\circle*{3}}
\put(20,35){\line(-1,-1){15}}
\put(20,35){\line(1,-1){15}}
\put(-28,-5){$(\cdot ^{\Delta 1^{-L},1^{-M}},\cdot ^{\nabla 0^{+L},0^{+M}})$}
\put(-69,17){$(\cdot ^{\Delta A},\cdot ^{\nabla 0^{+L},0^{+M}})$}
\put(37,17){$(\cdot ^{\Delta 1^{-L},1^{-M}},\cdot ^{\nabla 0})$}
\put(0,40){$(\cdot ^{\Delta A},\cdot ^{\nabla A})$}
\end{picture}\end{center}

We have, of course, ${\rm Con}_{\WDL }(A,\cdot ^{\Delta },\cdot ^{\nabla })={\rm Con}_{\WCL }(A,\cdot ^{\Delta })\cap {\rm Con}_{\WDCL }(A,\cdot ^{\nabla })$ for each $(\cdot ^{\Delta },\cdot ^{\nabla })$ as in Corollary \ref{cghsumdinontriv}, and immediate subdirect irreducibility characterizations follow.\end{remark}

\begin{example} Let $L={\cal C}_2\oplus {\cal C}_2^2\oplus {\cal C}_2$ and $M={\cal C}_2^2\oplus {\cal C}_2$, and let the elements of $A=L\boxplus M$ be denoted as in the following Hasse diagram:

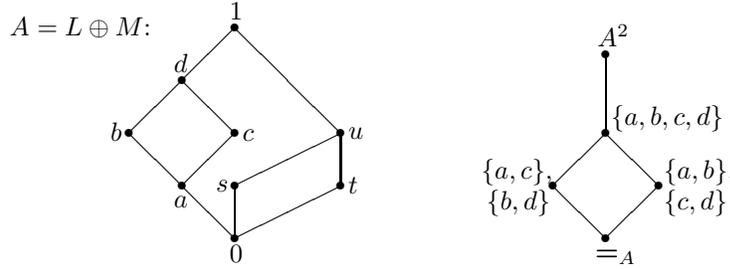
\begin{center}\begin{tabular}{cc}\begin{picture}(40,85)(0,0)
\put(-65,77){$A=L\oplus M$:}
\put(20,0){\circle*{3}}
\put(20,0){\line(-1,1){20}}
\put(0,20){\line(-1,1){20}}
\put(0,20){\line(1,1){20}}
\put(0,60){\line(-1,-1){20}}
\put(0,60){\line(1,-1){20}}
\put(20,80){\line(-1,-1){20}}
\put(20,80){\circle*{3}}
\put(0,20){\circle*{3}}
\put(20,20){\circle*{3}}
\put(-20,40){\circle*{3}}
\put(20,40){\circle*{3}}
\put(0,60){\circle*{3}}
\put(60,40){\circle*{3}}
\put(60,20){\circle*{3}}
\put(20,80){\line(1,-1){40}}
\put(20,0){\line(2,1){40}}
\put(20,20){\line(2,1){40}}
\put(20,0){\line(0,1){20}}
\put(60,40){\line(0,-1){20}}
\put(18,-9){$0$}
\put(18,83){$1$}
\put(-3,12){$a$}
\put(-3,63){$d$}
\put(-27,37){$b$}
\put(23,37){$c$}
\put(13,18){$s$}
\put(63,17){$t$}
\put(63,38){$u$}
\end{picture}
&\hspace{85pt}
\begin{picture}(40,85)(0,0)
\put(20,0){\circle*{3}}
\put(0,20){\circle*{3}}
\put(40,20){\circle*{3}}
\put(20,40){\circle*{3}}
\put(20,70){\circle*{3}}
\put(20,0){\line(1,1){20}}
\put(20,0){\line(-1,1){20}}
\put(20,40){\line(-1,-1){20}}
\put(20,40){\line(1,-1){20}}
\put(20,40){\line(0,1){30}}
\put(17,73){$A^2$}
\put(22,43){$\{a,b,c,d\}$}
\put(42,23){$\{a,b\},$}
\put(42,11){$\{c,d\}$}
\put(-27,23){$\{a,c\},$}
\put(-25,11){$\{b,d\}$}
\put(16,-8){$\mbox{\bf =}_A$}
\end{picture}\end{tabular}\end{center}

Then, by Theorem \ref{mainth} and Corollary \ref{cghsumdualnontriv}, $A$ has the weak complementations $\cdot ^{\Delta A}$ and $\cdot ^{\Delta d,u}$ and only the trivial dual weak complementation $\cdot ^{\nabla A}$.

We have, by Remark \ref{cghsumtriv}: ${\rm Con}_{\WDL }(A,\cdot ^{\Delta A},\cdot ^{\nabla A})={\rm Con}_{\WDCL }(A,\cdot ^{\Delta A})={\rm Con}_{\WDCL }(A,\cdot ^{\nabla A})={\rm Con}_{01}(A)\cup \{A^2\}=\{(\mbox{\bf =}_{{\cal C}_2}\oplus \zeta \oplus \mbox{\bf =}_{{\cal C}_2})\boxplus \mbox{\bf =}_M\ |\ \zeta \in {\rm Con}({\cal C}_2^2)\}\cup \{A^2\}\cong {\rm Con}({\cal C}_2^2)\oplus {\cal C}_2\cong {\cal C}_2^2\oplus {\cal C}_2$, which has the lattice structure represented in the rightmost diagram above, in which the proper nontrivial congruences are indicated by their nonsingleton classes.

By Proposition \ref{cghsumnontriv}, (\ref{cghsumnontriv4}), we have: ${\rm Con}_{\WCL }(A,\cdot ^{\Delta d,u})=\{\mbox{\bf =}_A,A^2\}\cong {\cal C}_2$.

Consequently, the weak dicomplementations on $A$ are $(\cdot ^{\Delta A},\cdot ^{\nabla A})$ and $(\cdot ^{\Delta d,u},\cdot ^{\nabla A})$, and we have:\linebreak ${\rm Con}_{\WCL }(A,\cdot ^{\Delta },\cdot ^{\nabla A})={\rm Con}_{\WCL }(A,\cdot ^{\Delta })$ for every $\cdot ^{\Delta }\in \{\cdot ^{\Delta A},\cdot ^{\Delta d,u}\}$.

Now let us consider the five--element non--modular lattice ${\cal N}_5={\cal C}_3\boxplus {\cal C}_4$, with the elements denoted as in the following Hasse diagram. Then, by Theorem \ref{mainth} and Corollary \ref{cghsumdualnontriv}, as in the case of the four--element Boolean algebra ${\cal C}_2^2={\cal C}_3\boxplus {\cal C}_3$, ${\cal N}_5$ has four weak dicomplementations: $\{(\cdot ^{\Delta },\cdot ^{\nabla })\ |\ \cdot ^{\Delta }\in \{\cdot ^{\Delta {\cal N}_5},\cdot ^{\Delta a,c}\},\cdot ^{\nabla }\in \{\cdot ^{\nabla {\cal N}_5},\cdot ^{\nabla a,b}\}\}$, where $L={\cal C}_3$ and $M={\cal C}_4$.\vspace*{5pt}

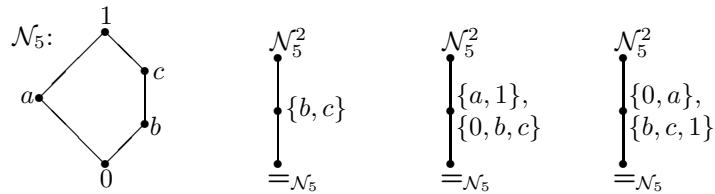
\begin{center}\begin{tabular}{cccc}
\begin{picture}(30,50)(0,0)
\put(-20,45){${\cal N}_5$:}
\put(15,0){\circle*{3}}
\put(30,15){\circle*{3}}
\put(15,0){\line(1,1){15}}
\put(15,0){\line(-1,1){25}}
\put(30,15){\line(0,1){20}}
\put(15,50){\circle*{3}}
\put(30,35){\circle*{3}}
\put(-10,25){\circle*{3}}
\put(15,50){\line(1,-1){15}}
\put(15,50){\line(-1,-1){25}}
\put(13,-9){$0$}
\put(13,53){$1$}
\put(32,11){$b$}
\put(33,32){$c$}
\put(-17,23){$a$}
\end{picture}
&\hspace*{20pt}
\begin{picture}(30,50)(0,0)
\put(15,0){\circle*{3}}
\put(15,20){\circle*{3}}
\put(15,40){\circle*{3}}
\put(15,0){\line(0,1){40}}
\put(11,-8){$\mbox{\bf =}_{{\cal N}_5}$}
\put(12,43){${\cal N}_5^2$}
\put(18,18){$\{b,c\}$}
\end{picture}
&\hspace*{20pt}
\begin{picture}(30,50)(0,0)
\put(15,0){\circle*{3}}
\put(15,20){\circle*{3}}
\put(15,40){\circle*{3}}
\put(15,0){\line(0,1){40}}
\put(11,-8){$\mbox{\bf =}_{{\cal N}_5}$}
\put(12,43){${\cal N}_5^2$}
\put(17,23){$\{a,1\},$}
\put(17,11){$\{0,b,c\}$}
\end{picture}
&\hspace*{20pt}
\begin{picture}(30,50)(0,0)
\put(15,0){\circle*{3}}
\put(15,20){\circle*{3}}
\put(15,40){\circle*{3}}
\put(15,0){\line(0,1){40}}
\put(11,-8){$\mbox{\bf =}_{{\cal N}_5}$}
\put(12,43){${\cal N}_5^2$}
\put(17,23){$\{0,a\},$}
\put(17,11){$\{b,c,1\}$}
\end{picture}\end{tabular}\end{center}\vspace*{3pt}

By Remark \ref{cghsumtriv}, ${\rm Con}_{\WDL }({\cal N}_5,\cdot ^{\Delta {\cal N}_5},\cdot ^{\nabla {\cal N}_5})={\rm Con}_{\WCL }({\cal N}_5,\cdot ^{\Delta {\cal N}_5})={\rm Con}_{\WDCL }({\cal N}_5,\cdot ^{\nabla {\cal N}_5})={\rm Con}_{01}({\cal N}_5)\cup \{{\cal N}_5^2\}=\{\mbox{\bf =}_{{\cal C}_3}\boxplus \mbox{\bf =}_{{\cal C}_4},\mbox{\bf =}_{{\cal C}_3}\boxplus (\mbox{\bf =}_{{\cal C}_2}\oplus {\cal C}_2^2\oplus \mbox{\bf =}_{{\cal C}_2})\}\cup \{{\cal N}_5^2\}=\{\mbox{\bf =}_{{\cal N}_5},\mbox{\bf =}_{{\cal C}_3}\boxplus (\mbox{\bf =}_{{\cal C}_2}\oplus {\cal C}_2^2\oplus \mbox{\bf =}_{{\cal C}_2}),{\cal N}_5^2\}\cong {\cal C}_3$, represented in the second diagram above, since ${\cal C}_2\cong {\rm Con}_{01}({\cal N}_5)=\{\mbox{\bf =}_{{\cal N}_5},\mbox{\bf =}_{{\cal C}_3}\boxplus (\mbox{\bf =}_{{\cal C}_2}\oplus {\cal C}_2^2\oplus \mbox{\bf =}_{{\cal C}_2})\}=\{\mbox{\bf =}_{{\cal N}_5},eq(\{0\},\{a\},\{b,c\},\{1\})\}={\rm Con}_{\WDL 01}({\cal N}_5,\cdot ^{\Delta {\cal N}_5},\cdot ^{\nabla {\cal N}_5})={\rm Con}_{\WCL 01}({\cal N}_5,\cdot ^{\Delta {\cal N}_5})={\rm Con}_{\WDCL 01}({\cal N}_5,\cdot ^{\nabla {\cal N}_5})={\rm Con}_{\WDL 0}({\cal N}_5,\cdot ^{\Delta {\cal N}_5},\cdot ^{\nabla {\cal N}_5})=\linebreak {\rm Con}_{\WCL 0}({\cal N}_5,\cdot ^{\Delta {\cal N}_5})={\rm Con}_{\WDCL 0}({\cal N}_5,\cdot ^{\nabla {\cal N}_5})={\rm Con}_{\WDL 1}({\cal N}_5,\cdot ^{\Delta {\cal N}_5},\cdot ^{\nabla {\cal N}_5})={\rm Con}_{\WCL 1}({\cal N}_5,\cdot ^{\Delta {\cal N}_5})={\rm Con}_{\WDCL 1}({\cal N}_5,\linebreak \cdot ^{\nabla {\cal N}_5})$.

By Proposition \ref{cghsumnontriv}, ${\rm Con}_{\WCL }({\cal N}_5,\cdot ^{\Delta a,c})=\{\mbox{\bf =}_{{\cal C}_3}\boxplus \mbox{\bf =}_{{\cal C}_4},eq(\{0,b,c\},\{a,1\}),{\cal N}_5^2\}=\{\mbox{\bf =}_{{\cal N}_5},eq(\{0,b,c\},\{a,1\}),\linebreak {\cal N}_5^2\}\cong {\cal C}_3$, represented in the third diagram above, so ${\rm Con}_{\WCL 01}({\cal N}_5,\cdot ^{\Delta a,c})={\rm Con}_{\WCL 0}({\cal N}_5,\cdot ^{\Delta a,c})={\rm Con}_{\WCL 1}({\cal N}_5,\linebreak \cdot ^{\Delta a,c})=\{\mbox{\bf =}_{{\cal N}_5}\}\cong {\cal C}_1$.

By Corollary \ref{cghsumdualnontriv}, ${\rm Con}_{\WDCL }({\cal N}_5,\cdot ^{\nabla a,b})=\{\mbox{\bf =}_{{\cal C}_3}\boxplus \mbox{\bf =}_{{\cal C}_4},eq(\{0,a\},\{b,c,1\}),{\cal N}_5^2\}=\{\mbox{\bf =}_{{\cal N}_5},eq(\{0,a\},\{b,c,1\}),\linebreak {\cal N}_5^2\}\cong {\cal C}_3$, represented in the fourth diagram above, so ${\rm Con}_{\WDCL 01}({\cal N}_5,\cdot ^{\nabla a,b})={\rm Con}_{\WDCL 0}({\cal N}_5,\cdot ^{\nabla a,b})=\linebreak {\rm Con}_{\WDCL 1}({\cal N}_5,\cdot ^{\nabla a,b})=\{\mbox{\bf =}_{{\cal N}_5}\}\cong {\cal C}_1$.

Consequently: ${\rm Con}_{\WDL }({\cal N}_5,\cdot ^{\Delta {\cal N}_5},\cdot ^{\nabla a,b})={\rm Con}_{\WDL }({\cal N}_5,\cdot ^{\Delta a,c},\cdot ^{\nabla {\cal N}_5})={\rm Con}_{\WDL }({\cal N}_5,\cdot ^{\Delta a,c},\cdot ^{\nabla a,b})=\{\mbox{\bf =}_{{\cal N}_5},{\cal N}_5^2\}$\linebreak $\cong {\cal C}_2$, and, w.r.t. any of these nontrivial weak dicomplementations $(\cdot ^{\Delta },\cdot ^{\nabla })$, ${\rm Con}_{\WDL 01}({\cal N}_5,\cdot ^{\Delta },\cdot ^{\nabla })={\rm Con}_{\WDL 0}({\cal N}_5,\linebreak \cdot ^{\Delta },\cdot ^{\nabla })={\rm Con}_{\WDL 1}({\cal N}_5,\cdot ^{\Delta },\cdot ^{\nabla })=\{\mbox{\bf =}_{{\cal N}_5}\}\cong {\cal C}_1$.

More generally, if $r,s\in \N $ such that $r\geq 3$ and $s\geq 4$, so that $|{\cal C}_r\boxplus {\cal C}_s|=r+s-2$, then we may notice that ${\rm Con}({\cal C}_r\boxplus {\cal C}_s)\cong {\cal C}_2^{r+s-6}\oplus {\cal C}_2^2$ and ${\rm Con}_{01}({\cal C}_r\boxplus {\cal C}_s)={\rm Con}_0({\cal C}_r\boxplus {\cal C}_s)={\rm Con}_1({\cal C}_r\boxplus {\cal C}_s)\cong {\cal C}_2^{r+s-6}$ (see also \cite{eunoucard,eucfifin}), hence ${\rm Con}_{\WDL }({\cal C}_r\boxplus {\cal C}_s,\cdot ^{\Delta {\cal C}_r\boxplus {\cal C}_s},\cdot ^{\nabla {\cal C}_r\boxplus {\cal C}_s})={\rm Con}_{\WCL }({\cal C}_r\boxplus {\cal C}_s,\cdot ^{\Delta {\cal C}_r\boxplus {\cal C}_s})={\rm Con}_{\WDCL }({\cal C}_r\boxplus {\cal C}_s,\cdot ^{\nabla {\cal C}_r\boxplus {\cal C}_s})={\rm Con}_{01}({\cal C}_r\boxplus {\cal C}_s)\cup \{({\cal C}_r\boxplus {\cal C}_s)^2\}\cong {\cal C}_2^{r+s-6}\oplus {\cal C}_2$.

If ${\rm At}({\cal C}_r\boxplus {\cal C}_s)=\{a,b\}$ and ${\rm CoAt}({\cal C}_r\boxplus {\cal C}_s)=\{c,d\}$, then:

\noindent ${\rm Con}_{\WCL }({\cal C}_r\boxplus {\cal C}_s,\cdot ^{\Delta c,d})\cong {\rm Con}_{\WDCL }({\cal C}_r\boxplus {\cal C}_s,\cdot ^{\nabla a,b})\cong \begin{cases}({\rm Con}_{01}({\cal C}_{r-1})\times {\rm Con}_{01}({\cal C}_{s-1}))\oplus {\cal C}_3\cong {\cal C}_2^{s-4}\oplus {\cal C}_3,& r=3,\\ ({\rm Con}_{01}({\cal C}_{r-1})\times {\rm Con}_{01}({\cal C}_{s-1}))\oplus {\cal C}_2\cong {\cal C}_2^{r+s-8}\oplus {\cal C}_2,& r\geq 4;\end{cases}$

\noindent ${\rm Con}_{\WDL }({\cal C}_r\boxplus {\cal C}_s,\cdot ^{\Delta c,d},\cdot ^{\nabla {\cal C}_r\boxplus {\cal C}_s})\cong {\rm Con}_{\WDL }({\cal C}_r\boxplus {\cal C}_s,\cdot ^{\Delta {\cal C}_r\boxplus {\cal C}_s},\cdot ^{\nabla a,b})\cong ({\rm Con}_{01}({\cal C}_{r-1})\times {\rm Con}_{01}({\cal C}_{s-1}))\oplus {\cal C}_2\cong \begin{cases}{\cal C}_2^{s-4}\oplus {\cal C}_2,& r=3,\\ {\cal C}_2^{r+s-8}\oplus {\cal C}_2,& r\geq 4;\end{cases}$

\noindent ${\rm Con}_{\WDL }({\cal C}_r\boxplus {\cal C}_s,\cdot ^{\Delta c,d},\cdot ^{\nabla a,b})\cong ({\rm Con}_{01}({\cal C}_{r-2})\times {\rm Con}_{01}({\cal C}_{s-2}))\oplus {\cal C}_2\cong \begin{cases}{\cal C}_2,& r\leq 4,s\leq 4;\\ {\cal C}_2^{r-5}\oplus {\cal C}_2,& r\geq 5,s\leq 4;\\ {\cal C}_2^{s-5}\oplus {\cal C}_2,& r\leq 4,s\geq 5;\\ {\cal C}_2^{r+s-10}\oplus {\cal C}_2,& r\geq 5,s\geq 5.\end{cases}$

As for the congruences with the classes of $0$ and/or $1$ singletons, we obtain:

\noindent ${\rm Con}_{\WCL 01}({\cal C}_r\boxplus {\cal C}_s,\cdot ^{\Delta c,d})={\rm Con}_{\WCL 0}({\cal C}_r\boxplus {\cal C}_s,\cdot ^{\Delta c,d})={\rm Con}_{\WDCL 1}({\cal C}_r\boxplus {\cal C}_s,\cdot ^{\Delta c,d})\cong {\rm Con}_{\WDCL 01}({\cal C}_r\boxplus {\cal C}_s,\cdot ^{\nabla a,b})={\rm Con}_{\WDCL 0}({\cal C}_r\boxplus {\cal C}_s,\cdot ^{\nabla a,b})={\rm Con}_{\WDCL 1}({\cal C}_r\boxplus {\cal C}_s,\cdot ^{\nabla a,b})\cong {\rm Con}_{\WDL 01}({\cal C}_r\boxplus {\cal C}_s,\cdot ^{\Delta c,d},\cdot ^{\nabla {\cal C}_r\boxplus {\cal C}_s})={\rm Con}_{\WDL 0}({\cal C}_r\boxplus {\cal C}_s,\cdot ^{\Delta c,d},\cdot ^{\nabla {\cal C}_r\boxplus {\cal C}_s})={\rm Con}_{\WDL 1}({\cal C}_r\boxplus {\cal C}_s,\cdot ^{\Delta c,d},\cdot ^{\nabla {\cal C}_r\boxplus {\cal C}_s})\cong {\rm Con}_{\WDL 01}({\cal C}_r\boxplus {\cal C}_s,\cdot ^{\Delta {\cal C}_r\boxplus {\cal C}_s},\cdot ^{\nabla a,b})={\rm Con}_{\WDL 0}({\cal C}_r\boxplus {\cal C}_s,\cdot ^{\Delta {\cal C}_r\boxplus {\cal C}_s},\cdot ^{\nabla a,b})={\rm Con}_{\WDL 1}({\cal C}_r\boxplus {\cal C}_s,\cdot ^{\Delta {\cal C}_r\boxplus {\cal C}_s},\cdot ^{\nabla a,b})\cong {\rm Con}_{01}({\cal C}_{r-1})\times {\rm Con}_{01}({\cal C}_{s-1})\cong \begin{cases}{\cal C}_2^{s-4},& r=3,\\ {\cal C}_2^{r+s-8},& r\geq 4;\end{cases}$

\noindent ${\rm Con}_{\WDL 01}({\cal C}_r\boxplus {\cal C}_s,\cdot ^{\Delta c,d},\cdot ^{\nabla a,b})={\rm Con}_{\WDL 0}({\cal C}_r\boxplus {\cal C}_s,\cdot ^{\Delta c,d},\cdot ^{\nabla a,b})={\rm Con}_{\WDL 1}({\cal C}_r\boxplus {\cal C}_s,\cdot ^{\Delta c,d},\cdot ^{\nabla a,b})\cong {\rm Con}_{01}({\cal C}_{r-2})\times {\rm Con}_{01}({\cal C}_{s-2})\cong \begin{cases}{\cal C}_1,& r\leq 4,s\leq 4;\\ {\cal C}_2^{r-5},& r\geq 5,s\leq 4;\\ {\cal C}_2^{s-5},& r\leq 4,s\geq 5;\\ {\cal C}_2^{r+s-10},& r\geq 5,s\geq 5.\end{cases}$\label{exhsums}\end{example}

\section{An Application: Determining the Largest Numbers of Congruences of Finite (Dual) Weakly Complemented Lattices}

\begin{lemma}{\rm \cite{eucgbiklaol}} If $A$ is a congruence--distributive algebra and $\alpha $ is an atom of the lattice of congruences of $A$, then $A$ has at most twice as many congruences as $A/\alpha $.\label{atcg}\end{lemma}

Let $A$ be a member of a variety $\V $ of lattice--ordered algebras. Then, for any $S\subseteq A$, we have $S\subseteq Cg_A(S)\subseteq Cg_{\V ,A}(S)$, hence $Cg_{\V ,A}(S)=Cg_{\V ,A}(Cg_A(S))$.

Clearly, any atom of a congruence lattice is a principal congruence. By the convexity of the congruence classes of lattice congruences, any principal congruence of a lattice and thus any principal congruence of $A$ is generated by a pair of elements $a,b$ with $a\leq b$. Thus, clearly, if $A$ is finite, then any atom of the lattice of congruences of $A$ is generated by a pair of elements $a,b\in A$ with $a\prec b$.

\begin{remark} Let $(L,\cdot ^{\Delta })$ be a weakly complemented lattice.

By the above, if $L$ is finite and $\alpha \in {\rm At}({\rm Con}_{\WCL }(L))$, then, for some $a,b\in L$ with $a\prec b$, we have $\alpha =Cg_{\WCL ,L}(a,b)=Cg_{\WCL ,L}(Cg_L(a,b))$.

Let us also note that, for any $a\in L$, no proper congruence $\theta $ of $(L,\cdot ^{\Delta })$ can have $a$ and $a^{\Delta }$ in the same class, because then we would have $a/\theta =a^{\Delta }/\theta =(a\vee a^{\Delta })/\theta =1/\theta $, hence $1/\theta =a^{\Delta }/\theta =1^{\Delta }/\theta =0/\theta $.\label{onWCLcg}\end{remark}

For any set $M$ and any nonempty subset $S\subseteq M$, let us denote by $\varepsilon _M(S)$ the equivalence on $M$ having $S$ as a class and all other classes singletons: $\varepsilon _M(S)=eq(\{S\}\cup \{\{x\}:x\in M\setminus S\})$. Clearly, for any nonempty family $(S_i)_{i\in I}$ of parwise disjoint nonempty subsets of $M$, we have, in the lattice ${\rm Eq}(M)$: $\displaystyle \bigvee _{i\in I}\varepsilon _M(S_i)=eq(\{S_i:i\in I\}\cup \{\{x\}:x\in M\setminus \bigcup _{i\in I}S_i\})=\bigcup _{i\in I}\varepsilon _M(S_i)$. For brevity, if $a_1,\ldots ,a_k\in M$ for some $k\in \N ^*$, then we denote by $\varepsilon _M(a_1,\ldots ,a_k)=\varepsilon _M(\{a_1,\ldots ,a_k\})$.

For statement (\ref{pcg3}) of the following lemma, see \cite[Lemma $3.3$ and Remark $3.4$]{eucfifin}.

\begin{lemma}{\rm \cite{gcze,gratzer,grafin,eucfifin}} If $L$ is a finite lattice and $a,b\in L$ are such that $a\prec b$:\begin{enumerate}
\item\label{pcg1} $|L/Cg_L(a,b)|=|L|-1$ iff $a\in {\rm Mi}(L)$ and $b\in {\rm Ji}(L)$ iff $Cg_L(a,b)=\varepsilon _L(a,b)$;
\item\label{pcg2} $|L/Cg_L(a,b)|=|L|-2$ iff the following or its dual (by dual meaning the case when $b$ has a predecessor $c$ such that in $L^d$ the following hold for the interval $[a\wedge c,b]_L$ instead of $[a,b\vee c]_L$) holds: $a\notin {\rm Mi}(L)$ and, for some $c\in L$, $a\prec c$, $b\prec b\vee c$, $c\prec b\vee c$, $[a,b\vee c]_L=\{a,b,c,b\vee c\}\cong {\cal C}_2^2$ and $Cg_L(a,b)=\varepsilon _L(a,b)\cup \varepsilon _L(c,b\vee c)$;
\item\label{pcg3} $|L/Cg_L(a,b)|=|L|-3$ iff the following or its dual (with the dual of the following being formulated as above) holds: $a\notin {\rm Mi}(L)$ so, for some $c\in L\setminus \{b\}$, we have $a\prec c$, and we are in one of the following situations, depicted in the following diagrams:

\textcircled{1}$\ $ $b\prec b\vee c$, $c\prec b\vee c$, $[a,b\vee c]_L=\{a,b,c,b\vee c\}\cong {\cal C}_2^2$ and $Cg_L(a,b)=\varepsilon _L([a,b\vee c]_L)$;

\textcircled{2}$\ $ $c\prec b\vee c$ and, for some $d\in L$, $b\prec d\prec b\vee c$, $[a,b\vee c]_L=\{a,b,c,d,b\vee c\}\cong {\cal N}_5$ and $Cg_L(a,b)=\varepsilon _L(a,b,d)\cup \varepsilon _L(c,b\vee c)$;

\textcircled{3}$\ $ $b\prec b\vee c$ and, for some $d\in L$, $c\prec d\prec b\vee c$, $[a,b\vee c]_L=\{a,b,c,d,b\vee c\}\cong {\cal N}_5$ and $Cg_L(a,b)=\varepsilon _L(a,b)\cup \varepsilon _L(c,d,b\vee c)$;

\textcircled{4}$\ $ $b\prec b\vee c$, $c\prec b\vee c$, $[a,b\vee c]_L=\{a,b,c,b\vee c\}\cong {\cal C}_2^2$ and, for some $d,e\in L\setminus \{a,b,c,b\vee c\}$ such that $d\prec e$, $Cg_L(a,b)=\varepsilon _L(a,b)\cup \varepsilon _L(c,b\vee c)\cup \varepsilon _L(d,e)$.\end{enumerate}\label{pcg}\end{lemma}

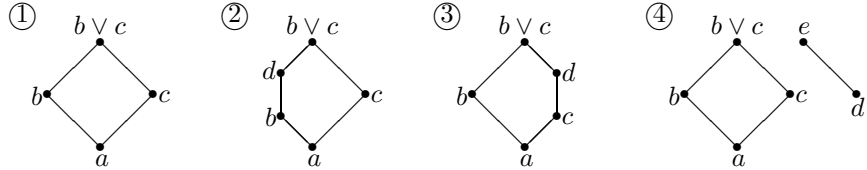
\begin{center}\begin{tabular}{cccc}\begin{picture}(40,45)(0,0)
\put(-15,47){\textcircled{1}}
\put(20,0){\circle*{3}}
\put(20,40){\circle*{3}}
\put(0,20){\circle*{3}}
\put(40,20){\circle*{3}}
\put(20,0){\line(-1,1){20}}
\put(20,0){\line(1,1){20}}
\put(20,40){\line(-1,-1){20}}
\put(20,40){\line(1,-1){20}}
\put(18,-7){$a$}
\put(10,43){$b\vee c$}
\put(-6,16){$b$}
\put(42,17){$c$}
\end{picture}
&\hspace*{25pt}
\begin{picture}(40,45)(0,0)
\put(-15,47){\textcircled{2}}
\put(20,0){\circle*{3}}
\put(20,40){\circle*{3}}
\put(8,12){\circle*{3}}
\put(8,28){\circle*{3}}
\put(40,20){\circle*{3}}
\put(20,0){\line(-1,1){12}}
\put(20,0){\line(1,1){20}}
\put(20,40){\line(-1,-1){12}}
\put(20,40){\line(1,-1){20}}
\put(8,12){\line(0,1){16}}
\put(18,-7){$a$}
\put(10,43){$b\vee c$}
\put(2,8){$b$}
\put(1,26){$d$}
\put(42,17){$c$}
\end{picture}
&\hspace*{25pt}
\begin{picture}(40,45)(0,0)
\put(-15,47){\textcircled{3}}
\put(20,0){\circle*{3}}
\put(20,40){\circle*{3}}
\put(0,20){\circle*{3}}
\put(32,12){\circle*{3}}
\put(32,28){\circle*{3}}
\put(20,0){\line(-1,1){20}}
\put(20,0){\line(1,1){12}}
\put(20,40){\line(-1,-1){20}}
\put(20,40){\line(1,-1){12}}
\put(32,12){\line(0,1){16}}
\put(18,-7){$a$}
\put(10,43){$b\vee c$}
\put(-6,16){$b$}
\put(34,9){$c$}
\put(34,25){$d$}
\end{picture}
&\hspace*{25pt}
\begin{picture}(60,45)(0,0)
\put(-15,47){\textcircled{4}}
\put(20,0){\circle*{3}}
\put(20,40){\circle*{3}}
\put(0,20){\circle*{3}}
\put(40,20){\circle*{3}}
\put(20,0){\line(-1,1){20}}
\put(20,0){\line(1,1){20}}
\put(20,40){\line(-1,-1){20}}
\put(20,40){\line(1,-1){20}}
\put(18,-7){$a$}
\put(10,43){$b\vee c$}
\put(-6,16){$b$}
\put(42,17){$c$}
\put(45,40){\circle*{3}}
\put(65,20){\circle*{3}}
\put(65,20){\line(-1,1){20}}
\put(63,12){$d$}
\put(43,43){$e$}
\end{picture}
\end{tabular}\end{center}

Recall from Section \ref{thealg} that, for any bounded lattice $L$, $\cdot ^{\Delta L}$ is the trivial weak complementation on $L$.

Throughout the rest of this section, we will be using the notation in Sections \ref{coatomic} and \ref{hsums} for the operation $\cdot ^{\Delta a,b}$ on a bounded lattice $L$ with two different coatoms $a$ and $b$ defined by: $1^{\Delta a,b}=0$, $a^{\Delta a,b}=b$, $b^{\Delta a,b}=a$ and $x^{\Delta a,b}=1$ for any $x\in L\setminus \{a,b,1\}$. Recall from Proposition \ref{2coat}, (\ref{2coat2}), that, if $L$ is distributive, then $\cdot ^{\Delta a,b}$ is a weak complementation on $L$, and from Theorem \ref{mainth} that, if $L$ is the horizontal sum of a coatomic bounded lattice with the unique coatom $a$ and a coatomic bounded lattice with the unique coatom $b$, then $\cdot ^{\Delta a,b}$ is the unique nontrivial weak complementation on $L$.

\begin{lemma} For any finite weakly complemented lattice $(L,\cdot ^{\Delta })$ and any $a,b\in L$ such that $a\prec b$, we have:\begin{enumerate}
\item\label{pWCLcg1} $|L/Cg_{\WCL ,L}(a,b)|=|L|-1$ iff $Cg_{\WCL ,L}(a,b)=Cg_L(a,b)=\varepsilon _L(a,b)$ iff $a\in {\rm Mi}(L)$, $b\in {\rm Ji}(L)$ and either $L=\{a,b\}\cong {\cal C}_2$ or $a^{\Delta }=b^{\Delta }=1$;

\item\label{pWCLcg2} $|L/Cg_{\WCL ,L}(a,b)|=|L|-2$ iff we are in one of the following cases:

\textcircled{$\alpha $}$\ $ $b=1$ and $L=\{0,a,1\}\cong {\cal C}_3$, in particular $Cg_L(a,b)=\varepsilon _L(a,b)\subsetneq Cg_{\WCL ,L}(a,b)$;

\textcircled{$\beta $}$\ $ $Cg_{\WCL ,L}(a,b)=Cg_L(a,b)$, $Cg_L(a,b)$ is as in Lemma \ref{pcg}, (\ref{pcg2}), and one of the following holds:

\quad \textcircled{$\beta\! _1$}$\ $ $a^{\Delta }=b^{\Delta }$ and, in the case when $a\prec c\in L\setminus \{b\}$, $c^{\Delta }=(b\vee c)^{\Delta }$, while, in the lattice dual of this situation, $c^{\Delta }=(a\wedge c)^{\Delta }$;

\quad \textcircled{$\beta\! _2$}$\ $ in the case when $a\prec c\in L\setminus \{b\}$, $L=\{a,b,c,b\vee c\}\cong {\cal C}_2^2$ and $\cdot ^{\Delta }=\cdot ^{\Delta b,c}$ is the Boolean complementation, and, in the lattice dual of this situation, $L=\{a\wedge c,c,a,b\}\cong {\cal C}_2^2$ and $\cdot ^{\Delta }=\cdot ^{\Delta a,c}$ is the Boolean complementation;

\item\label{pWCLcg3} $|L/Cg_{\WCL ,L}(a,b)|=|L|-3$ iff we are in one of the following cases or their lattice duals:

\textcircled{$\gamma $}$\ $ $a\in {\rm Mi}(L)$, $b\in {\rm Ji}(L)$, $Cg_L(a,b)=\varepsilon _L(a,b)$ and $|L/Cg_{\WCL ,L}(a,b)|=|L/Cg_L(a,b)|-2$, case in which $L\cong {\cal C}_4$, $b=1$ and $a$ is the coatom of $L$;

\textcircled{$\delta $}$\ $ $a\prec c$ for some $c\in L\setminus \{a\}$ such that $b\prec b\vee c$, $c\prec b\vee c$ and $[a,b\vee c]_L=\{a,b,c,b\vee c\}\cong {\cal C}_2^2$, $Cg_L(a,b)=\varepsilon _L(a,b)\cup \varepsilon _L(c,b\vee c)$, and $|L/Cg_{\WCL ,L}(a,b)|=|L/Cg_L(a,b)|-1$, case in which we are in one of the following subcases:

\quad \textcircled{$\delta\! _1$}$\ $ $a=0$ and $b\vee c=1$, so $L=\{a,b,c,b\vee c\}\cong {\cal C}_2^2$, and $\cdot ^{\Delta }=\cdot ^{\Delta L}$ in the trivial weak complementation, so $Cg_{\WCL ,L}(a,b)=L^2$;

\quad \textcircled{$\delta\! _2$}$\ $ $0\prec a$, $b\vee c=1$, $L=\{0,a,b,c,1\}\cong {\cal C}_2\oplus {\cal C}_2^2$, $\cdot ^{\Delta }=\cdot ^{\Delta b,c}$ restricts to the Boolean complementation on ${\cal C}_2^2$ and $Cg_{\WCL ,L}(a,b)=\varepsilon _L(0,a,b)\cup \varepsilon _L(c,1)$;

\quad \textcircled{$\delta\! _3$}$\ $ $(b\vee c)^{\Delta }\prec c^{\Delta }<b^{\Delta }=a^{\Delta }$, in particular $\cdot ^{\Delta }$ has at least three distinct values and thus it is nontrivial, $Cg_{\WCL ,L}(a,b)=\varepsilon _L(a,b)\cup \varepsilon _L(c,b\vee c)\cup \varepsilon _L(c^{\Delta },(b\vee c)^{\Delta })$, $a^{\Delta \Delta }=b^{\Delta \Delta }=a$, $c^{\Delta \Delta }=c$ and $(b\vee c)^{\Delta \Delta }=b\vee c$;

\quad \textcircled{$\delta\! _4$}$\ $ $(b\vee c)^{\Delta }=b^{\Delta }\prec a^{\Delta }=c^{\Delta }$, in particular $b\vee c\neq 1$, so $\cdot ^{\Delta }$ has at least three distinct values and thus it is nontrivial, $Cg_{\WCL ,L}(a,b)=\varepsilon _L(a,b)\cup \varepsilon _L(c,b\vee c)\cup \varepsilon _L(a^{\Delta },b^{\Delta })$, $a^{\Delta \Delta }=c^{\Delta \Delta }=a$ and $b^{\Delta \Delta }=(b\vee c)^{\Delta \Delta }=b$;

\textcircled{$\epsilon $}$\ $ $Cg_{\WCL ,L}(a,b)=Cg_L(a,b)$ and $Cg_L(a,b)$ is as in case \textcircled{1} in Lemma \ref{pcg}, (\ref{pcg3}), case in which one of the following holds:

\quad \textcircled{$\epsilon _1$}$\ $ $a^{\Delta }=b^{\Delta }=c^{\Delta }=(b\vee c)^{\Delta }$;

\quad \textcircled{$\epsilon _2$}$\ $ $L=\{a,b,c,b\vee c\}\cong {\cal C}_2^2$ and $\cdot ^{\Delta }=\cdot ^{\Delta b,c}$ is the Boolean complementation;

\textcircled{$\varphi $}$\ $ $Cg_{\WCL ,L}(a,b)=Cg_L(a,b)$ and $Cg_L(a,b)$ is as in case \textcircled{2} in Lemma \ref{pcg}, (\ref{pcg3}), case in which one of the following holds:

\quad \textcircled{$\varphi _1$}$\ $ $a^{\Delta }=b^{\Delta }=d^{\Delta }$ and $c^{\Delta }=(b\vee c)^{\Delta }$, in particular $b\vee c\neq 1$;

\quad \textcircled{$\varphi _2$}$\ $ $L=\{a,b,c,d,b\vee c\}\cong {\cal N}_5$ and $\cdot ^{\Delta }=\cdot ^{\Delta c,d}$ is the nontrivial weak complementation on ${\cal N}_5$;

\textcircled{$\psi $}$\ $ $Cg_{\WCL ,L}(a,b)=Cg_L(a,b)$ and $Cg_L(a,b)$ is as in case \textcircled{3} in Lemma \ref{pcg}, (\ref{pcg3}), case in which one of the following holds:

\quad \textcircled{$\psi _1$}$\ $ $a^{\Delta }=b^{\Delta }$ and $c^{\Delta }=d^{\Delta }=(b\vee c)^{\Delta }$, in particular $b\vee c\neq 1$;

\quad \textcircled{$\psi _2$}$\ $ $L=\{a,b,c,d,b\vee c\}\cong {\cal N}_5$ and $\cdot ^{\Delta }=\cdot ^{\Delta b,d}$ is the nontrivial weak complementation on ${\cal N}_5$;

\textcircled{$\chi $}$\ $ $Cg_{\WCL ,L}(a,b)=Cg_L(a,b)$ and $Cg_L(a,b)$ is as in case \textcircled{4} in Lemma \ref{pcg}, (\ref{pcg3}), case in which one of the following holds:

\quad \textcircled{$\chi\! _1$}$\ $ $a^{\Delta }=b^{\Delta }$, $c^{\Delta }=(b\vee c)^{\Delta }$ and $d^{\Delta }=e^{\Delta }$, in particular $b\vee c\neq 1$ and $e\neq 1$;

\quad \textcircled{$\chi\! _2$}$\ $ $\cdot ^{\Delta }$ is nontrivial and ${\cal C}_4\boxplus {\cal C}_4\cong \{c\wedge d,c,d,e,b\vee c,1\}$ is a sublattice of $L$;

\quad \textcircled{$\chi\! _3$}$\ $ $a=0$, $e=1$, ${\cal C}_2\times {\cal C}_3\cong \{a,b,c,b\vee c,d,e\}$ is a bounded sublattice of $L$, $b^{\Delta }=d$ and $d^{\Delta }=b$;

\quad \textcircled{$\chi\! _4$}$\ $ $d=0$, $b\vee c=1$, ${\cal C}_2\times {\cal C}_3\cong \{d,e,a,b,c,b\vee c\}$ is a bounded sublattice of $L$, $e^{\Delta }=c$ and $c^{\Delta }=e$.\end{enumerate}\label{pWCLcg}\end{lemma}

\begin{proof} We will repeatedly use Remark \ref{onWCLcg}.

\noindent (\ref{pWCLcg1}) By Lemma \ref{pcg}, (\ref{pcg1}), $a\in {\rm Mi}(L)$ and $b\in {\rm Ji}(L)$. But $b=(b\wedge a)\vee (b\wedge a^{\Delta })=a\vee (b\wedge a^{\Delta })$, and $a<b$, hence $b\wedge a^{\Delta }=b$, so $a^{\Delta }\geq b>a$, thus, since $a\vee a^{\Delta }=1$, it follows that $a^{\Delta }=1$, and, since $Cg_{\WCL ,L}(a,b)$ only collapses $a$ with $b$, either $b^{\Delta }=a^{\Delta }=1$ or $b^{\Delta }\in \{a,b\}$, case in which $Cg_{\WCL ,L}(a,b)=L^2$, thus $|L|=|L/Cg_{\WCL ,L}(a,b)|+1=1+1=2$, so $L=\{a,b\}\cong {\cal L}_2$.

\noindent (\ref{pWCLcg2}) The fact that $|L/Cg_{\WCL ,L}(a,b)|=|L|-2$ implies that we are in one of the following cases:

\textcircled{$\alpha $}$\ $ $a\in {\rm Mi}(L)$ and $b\in {\rm Ji}(L)$, so that $Cg_L(a,b)=\varepsilon _L(a,b)$ by Lemma \ref{pcg}, (\ref{pcg1}), and $|L/Cg_{\WCL ,L}(a,b)|=|L/Cg_L(a,b)|-1$, in particular $Cg_L(a,b)\subsetneq Cg_{\WCL ,L}(a,b)$;

\textcircled{$\beta $}$\ $ $Cg_{\WCL ,L}(a,b)=Cg_L(a,b)$, which is as in Lemma \ref{pcg}, (\ref{pcg2}), and below we will consider the case when $a\prec c\in L\setminus \{b\}$, with its dual being treated similarly.

We can not have $a\in {\rm Mi}(L)$, $b\in {\rm Ji}(L)$ and $a^{\Delta }=b^{\Delta }$, because then $Cg_{\WCL ,L}(a,b)=Cg_L(a,b)=\varepsilon _L(a,b)$, thus $|L/Cg_{\WCL ,L}(a,b)|=|L|-1$, which would contradict the current hypothesis.

Hence, in the case \textcircled{$\alpha $}, where $a\in {\rm Mi}(L)$ and $b\in {\rm Ji}(L)$, so that $Cg_L(a,b)=\varepsilon _L(a,b)$, we have $a^{\Delta }\neq b^{\Delta }$, thus also $a^{\Delta \Delta }\neq b^{\Delta \Delta }$, since otherwise we would get $a^{\Delta }=a^{\Delta \Delta \Delta }=b^{\Delta \Delta \Delta }=b^{\Delta }$. Of course, since $(a,b)\in Cg_{\WCL ,L}(a,b)$, we also have $(a^{\Delta },b^{\Delta }),(a^{\Delta \Delta },b^{\Delta \Delta })\in Cg_{\WCL ,L}(a,b)$, and, since in this case $|L/Cg_{\WCL ,L}(a,b)|=|L/Cg_L(a,b)|-1$, either at most one of the elements $a^{\Delta },b^{\Delta },a^{\Delta \Delta },b^{\Delta \Delta }$ does not belong to $\{a,b\}=a/Cg_L(a,b)=b/Cg_L(a,b)\subseteq a/Cg_{\WCL ,L}(a,b)=b/Cg_{\WCL ,L}(a,b)$, so that $a^{\Delta }$ or $b^{\Delta }$ belongs to $a/Cg_{\WCL ,L}(a,b)=b/Cg_{\WCL ,L}(a,b)$, or\linebreak $a^{\Delta }/Cg_{\WCL ,L}(a,b)=\{a^{\Delta },b^{\Delta }\}=\{a^{\Delta \Delta },b^{\Delta \Delta }\}$; in either of these subcases, we have $x^{\Delta }\in Cg_{\WCL ,L}(a,b)$ for some $x\in L$, hence $Cg_{\WCL ,L}(a,b)=L^2$, so that $|L/Cg_{\WCL ,L}(a,b)|=1$, thus $|L|=|L/Cg_{\WCL ,L}(a,b)|+2=3$, therefore $L\cong {\cal C}_3$, so $\cdot ^{\Delta }$ is the trivial weak complementation, thus, since $a^{\Delta }\neq b^{\Delta }$ and $a\prec b$, it follows that $b=1$ and $a$ is the single element of $L\setminus \{0,1\}$.

In case \textcircled{$\beta $}, since $Cg_{\WCL ,L}(a,b)$ collapses no other elements but $a$ with $b$ and $c$ with $b\vee c$:

either $a^{\Delta }=b^{\Delta }$ and $c^{\Delta }=(b\vee c)^{\Delta }$,

or $x^{\Delta }\in \{a,b,c,b\vee c\}$ for some $x\in \{a,b,c,b\vee c\}$, but then, since $x\vee x^{\Delta }=1$ and $b\vee c=\max \{a,b,c,b\vee c\}$, it follows that $b\vee c=1$, thus $c/\alpha =1/\alpha $, hence $c^{\Delta }/\alpha =0/\alpha $, but, since $c\neq 1$ and thus $c^{\Delta }\neq 0$, we have $0,c^{\Delta }\in \{a,b,c,b\vee c\}$, hence $a=0$ and $c^{\Delta }=b$ since $a=\min \{a,b,c,b\vee c\}$, therefore, by Lemma \ref{pcg}, (\ref{pcg2}), $L=[0,1]_L=[a,b\vee c]_L=\{a,b,c,b\vee c\}\cong {\cal C}_2^2$ and $\cdot ^{\Delta }=\cdot ^{\Delta b,c}$.

\noindent (\ref{pWCLcg3}) $|L/Cg_{\WCL ,L}(a,b)|=|L|-3$ iff one of the following holds:

\textcircled{$\gamma $} $|L/Cg_L(a,b)|=|L|-1$ and $|L/Cg_{\WCL ,L}(a,b)|=|L/Cg_L(a,b)|-2$, so that $Cg_L(a,b)$ is as in Lemma \ref{pcg}, (\ref{pcg1}), and $Cg_L(a,b)\subsetneq Cg_{\WCL ,L}(a,b)$;

\textcircled{$\delta $} $|L/Cg_L(a,b)|=|L|-2$ and $|L/Cg_{\WCL ,L}(a,b)|=|L/Cg_L(a,b)|-1$, so that $Cg_L(a,b)$ is as in Lemma \ref{pcg}, (\ref{pcg2}), and $Cg_L(a,b)\subsetneq Cg_{\WCL ,L}(a,b)$;

\textcircled{$\zeta $} $|L/Cg_{\WCL ,L}(a,b)|=|L/Cg_L(a,b)|=|L|-3$, so that $Cg_{\WCL ,L}(a,b)=Cg_L(a,b)$, which is as in Lemma \ref{pcg}, (\ref{pcg3}).

\textcircled{$\gamma $} In this subcase, by Lemma \ref{pcg}, (\ref{pcg1}), $a\in {\rm Mi}(L)$, $b\in {\rm Ji}(L)$ and $Cg_L(a,b)=\varepsilon _L(a,b)$, and, since $Cg_L(a,b)\subsetneq Cg_{\WCL ,L}(a,b)$, we have $a^{\Delta }\neq b^{\Delta }$, thus also $a^{\Delta \Delta }\neq b^{\Delta \Delta }$.

Since $(a,b)\in Cg_{\WCL ,L}(a,b)$, we also have $(a^{\Delta },b^{\Delta }),(a^{\Delta \Delta },b^{\Delta \Delta })\in Cg_{\WCL ,L}(a,b)$, thus $Cg_{\WCL ,L}(a,b)\supseteq Cg_L(a,b)\vee Cg_L(a^{\Delta },b^{\Delta })\vee Cg_L(a^{\Delta \Delta },b^{\Delta \Delta })=\varepsilon _L(a,b)\vee Cg_L(a^{\Delta },b^{\Delta })\vee Cg_L(a^{\Delta \Delta },b^{\Delta \Delta })$.

$\bullet \ $ If the sets $\{a,b\}$, $\{a^{\Delta },b^{\Delta }\}$ and $\{a^{\Delta \Delta },b^{\Delta \Delta }\}$ are pairwise disjoint, then $\cdot ^{\Delta }$ is nontrivial, and the fact that $|L/Cg_{\WCL ,L}(a,b)|=|L|-3\geq 3$ ensures us that $Cg_{\WCL ,L}(a,b)=\varepsilon _L(a,b)\cup \varepsilon _L(a^{\Delta },b^{\Delta })\cup \varepsilon _L(a^{\Delta \Delta },b^{\Delta \Delta })\subsetneq L^2$, hence $b^{\Delta }\prec a^{\Delta }$, $a^{\Delta \Delta }\prec b^{\Delta \Delta }$ and either $Cg_L(a^{\Delta },b^{\Delta })=\varepsilon _L(a^{\Delta },b^{\Delta })$ and $Cg_L(a^{\Delta \Delta },b^{\Delta \Delta })=\varepsilon _L(a^{\Delta \Delta },b^{\Delta \Delta })$ or $Cg_L(a^{\Delta },b^{\Delta })=Cg_L(a^{\Delta \Delta },b^{\Delta \Delta })=\varepsilon _L(a^{\Delta },b^{\Delta })\cup \varepsilon _L(a^{\Delta \Delta },b^{\Delta \Delta })$. Note also that $a^{\Delta \Delta }<a$ and $b^{\Delta \Delta }<b$ since $a^{\Delta \Delta }\leq a$, $b^{\Delta \Delta }\leq b$ and $\{a,b\}\cap \{a^{\Delta \Delta },b^{\Delta \Delta }\}=\emptyset $.

If $Cg_L(a^{\Delta },b^{\Delta })=Cg_L(a^{\Delta \Delta },b^{\Delta \Delta })=\varepsilon _L(a^{\Delta },b^{\Delta })\cup \varepsilon _L(a^{\Delta \Delta },b^{\Delta \Delta })$, then, by Lemma \ref{pcg}, (\ref{pcg2}), $\{a^{\Delta },b^{\Delta },a^{\Delta \Delta },b^{\Delta \Delta }\}\cong {\cal C}_2^2$, so that $1=a^{\Delta }\vee a^{\Delta \Delta }\in {\rm Max}(\{a^{\Delta },b^{\Delta },a^{\Delta \Delta },b^{\Delta \Delta }\})\subseteq \{a^{\Delta },b^{\Delta \Delta }\}$. Since $b^{\Delta \Delta }<b$, it follows that $b^{\Delta \Delta }\neq 1$, hence $a^{\Delta }=1$ and, by Lemma \ref{pcg}, (\ref{pcg2}), $a^{\Delta \Delta }\prec b^{\Delta }$ and $b^{\Delta \Delta }\prec a^{\Delta }=1$; but $b^{\Delta \Delta }<b\leq 1=a^{\Delta }$, thus $b=1=a^{\Delta }$, contradicting the fact that $\{a,b\}$ and $\{a^{\Delta },b^{\Delta }\}$ are disjoint.

Hence $Cg_L(a^{\Delta },b^{\Delta })=\varepsilon _L(a^{\Delta },b^{\Delta })$ and $Cg_L(a^{\Delta \Delta },b^{\Delta \Delta })=\varepsilon _L(a^{\Delta \Delta },b^{\Delta \Delta })$, so that, by Lemma \ref{pcg}, (\ref{pcg1}), $b^{\Delta },a^{\Delta \Delta }\in {\rm Mi}(L)$ and $a^{\Delta },b^{\Delta \Delta }\in {\rm Ji}(L)$. But $a^{\Delta }=(a^{\Delta }\wedge b^{\Delta })\vee (a^{\Delta }\wedge b^{\Delta \Delta })=b^{\Delta }\vee (a^{\Delta }\wedge b^{\Delta \Delta })$, and, since $b^{\Delta }<a^{\Delta }\in {\rm Ji}(L)$, it follows that $a^{\Delta }=a^{\Delta }\wedge b^{\Delta \Delta }$, thus $a^{\Delta }\leq b^{\Delta \Delta }>a^{\Delta \Delta }$, hence $1=a^{\Delta }\vee a^{\Delta \Delta }\leq b^{\Delta \Delta }$, thus $1=b^{\Delta \Delta }<b$, a contradiction.

$\bullet \ $If $\{a,b\}\cap \{a^{\Delta },b^{\Delta }\}\neq \emptyset $ or $\{a^{\Delta },b^{\Delta }\}\cap \{a^{\Delta \Delta },b^{\Delta \Delta }\}\neq \emptyset $, then $Cg_{\WCL ,L}(a,b)=L^2$, thus $|L/Cg_{\WCL ,L}(a,b)|=1$, hence $|L|=1+3=4$, so $L\cong {\cal C}_2^2$ or $L\cong {\cal C}_4$, but $L\cong {\cal C}_2^2$ would contradict the fact that $a\prec b$, $a\in {\rm Mi}(L)$ and $b\in {\rm Ji}(L)$, thus $L\cong {\cal C}_4$, so $\cdot ^{\Delta }=\cdot ^{\Delta (L)}$, and, since $a\prec b$ and $a^{\Delta }\neq b^{\Delta }$, it follows that $b=1$ and $a$ is the coatom of $L$. 

$\bullet \ $ If $\{a,b\}\cap \{a^{\Delta \Delta },b^{\Delta \Delta }\}\neq \emptyset $, then, since $a^{\Delta \Delta }<b^{\Delta \Delta }$, but also $a^{\Delta \Delta }\leq a$ and $b^{\Delta \Delta }\leq b$, it follows that $b^{\Delta \Delta }=a$, so, by the fact that $|L/Cg_{\WCL ,L}(a,b)|=|L|-3$, either, as above, $Cg_{\WCL ,L}(a,b)=L^2$, $L\cong {\cal C}_4$ and $b=1$, or $\{a,b\}\cap \{a^{\Delta },b^{\Delta }\}=\{a^{\Delta },b^{\Delta }\}\cap \{a^{\Delta \Delta },b^{\Delta \Delta }\}=\emptyset $, $a^{\Delta \Delta }\prec a=b^{\Delta \Delta }$ and $b^{\Delta }\prec a^{\Delta }$; however, the latter implies $b^{\Delta \Delta }\leq a$, hence $a^{\Delta }\leq b^{\Delta }$, contradicting $b^{\Delta }<a^{\Delta }$.

\textcircled{$\delta $} In this subcase, by Lemma \ref{pcg}, (\ref{pcg2}), the following or its lattice dual holds: $a\prec c$ for some $c\in L\setminus \{b\}$ such that $b\prec b\vee c$, $c\prec b\vee c$ and $Cg_L(a,b)=\varepsilon_L(a,b)\cup \varepsilon_L(c,b\vee c)$.

We can't have $a^{\Delta }=b^{\Delta }$ and $c^{\Delta }=(b\vee c)^{\Delta }$, because then we'd have $Cg_{\WCL ,L}(a,b)=Cg_L(a,b)$, contradicting the above.

\noindent $\blacksquare \ $ If $x^{\Delta }\in x/Cg_{\WCL ,L}(a,b)$ for some $x\in \{a,b,c,b\vee c\}$, then $Cg_{\WCL ,L}(a,b)=L^2$, thus $|L|=|L/Cg_{\WCL ,L}(a,b)|+3=1+3=4$, hence $L=\{a,b,c,b\vee c\}\cong {\cal C}_4$, with $a=0$ and $b\vee c=1$. The fact that $x^{\Delta }\in x/Cg_{\WCL ,L}(a,b)$ for some $x\in \{a,b,c,b\vee c\}$ for some $x\in \{a,b,c,b\vee c\}$ implies that $\cdot ^{\Delta }\neq \cdot ^{\Delta b,c}$, hence $\cdot ^{\Delta }=\cdot ^{\Delta (L)}$ (so that $c^{\Delta }=1\in c/Cg_{\WCL ,L}(a,b)$).

\noindent $\blacksquare \ $ Now let us assume that no $x\in \{a,b,c,b\vee c\}$ has $x^{\Delta }\in x/Cg_{\WCL ,L}(a,b)$.

\noindent $\ \blacktriangleright \ $ If $x^{\Delta }\in \{a,b,c,b\vee c\}$ for some $x\in \{a,b,c,b\vee c\}$, then $b\vee c=1$, hence $0=(b\vee c)^{\Delta }\in c^{\Delta }/Cg_{\WCL ,L}(a,b)$. If $a=0$, so that $L=[0,1]_L=\{a,b,c,b\vee c\}$, then $c^{\Delta }=b$, so $c^{\Delta }\leq b$, hence $b^{\Delta }\leq c$, thus $b^{\Delta }=c$, so $\cdot ^{\Delta }=\cdot ^{\Delta b,c}$, but then $Cg_{\WCL ,L}(a,b)=Cg_L(a,b)$, contradicting the above. Hence $a\neq 0$, thus $0\notin \{a,b,c,b\vee c\}$.

$\bullet \ $ If $c^{\Delta }\in \{a,b,c,b\vee c\}\setminus c/Cg_{\WCL ,L}(a,b)=\{a,b\}$, then $c^{\Delta }\leq b$, thus $b^{\Delta }\leq c$, and the fact that $c^{\Delta }/Cg_{\WCL ,L}(a,b)\linebreak =0/Cg_{\WCL ,L}(a,b)$ and $|L/Cg_{\WCL ,L}(a,b)|=|L|-3$ ensures us that $Cg_{\WCL ,L}(a,b)=\varepsilon _L(0,a,b)\cup \varepsilon _L(c,b\vee c)=\varepsilon _L(0,a,b)\cup \varepsilon _L(c,1)$, so $0\prec a$; then $c^{\Delta }\neq a$, because otherwise $c^{\Delta }\leq a$, thus $a^{\Delta }\leq c$, so $a\vee a^{\Delta }\leq c<1$, a contradiction; thus $c^{\Delta }=b$, hence $b^{\Delta }=c$, because $b^{\Delta }<c$ would imply $c^{\Delta }\leq b^{\Delta }<c$, a contradiction. If there existed any $x\in L\setminus \{0,a,b,c,b\vee c\}=L\setminus \{0,a,b,c,1\}$, then we can`t have $x>a$, because then $x\in {\rm CoAt}(L)\setminus \{b,c\}$, so that $\{a,b,c,x,1\}\cong {\cal M}_3$ is a sublattice of $L$, which would contradict the fact that $(b,c)\notin Cg_L(a,b)$ since ${\cal M}_3$ is a simple lattice. Thus there exists a $y\in {\rm At}(L)\setminus \{a\}$ and ${\rm CoAt}(L)=\{b,c\}$, therefore either $y<b$, case in which we'd get the contradiction $y\in b/Cg_L(a,b)$, or $y<c$, case in which we'd get the contradiction $y\in c/Cg_L(a,b)$. Therefore, in this case, $L=\{0,a,b,c,b\vee c\}=\{0,a,b,c,1\}\cong {\cal C}_2\oplus {\cal C}_2^2$ and $\cdot ^{\Delta }=\cdot ^{\Delta b,c}$.

$\bullet $ If $c^{\Delta }\notin \{a,b,c,b\vee c\}$, then the fact that $|L/Cg_{\WCL ,L}(a,b)|=|L|-3$ implies $0\prec c^{\Delta }$ and $Cg_{\WCL ,L}(a,b)=\varepsilon _L(0,c^{\Delta })\cup \varepsilon _L(a,b)\cup \varepsilon _L(b,b\vee c)$. Thus $c^{\Delta \Delta }\in 0^{\Delta }/Cg_{\WCL ,L}(a,b)=1/Cg_{\WCL ,L}(a,b)=\{c,1\}$ and $c^{\Delta \Delta }\leq c$, so $c^{\Delta \Delta }=c$. Also, by the same argument as above, the atom $c^{\Delta }$ can not be comparable to either of $b$ and $c$, thus $c^{\Delta }\vee b=1$ and $c>c\wedge c^{\Delta }$, thus $c\wedge c^{\Delta }=0$. But then $L$ has the bounded sublattice $\{0,c^{\Delta },c,1\}=\{0,c^{\Delta },c,b\vee c\}\cong {\cal C}_2^2$, so that $(0,c^{\Delta })\in Cg_L(c,b\vee c)=Cg_L(a,b)=\varepsilon _L(a,b)\cup \varepsilon _L(c,b\vee c)$, and we have a contradiction.

\noindent $\ \blacktriangleright \ $ The remaining subcase is $x^{\Delta }\notin \{a,b,c,b\vee c\}$ for any $x\in \{a,b,c,b\vee c\}$. But $(a,b),(c,b\vee c)\in Cg_L(a,b)\subsetneq Cg_{\WCL ,L}(a,b)$, thus $(a^{\Delta },b^{\Delta }),(c^{\Delta },(b\vee c)^{\Delta })\in Cg_{\WCL ,L}(a,b)$, and $a^{\Delta }\neq b^{\Delta }$ or $c^{\Delta }\neq (b\vee c)^{\Delta }$, thus $Cg_{\WCL ,L}(a,b)=\varepsilon _L(a,b)\cup \varepsilon _L(c,b\vee c)\cup \varepsilon _L(C)$, where:

$\bullet \ $ if $(a^{\Delta },c^{\Delta })\notin Cg_{\WCL ,L}(a,b)$, then one of the sets $\{a^{\Delta },b^{\Delta }\}$ and $\{c^{\Delta },(b\vee c)^{\Delta }\}$ is a singleton and the other one is a two--element class $C$ of $Cg_{\WCL ,L}(a,b)$;

$\ \ast \ $ if $a^{\Delta }\neq b^{\Delta }$, then $c^{\Delta }=(b\vee c)^{\Delta }$ and $a^{\Delta \Delta }\neq b^{\Delta \Delta }$, so $a^{\Delta \Delta },b^{\Delta \Delta }$ belong to a nonsingleton class of $Cg_{\WCL ,L}(a,b)$, which in the current subcase, since $a^{\Delta \Delta }\leq a$ and $b^{\Delta \Delta }\leq b$, can only be $\{a,b\}$, so $a^{\Delta \Delta }=a<b=b^{\Delta \Delta }$; also, $b^{\Delta }\neq a^{\Delta }=(b\wedge c)^{\Delta }=b^{\Delta }\vee c^{\Delta }$, thus $c^{\Delta }\nleq b^{\Delta }$, a contradiction;

$\ \centerdot \ $ if $a^{\Delta }=b^{\Delta }$, then $c^{\Delta }\neq (b\vee c)^{\Delta }$, so $b^{\Delta }=a^{\Delta }=(b\wedge c)^{\Delta }=b^{\Delta }\vee c^{\Delta }$ and $c^{\Delta \Delta }\neq (b\vee c)^{\Delta \Delta }$, thus $(b\vee c)^{\Delta }<c^{\Delta }\leq b^{\Delta }=a^{\Delta }$; but $(a^{\Delta },c^{\Delta })\notin Cg_{\WCL ,L}(a,b)$, so $a^{\Delta }\neq c^{\Delta }$, thus $(b\vee c)^{\Delta }\prec c^{\Delta }<b^{\Delta }=a^{\Delta }$ since here $C=\{c^{\Delta },(b\vee c)^{\Delta }\}$ is a class of $Cg_{\WCL ,L}(a,b)$, hence $(b\vee c)^{\Delta \Delta }>c^{\Delta \Delta }>b^{\Delta \Delta }=a^{\Delta \Delta }$ and $a^{\Delta \Delta }=b^{\Delta \Delta },c^{\Delta \Delta },(b\vee c)^{\Delta \Delta }\in \{a,b,c,b\vee c\}\cong {\cal C}_2^2$, thus, since $c^{\Delta \Delta }\leq c$, we have $a^{\Delta \Delta }=b^{\Delta \Delta }=a$, $c^{\Delta \Delta }=c$ and $(b\vee c)^{\Delta \Delta }=b\vee c$;

$\bullet \ $ if $(a^{\Delta },c^{\Delta })\in Cg_{\WCL ,L}(a,b)$, then $C=\{a^{\Delta },b^{\Delta },c^{\Delta },(b\vee c)^{\Delta }\}$ is a two--element class of $Cg_{\WCL ,L}(a,b)$, thus $(b\vee c)^{\Delta }<a^{\Delta }=(b\wedge c)^{\Delta }=b^{\Delta }\vee c^{\Delta }$; as in subcase $\ast $ above, we can not have $a^{\Delta }\neq b^{\Delta }$ and $c^{\Delta }=(b\vee c)^{\Delta }$, thus: $a^{\Delta }=b^{\Delta }$ iff $c^{\Delta }=(b\vee c)^{\Delta }$;

$\ \centerdot \ $ if $a^{\Delta }=c^{\Delta }$, then, by the above, $b^{\Delta }\neq a^{\Delta }=c^{\Delta }\neq (b\vee c)^{\Delta }$, so $b^{\Delta }=(b\vee c)^{\Delta }\prec a^{\Delta }=c^{\Delta }$ since $b^{\Delta }\leq a^{\Delta }$ and $\{a^{\Delta },b^{\Delta },c^{\Delta },(b\vee c)^{\Delta }\}$ is a two--element class of $Cg_{\WCL ,L}(a,b)$, hence $a^{\Delta \Delta }=c^{\Delta \Delta }<b^{\Delta \Delta }=(b\vee c)^{\Delta \Delta }$, but, since $(a,b)\in Cg_{\WCL ,L}(a,b)$, $a^{\Delta \Delta }$ and $b^{\Delta \Delta }$ must belong to a nonsingleton class of $Cg_{\WCL ,L}(a,b)$, that in the current subcase can only be $\{a,b\}$, thus $a^{\Delta \Delta }=c^{\Delta \Delta }=a\prec b=b^{\Delta \Delta }=(b\vee c)^{\Delta \Delta }$;

$\ \centerdot \ $ if $a^{\Delta }\neq c^{\Delta }$, then, by the above, $a^{\Delta }=b^{\Delta }$ and $c^{\Delta }=(b\vee c)^{\Delta }$, which implies $Cg_L(a,b)=\varepsilon _L(a,b)\cup \varepsilon _L(c,b\vee c)=Cg_{\WCL ,L}(a,b)$, a contradiction.

\textcircled{$\zeta $} In this subcase, we treat separately the subcases of Lemma \ref{pcg}, (\ref{pcg3}), using the notations from this lemma.

\textcircled{$\epsilon $} If $Cg_L(a,b)$ is as in case \textcircled{$1$} in Lemma \ref{pcg}, (\ref{pcg3}), then $\{a,b,c,b\vee c\}$ is the only nonsingleton class of $Cg_{\WCL ,L}(a,b)$, so we have either $a^{\Delta }=b^{\Delta }=c^{\Delta }=(b\vee c)^{\Delta }$ or $x^{\Delta }\in \{a,b,c,b\vee c\}$ for some $x\in \{a,b,c,b\vee c\}$, which implies $\alpha =L^2$, so that $|L/\alpha |=1$, thus $|L|=|L/\alpha |+3=4$, hence $L=\{a,b,c,b\vee c\}\cong {\cal C}_2^2$.

\textcircled{$\varphi $} If $Cg_L(a,b)$ is as in case \textcircled{$2$} in Lemma \ref{pcg}, (\ref{pcg3}), then $\{a,b,d\}$ and $\{c,b\vee c\}$ are the only nonsingleton classes of $Cg_{\WCL ,L}(a,b)$, hence either $a^{\Delta }=b^{\Delta }=d^{\Delta }$ and $c^{\Delta }=(b\vee c)^{\Delta }$ or $x^{\Delta }\in \{a,b,c,d,b\vee c\}$ for some $x\in \{a,b,c,d,b\vee c\}$, so that $1=x\vee x^{\Delta }\in \{a,b,c,d,b\vee c\}$, thus $b\vee c=1$, so $c/\alpha =1/\alpha $, hence $c^{\Delta }/\alpha =0/\alpha $ and $c^{\Delta }\neq 0$ since $c\neq 1$, therefore $0,c^{\Delta }\in \{a,b,c,d,b\vee c\}$, thus $a=0$, hence $L=[0,1]_L=[a,b\vee c]_L=\{a,b,c,d,b\vee c\}\cong {\cal N}_5$ and $\cdot ^{\Delta }$ is the only nontrivial weak complementation on ${\cal N}_5$, namely $\cdot ^{\Delta c,d}$ here.

\textcircled{$\psi $} Analogously for the situation when $Cg_L(a,b)$ is as in case \textcircled{$3$} in Lemma \ref{pcg}, (\ref{pcg3}). 

\textcircled{$\chi $} If $Cg_L(a,b)$ is as in case \textcircled{$4$} in Lemma \ref{pcg}, (\ref{pcg3}), so $Cg_{\WCL ,L}(a,b)=Cg_L(a,b)=\varepsilon 
_L(a,b)\cup \varepsilon 
_L(c,b\vee c)\cup \varepsilon 
_L(d,e)\neq L^2$, hence $x^{\Delta }\notin x/Cg_{\WCL ,L}(a,b)$ for any $x\in L$; then we are in one of the following subcases:

\noindent $\blacksquare \ $ $x^{\Delta }\notin \{a,b,c,b\vee c,d,e\}$ for any $x\in \{a,b,c,b\vee c,d,e\}$; then, since $Cg_{\WCL ,L}(a,b)$ collapses only $a$ with $b$, $c$ with $b\vee c$ and $d$ with $e$, it follows that $a^{\Delta }=b^{\Delta }$, $c^{\Delta }=(b\vee c)^{\Delta }$ and $d^{\Delta }=e^{\Delta }$;

\noindent $\blacksquare \ $ $x^{\Delta }\in \{a,b,c,b\vee c,d,e\}$ for some $x\in \{a,b,c,b\vee c,d,e\}$ and one of the following holds: $a^{\Delta }\neq b^{\Delta }$, $c^{\Delta }\neq (b\vee c)^{\Delta }$ or $d^{\Delta }\neq e^{\Delta }$; then, since $x\vee x^{\Delta }=1$ and ${\rm Max}\{a,b,c,b\vee c,d,e\}\subseteq \{b\vee c,e\}$, we have $b\vee c\vee e=1$; also, $\{a^{\Delta },b^{\Delta }, c^{\Delta },(b\vee c)^{\Delta },d^{\Delta },e^{\Delta }\}\nsubseteq \{1\}$, thus, since $\cdot ^{\Delta }$ is order--reversing, we have $(b\vee c)^{\Delta }<1$ or $e^{\Delta }<1$, in particular $\cdot ^{\Delta }$ is nontrivial since $b\vee c>a\geq 0$ and $e>d\geq 0$; hence we are in one of the following subcases:

$\bullet \ $ $b\vee c\neq 1$ and $e\neq 1$; then $1\notin \{a,b,c,b\vee c,d,e\}$, so $1/\alpha =\{1\}$, thus $(c\vee d)/\alpha =(b\vee c\vee e)/\alpha =1/\alpha =\{1\}$, so $c\vee d=1$, thus $c||d$ since $c<b\vee c\leq 1\geq e>d$, hence $c\wedge d\notin \{a,b,c,b\vee c,d,e\}$, so $(c\wedge d)/\alpha =\{c\wedge d\}$, thus $((b\vee c)\wedge e)/\alpha =(c\wedge d)/\alpha =\{c\wedge d\}$, hence $(b\vee c)\wedge e=c\wedge d$, therefore $\{c\wedge d,c,d,e,b\vee c,1\}$ is a sublattice of $L$ isomorphic to the hexagon: ${\cal C}_4\boxplus {\cal C}_4$;

$\bullet \ $ $b\vee c<1=e>d$, thus $d^{\Delta }\neq 0$ and $(d^{\Delta },0)=(d^{\Delta },e^{\Delta })\in Cg_{\WCL ,L}(a,b)=\varepsilon _L(a,b)\cup \varepsilon _L(c,b\vee c)\cup \varepsilon _L(d,e)$, hence $a=0$ and $d^{\Delta }=b\leq b\vee c$, hence $b^{\Delta }\leq d\geq (b\vee c)^{\Delta }$; but $b^{\Delta }\in a^{\Delta }/Cg_{\WCL ,L}(a,b)=1/Cg_{\WCL ,L}(a,b)=e/Cg_{\WCL ,L}(a,b)=\{d,e\}$, hence $b^{\Delta }=d$, while $b\vee c<1$ and $(b\vee c)^{\Delta }=d<1$ imply $b\vee c||d$, so $b\vee c>(b\vee c)\cap d\in ((b\vee c)\cap 1)/Cg_{\WCL ,L}(a,b)=(b\vee c)/Cg_{\WCL ,L}(a,b)=\{c,b\vee c\}$, thus $(b\vee c)\cap =c$, hence $\{0,b,c,b\vee c,d,1\}=\{a,b,c,b\vee c,d,e\}\cong {\cal C}_2\times {\cal C}_3$;

$\bullet \ $ $c<b\vee c=1>e$, thus $c^{\Delta }\neq 0$ and $(c^{\Delta },0)=(c^{\Delta },(b\vee c)^{\Delta })\in Cg_{\WCL ,L}(a,b)=\varepsilon _L(a,b)\cup \varepsilon _L(c,b\vee c)\cup \varepsilon _L(d,e)$; ${\rm Min}(\{a,b,c,b\vee c,d,e\})\subseteq \{a,d\}$, but we can't have $a=0$, because then we would get the contradiction $d,e\notin L=[0,1]_L=[a,b\vee c]_L=\{a,b,c,b\vee c\}$, thus $d=0$, hence $c^{\Delta }=e$; thus $e^{\Delta }\leq c$ and $e^{\Delta }\in a^{\Delta }/Cg_{\WCL ,L}(a,b)=0^{\Delta }/Cg_{\WCL ,L}(a,b)=1/Cg_{\WCL ,L}(a,b)=\{c,b\vee c\}$, thus $e^{\Delta }=c$; also, $e\nleq a$ since $a<c$ and $c\vee e=c\vee c^{\Delta }=1$, and $a\vee e\in (a\vee 0)/Cg_{\WCL ,L}(a,b)=a/Cg_{\WCL ,L}(a,b)=\{a,b\}$, thus $a\vee e=b$, hence $\{0,e,a,b,c,1\}=\{d,e,a,b,c,b\vee c\}\cong {\cal C}_2\times {\cal C}_3$.\end{proof}

\begin{theorem}{\rm \cite{gcze,free,eucfifin}} If $n\in \N ^*$ and $L$ is a lattice with $|L|=n$, then:\begin{itemize}
\item $|{\rm Con}(L)|\leq 2^{n-1}$;
\item $|{\rm Con}(L)|=2^{n-1}$ iff ${\rm Con}(L)\cong {\cal C}
_2^{n-1}$ iff $L\cong {\cal C}_n$;
\item if $|{\rm Con}(L)|<2^{n-1}$, then $|{\rm Con}(L)|\leq 2^{n-2}$;
\item $|{\rm Con}(L)|=2^{n-2}$ iff $n\geq 4$ and ${\rm Con}(L)\cong {\cal C}_2^{n-2}$ iff $n\geq 4$ and $L\cong {\cal C}_{n-k-2}\oplus {\cal C}_2^2\oplus {\cal C}_k$ for some $k\in [1,n-3]$;
\item if $|{\rm Con}(L)|<2^{n-2}$, then $|{\rm Con}(L)|\leq 5\cdot 2^{n-5}$;
\item $|{\rm Con}(L)|=5\cdot 2^{n-5}$ iff $n\geq 5$ and ${\rm Con}(L)\cong {\cal C}_2^{n-5}\times ({\cal C}_2\oplus {\cal C}_2^2)$ iff $n\geq 5$ and $L\cong {\cal C}_{n-k-3}\oplus {\cal N}_5\oplus {\cal C}_k$ for some $k\in [1,n-4]$;
\item if $|{\rm Con}(L)|<5\cdot 2^{n-5}$, then $|{\rm Con}(L)|\leq 2^{n-3}$;
\item $|{\rm Con}(L)|=2^{n-3}$ iff $n\geq 6$ and ${\rm Con}(L)\cong {\cal C}_2^{n-3}$ iff $n\geq 6$ and $L\cong {\cal C}_{n-k-4}\oplus ({\cal C}_2\times {\cal C}_3)\oplus {\cal C}_k$ for some $k\in [1,n-5]$ or $n\geq 7$ and $L\cong {\cal C}_{n-r-s-4}\oplus {\cal C}_2^2\oplus {\cal C}_r\oplus {\cal C}_2^2\oplus {\cal C}_s$ for some $r,s\in \N ^*$ with $r+s\leq n-5$;
\item if $|{\rm Con}(L)|<2^{n-3}$, then $|{\rm Con}(L)|\leq 7\cdot 2^{n-6}$;
\item $|{\rm Con}(L)|=7\cdot 2^{n-6}$ iff $n\geq 6$ and ${\rm Con}(L)\cong {\cal C}_2^{n-6}\times ({\cal C}_2^2\oplus {\cal C}_2^2)$ iff $n\geq 6$ and, for some $k\in [1,n-5]$, either $L\cong {\cal C}_{n-k-4}\oplus ({\cal C}_3\boxplus {\cal C}_5)\oplus {\cal C}_k$ or $L\cong {\cal C}_{n-k-4}\oplus ({\cal C}_4\boxplus {\cal C}_4)\oplus {\cal C}_k$.\end{itemize}\label{cglatfin}\end{theorem}

\begin{theorem} For any $n\in \N ^*$, any lattice $L$ with $|L|=n$ and any weak complementation $\cdot ^{\Delta }$ on $L$, we have:\begin{enumerate}
\item\label{cgwclfin1} $|{\rm Con}_{\WCL }(L,\cdot ^{\Delta })|\leq 2^{n-2}+1$;
\item\label{cgwclfin2} $|{\rm Con}_{\WCL }(L,\cdot ^{\Delta })|=2^{n-2}+1$ iff ${\rm Con}_{\WCL }(L,\cdot ^{\Delta })\cong {\cal C}_2^{n-2}\oplus {\cal C}_2$ iff $n\geq 2$ and $L\cong {\cal C}_n$;
\item\label{cgwclfin3} $|{\rm Con}_{\WCL }(L,\cdot ^{\Delta })|=2^{n-2}$ iff $n=4$ and ${\rm Con}_{\WCL }(L,\cdot ^{\Delta })\cong {\cal C}_2^2$ iff $L\cong {\cal C}_2^2$ and $\cdot ^{\Delta }$ is the Boolean complementation;
\item\label{cgwclfin4} $|{\rm Con}_{\WCL }(L,\cdot ^{\Delta })|<2^{n-2}$ iff $|{\rm Con}_{\WCL }(L,\cdot ^{\Delta })|\leq 2^{n-3}+1$;
\item\label{cgwclfin5} $|{\rm Con}_{\WCL }(L,\cdot ^{\Delta })|=2^{n-3}+1$ iff ${\rm Con}_{\WCL }(L,\cdot ^{\Delta })\cong {\cal C}_2^{n-3}\oplus {\cal C}_2$ iff $n\geq 5$ and $L\cong {\cal C}_{n-k-2}\oplus {\cal C}_2^2\oplus {\cal C}_k$ for some $k\in [2,n-3]$;
\item\label{cgwclfin6} $|{\rm Con}_{\WCL }(L,\cdot ^{\Delta })|=3\cdot 2^{n-5}$ iff $n=5$ and ${\rm Con}_{\WCL }(L,\cdot ^{\Delta })\cong {\cal C}_3$ or $n=6$ and ${\rm Con}_{\WCL }(L,\cdot ^{\Delta })\cong {\cal C}_2\times {\cal C}_3$ iff $L\cong {\cal N}_5$ or $L\cong {\cal C}_2\times {\cal C}_3$ and $\cdot ^{\Delta }$ is the direct product of the trivial weak complementations on the chains ${\cal C}_2$ and ${\cal C}_3$;
\item\label{cgwclfin7} if $L\ncong {\cal N}_5$ and $(L,\cdot ^{\Delta })\ncong _{\WCL }({\cal C}_2,\cdot ^{\Delta {\cal C}_2})\times ({\cal C}_3,\cdot ^{\Delta {\cal C}_3})$, then: $|{\rm Con}_{\WCL }(L,\cdot ^{\Delta })|\leq 2^{n-3}$ iff $|{\rm Con}_{\WCL }(L,\cdot ^{\Delta })|\leq 5\cdot 2^{n-6}+1$;
\item\label{cgwclfin8} $|{\rm Con}_{\WCL }(L,\cdot ^{\Delta })|=5\cdot 2^{n-6}+1$ iff ${\rm Con}_{\WCL }(L,\cdot ^{\Delta })\cong ({\cal C}_2^{n-6}\times ({\cal C}_2\oplus {\cal C}_2^2))\oplus {\cal C}_2$ iff $n\geq 6$ and $L\cong {\cal C}_{n-k-3}\oplus {\cal N}_5\oplus {\cal C}_k$ for some $k\in [2,n-4]$;
\item\label{cgwclfin9} $|{\rm Con}_{\WCL }(L,\cdot ^{\Delta })|=5\cdot 2^{n-6}$ iff $n=6$ and either ${\rm Con}_{\WCL }(L,\cdot ^{\Delta })\cong {\cal C}_2\oplus {\cal C}_2^2$ or ${\rm Con}_{\WCL }(L,\cdot ^{\Delta })\cong {\cal C}_2^2\oplus {\cal C}_2$ iff one of the following holds:

$\bullet \ $ $L\cong {\cal C}_2\times {\cal C}_3$ and $\cdot ^{\Delta }$ is neither trivial nor equal to the direct product of the weak complementations on the chains ${\cal C}_2$ and ${\cal C}_3$ (case in which ${\rm Con}_{\WCL }(L,\cdot ^{\Delta })\cong {\cal C}_2\oplus {\cal C}_2^2$);

$\bullet \ $  $L\cong {\cal C}_3\boxplus {\cal C}_5$ or $L\cong {\cal C}_4\boxplus {\cal C}_4$ and $\cdot ^{\Delta }$ is trivial (case in which ${\rm Con}_{\WCL }(L,\cdot ^{\Delta })\cong {\cal C}_2^2\oplus {\cal C}_2$);

\item\label{cgwclfin10} $|{\rm Con}_{\WCL }(L,\cdot ^{\Delta })|<5\cdot 2^{n-6}$ iff $|{\rm Con}_{\WCL }(L,\cdot ^{\Delta })|\leq 2^{n-4}+1$;
\item\label{cgwclfin11} $|{\rm Con}_{\WCL }(L,\cdot ^{\Delta })|=2^{n-4}+1$ iff $n\geq 5$ and ${\rm Con}_{\WCL }(L,\cdot ^{\Delta })\cong {\cal C}_2^{n-4}\oplus {\cal C}_2$ iff one of the following holds:

$\bullet \ $ $n\geq 5$, $L\cong {\cal C}_{n-r-s+3}\oplus ({\cal C}_r\boxplus {\cal C}_s)$ for some $r,s\in \N \setminus \{0,1,2\}$ such that $r+s\leq n+1$ and, if $r+s>6$ (that is if $L\ncong {\cal C}_{n-3}\oplus {\cal C}_2^2$), then $\cdot ^{\Delta }$ is trivial;

$\bullet \ $ $n\geq 7$ and $L\cong {\cal C}_{n-k-4}\oplus ({\cal C}_2\times {\cal C}_3)\oplus {\cal C}_k$ for some $k\in [2,n-5]$;

$\bullet \ $ $n\geq 8$ and $L\cong {\cal C}_{n-r-s-4}\oplus {\cal C}_2^2\oplus {\cal C}_r\oplus {\cal C}_2^2\oplus {\cal C}_s$ for some $r,s\in \N ^*$ such that $s>1$ and $r+s\leq n-5$.\end{enumerate}\label{cgwclfin}\end{theorem}

\begin{proof} For any bounded lattice $L$, any weak complementation $\cdot ^{\Delta }$ on $L$ and any $\alpha \in {\rm Con}_{\WCL }(L,\cdot ^{\Delta })$, we will denote by $(L/\alpha ,\cdot ^{\Delta [\alpha ]})=(L,\cdot ^{\Delta })/\alpha $. Of course, the trivial weak complementation $\cdot ^{\Delta L/\alpha }=\cdot ^{\Delta L[\alpha ]}$. 

For any bounded lattice $L$ with exactly two coatoms $a$ and $b$ such that $\cdot ^{\Delta a,b}$ is a weak complementation on $L$, we denote by $\cdot ^{\Delta _{{\rm CoAt}(L)}}=\cdot ^{\Delta a,b}$.

Recall from Section \ref{ordsum} and Theorem \ref{mainth} that, for any bounded lattices $K$, $L$, $M$ such that $|L|,|M|>3$:\begin{itemize}
\item if $1\notin {\rm Sji}(L)\cap {\rm Sji}(M)$, then $K\oplus (L\boxplus M)$ can only be endowed with the trivial weak complementation $\cdot ^{\Delta K\oplus (L\boxplus M)}$;
\item if $1\in {\rm Sji}(L)\cap {\rm Sji}(M)$, then the only nontrivial weak complementation on $K\oplus (L\boxplus M)$ is $\cdot ^{\Delta _{{\rm CoAt}(K\oplus (L\boxplus M))}}$.\end{itemize}

Thus, for any bounded lattice $K$, $\cdot ^{\Delta _{{\rm CoAt}(K\oplus {\cal C}_2^2)}}$ is the only nontrivial weak complementation on $K\oplus {\cal C}_2^2$, which restricts to the Boolean complementation on ${\cal C}_2^2$, and $\cdot ^{\Delta _{{\rm CoAt}(K\oplus {\cal N}_5)}}$ is the only nontrivial weak complementation on $K\oplus {\cal N}_5$.

Recall from Section \ref{ordsum} that any bounded lattice $L$ with the $1$ strictly join--irreducible can only be endowed with the trivial weak complementation $\cdot ^{\Delta L}$.

Recall, also, that, for any bounded lattice $L$, ${\rm Con}_{\WCL }(L,\cdot ^{\Delta L})={\rm Con}_1(L)\cup \{L^2\}\cong {\rm Con}_1(L)\oplus {\cal C}_2$.

Hence, for every $n\in \N ^*$, ${\cal C}_n$ can only be endowed with the trivial weak complementation $\cdot ^{\Delta {\cal C}_n}$ and ${\rm Con}_{\WCL }({\cal C}_n,\linebreak \cdot ^{\Delta {\cal C}_n})={\rm Con}_1({\cal C}_n)\cup \{{\cal C}_n^2\}$, thus ${\rm Con}_{\WCL }({\cal C}_1,\cdot ^{\Delta {\cal C}_1})\cong {\cal C}_1$ and, if $n\geq 2$, then ${\rm Con}_{\WCL }({\cal C}_n,\cdot ^{\Delta {\cal C}_n})\cong {\rm Con}({\cal C}_{n-1})\oplus {\cal C}_2\cong {\cal C}_2^{n-2}\oplus {\cal C}_2$, hence $|{\rm Con}_{\WCL }({\cal C}_1,\cdot ^{\Delta {\cal C}_1})|=1<2^{-1}+1$ and, if $n\geq 2$, then $|{\rm Con}_{\WCL }({\cal C}_n,\cdot ^{\Delta {\cal C}_n})|=2^{n-2}+1$.

Also, ${\cal C}_2^2\oplus {\cal C}_2$ has the $1$ join--irreducible, thus it can only be endowed with the trivial weak complementation, and ${\rm Con}_{\WCL }({\cal C}_2^2\oplus {\cal C}_2,\cdot ^{\Delta {\cal C}_2^2\oplus {\cal C}_2})={\rm Con}_1({\cal C}_2^2\oplus {\cal C}_2)\cup \{({\cal C}_2^2\oplus {\cal C}_2)^2\}\cong {\cal C}_2^2\oplus {\cal C}_2$, so $|{\rm Con}_{\WCL }({\cal C}_2^2\oplus {\cal C}_2,\cdot ^{\Delta {\cal C}_2^2\oplus {\cal C}_2})|=2^2+1$.

By the above, ${\cal M}_3\cong {\cal C}_3\boxplus {\cal C}_3\boxplus {\cal C}_3\cong {\cal C}_3\boxplus {\cal C}_2^2$ can only be endowed with $\cdot ^{\Delta {\cal M}_3}$, and ${\rm Con}_{\WCL }({\cal M}_3,\cdot ^{\Delta {\cal M}_3})={\rm Con}_1({\cal M}_3)\cup \{{\cal M}_3^2\}\cong {\cal C}_1\oplus {\cal C}_2\cong {\cal C}_2$, so $|{\rm Con}_{\WCL }({\cal M}_3,\cdot ^{\Delta {\cal M}_3})|=2^1$.

In what follows, as above, we will use the remarks in Section \ref{ordsum} and Theorem \ref{mainth} to determine the weak complementations on the following lattices and Propositions \ref{cgordsum} and \ref{cghsumnontriv} to determine the congruences of the weakly complemented lattices formed with those weak complementations. The cases not covered by these results need to be verified directly.

${\rm Con}_{\WCL }({\cal C}_2^2,\cdot ^{\Delta {\cal C}_2^2})={\rm Con}_1({\cal C}_2^2)\cup \{({\cal C}_2^2)^2\}\cong {\cal C}_1\oplus {\cal C}_2\cong {\cal C}_2$, so $|{\rm Con}_{\WCL }({\cal C}_2^2,\cdot ^{\Delta {\cal C}_2^2})|=2^1=2^0+1$, while ${\rm Con}_{\WCL }({\cal C}_2^2,\cdot ^{\Delta _{{\rm CoAt}({\cal C}_2^2)}})\cong {\cal C}_2^2$ since $({\cal C}_2^2,\cdot ^{\Delta _{{\rm CoAt}({\cal C}_2^2})})$ is a Boolean algebra, so $|{\rm Con}_{\WCL }({\cal C}_2^2,\cdot ^{\Delta _{{\rm CoAt}({\cal C}_2^2})})|=2^2$.

${\rm Con}_{\WCL }({\cal C}_2\oplus {\cal C}_2^2,\cdot ^{\Delta {\cal C}_2\oplus {\cal C}_2^2})={\rm Con}_1({\cal C}_2\oplus {\cal C}_2^2)\cup \{({\cal C}_2\oplus {\cal C}_2^2)^2\}\cong {\cal C}_2\oplus {\cal C}_2\cong {\cal C}_3$, so $|{\rm Con}_{\WCL }({\cal C}_2\oplus {\cal C}_2^2,\cdot ^{\Delta {\cal C}_2\oplus {\cal C}_2^2})|=2^1+1$, while ${\rm Con}_{\WCL }({\cal C}_2\oplus {\cal C}_2^2,\cdot ^{\Delta _{{\rm CoAt}({\cal C}_2\oplus {\cal C}_2^2)}})\cong ({\rm Con}({\cal C}_2)\times {\rm Con}_{\WCL 1}({\cal C}_2^2,\cdot ^{\Delta _{{\rm CoAt}({\cal C}_2^2)}})\oplus {\cal C}_2\cong ({\cal C}_2\times {\cal C}_1)\oplus {\cal C}_2\cong {\cal C}_2\oplus {\cal C}_2\cong {\cal C}_3$, so $|{\rm Con}_{\WCL }({\cal C}_2\oplus {\cal C}_2^2,\cdot ^{\Delta _{{\rm CoAt}({\cal C}_2\oplus {\cal C}_2^2)}})|=2^1+1$.

Recall from Example \ref{exhsums} that ${\rm Con}_{\WCL }({\cal N}_5,\cdot ^{\Delta {\cal N}_5})={\rm Con}_{\WCL }({\cal N}_5,\cdot ^{\Delta _{{\rm CoAt}({\cal N}_5)}})\cong {\cal C}_3$, thus $|{\rm Con}_{\WCL }({\cal N}_5,\cdot ^{\Delta {\cal N}_5})|=|{\rm Con}_{\WCL }({\cal N}_5,\cdot ^{\Delta _{{\rm CoAt}({\cal N}_5)}})|=3$, while ${\rm Con}_{\WCL }({\cal C}_3\boxplus {\cal C}_5,\cdot ^{\Delta {\cal C}_3\boxplus {\cal C}_5})={\rm Con}_{\WCL }({\cal C}_3\boxplus {\cal C}_5,\cdot ^{\Delta _{{\rm CoAt}({\cal C}_3\boxplus {\cal C}_5)}})\cong {\cal C}_2^2\oplus {\cal C}_2\cong {\rm Con}_{\WCL }({\cal C}_4\boxplus {\cal C}_4,\cdot ^{\Delta {\cal C}_4\boxplus {\cal C}_4})={\rm Con}_{\WCL }({\cal C}_4\boxplus {\cal C}_4,\cdot ^{\Delta _{{\rm CoAt}({\cal C}_4\boxplus {\cal C}_4)}})$, thus $|{\rm Con}_{\WCL }({\cal C}_3\boxplus {\cal C}_5,\cdot ^{\Delta {\cal C}_3\boxplus {\cal C}_5})|=|{\rm Con}_{\WCL }({\cal C}_3\boxplus {\cal C}_5,\cdot ^{\Delta _{{\rm CoAt}({\cal C}_3\boxplus {\cal C}_5)}})|=|{\rm Con}_{\WCL }({\cal C}_4\boxplus {\cal C}_4,\cdot ^{\Delta {\cal C}_4\boxplus {\cal C}_4})|=|{\rm Con}_{\WCL }({\cal C}_4\boxplus {\cal C}_4,\cdot ^{\Delta _{{\rm CoAt}({\cal C}_4\boxplus {\cal C}_4)}})|=5$.

See the weak complementations on ${\cal C}_2\times {\cal C}_3$ and their corresponding congruence lattices in Example \ref{C2xC3}.

By Theorem \ref{cglatfin} and some more quick verifications:\begin{itemize}
\item the six--element lattices with at least $2^2+1=5$ lattice congruences are ${\cal C}_6$, ${\cal C}_2\times {\cal C}_3$, ${\cal C}_2^2\oplus {\cal C}_3$, ${\cal C}_2\oplus {\cal C}_2^2\oplus {\cal C}_2$, ${\cal C}_3\oplus {\cal C}_2^2$, ${\cal N}_5\oplus {\cal C}_2$ and ${\cal C}_2\oplus {\cal N}_5$, ${\cal C}_3\boxplus {\cal C}_5$ and ${\cal C}_4\boxplus {\cal C}_4$;
\item the seven--element lattices with at least $2^3+1=9$ lattice congruences are ${\cal C}_7$, $({\cal C}_2\times {\cal C}_3)\oplus {\cal C}_2$, ${\cal C}_2\oplus ({\cal C}_2\times {\cal C}_3)$, ${\cal C}_2^2\oplus {\cal C}_4$, ${\cal C}_2\oplus {\cal C}_2^2\oplus {\cal C}_3$, ${\cal C}_3\oplus {\cal C}_2^2\oplus {\cal C}_2$, ${\cal C}_4\oplus {\cal C}_2^2$, ${\cal C}_2^2\oplus {\cal C}_2^2$, ${\cal N}_5\oplus {\cal C}_3$, ${\cal C}_2\oplus {\cal N}_5\oplus {\cal C}_2$, ${\cal C}_3\oplus {\cal N}_5$, $({\cal C}_3\boxplus {\cal C}_5)\oplus {\cal C}_2$, ${\cal C}_2\oplus ({\cal C}_3\boxplus {\cal C}_5)$, $({\cal C}_4\boxplus {\cal C}_4)\oplus {\cal C}_2$, ${\cal C}_2\oplus ({\cal C}_4\boxplus {\cal C}_4)$, ${\cal C}_3\boxplus {\cal C}_6$ and ${\cal C}_4\boxplus {\cal C}_5$;
\item the eight--element lattices with at least $2^4+1=17$ lattice congruences are ${\cal C}_8$, ${\cal C}_2\times {\cal C}_4$, $({\cal C}_2\times {\cal C}_3)\oplus {\cal C}_3$, ${\cal C}_2\oplus ({\cal C}_2\times {\cal C}_3)\oplus {\cal C}_2$, ${\cal C}_3\oplus ({\cal C}_2\times {\cal C}_3)$, ${\cal C}_2^2\oplus {\cal C}_5$, ${\cal C}_2\oplus {\cal C}_2^2\oplus {\cal C}_4$, ${\cal C}_3\oplus {\cal C}_2^2\oplus {\cal C}_3$, ${\cal C}_4\oplus {\cal C}_2^2\oplus {\cal C}_2$, ${\cal C}_5\oplus {\cal C}_2^2$, ${\cal C}_2^2\oplus {\cal C}_2^2\oplus {\cal C}_2$, ${\cal C}_2^2\oplus {\cal C}_2\oplus {\cal C}_2^2$, ${\cal C}_2\oplus {\cal C}_2^2\oplus {\cal C}_2^2$, ${\cal N}_5\oplus {\cal C}_4$, ${\cal C}_2\oplus {\cal N}_5\oplus {\cal C}_3$, and ${\cal C}_3\oplus {\cal N}_5\oplus {\cal C}_2$, ${\cal C}_4\oplus {\cal N}_5$, ${\cal C}_2^2\oplus {\cal N}_5$, ${\cal N}_5\oplus {\cal C}_2^2$, $({\cal C}_3\boxplus {\cal C}_5)\oplus {\cal C}_3$, ${\cal C}_2\oplus ({\cal C}_3\boxplus {\cal C}_5)\oplus {\cal C}_2$, ${\cal C}_3\oplus ({\cal C}_3\boxplus {\cal C}_5)$, $({\cal C}_4\boxplus {\cal C}_4)\oplus {\cal C}_3$, ${\cal C}_2\oplus ({\cal C}_4\boxplus {\cal C}_4)\oplus {\cal C}_2$, ${\cal C}_3\oplus ({\cal C}_4\boxplus {\cal C}_4)$, $({\cal C}_3\boxplus {\cal C}_6)\oplus {\cal C}_2$, ${\cal C}_2\oplus ({\cal C}_3\boxplus {\cal C}_6)$, $({\cal C}_4\boxplus {\cal C}_5)\oplus {\cal C}_2$ and ${\cal C}_2\oplus ({\cal C}_4\boxplus {\cal C}_5)$, ${\cal C}_3\boxplus {\cal C}_7$, ${\cal C}_4\boxplus {\cal C}_6$ and ${\cal C}_5\boxplus {\cal C}_5$;\end{itemize} 

\noindent and, with calculations similar to the above, one can easily check that all these lattices satisfy the statements in the enunciation. Note that, out of all the weakly complemented lattices $(L,\cdot ^{\Delta })$ above, the only one that has exactly $2^{|L|-2}$ congruences is $({\cal C}_2^2,\cdot ^{\Delta _{{\rm CoAt}({\cal C}_2^2)}})$, the only ones that have exactly $3\cdot 2^{|L|-5}$ congruences are ${\cal N}_5$ endowed with any of its weak complementations and ${\cal C}_2\times {\cal C}_3$ endowed with the direct product of $\cdot ^{\Delta {\cal C}_2}$ with $\cdot ^{\Delta {\cal C}_3}$ and the only ones that have exactly $5\cdot 2^{|L|-6}$ congruences are ${\cal C}_2\times {\cal C}_3$ endowed with one of the weak complementations $\cdot ^{\Delta a,b}$ and $\cdot ^{\Delta a,b,b}$ in Example \ref{C2xC3} (which has the congruence lattice isomorphic to ${\cal C}_2\oplus {\cal C}_2^2$) and ${\cal C}_3\boxplus {\cal C}_5$ and ${\cal C}_4\boxplus {\cal C}_4$ endowed with their trivial weak complementations (having the congruence lattices isomorphic to ${\cal C}_2^2\oplus {\cal C}_2$).

Now let $n\in \N $ such that $n\geq 9$ and assume that the statements in the enunciation hold for any lattice whose cardinality is at most $n-1$.

Let $(L,\cdot ^{\Delta })$ be a weakly complemented lattice with $|L|=n$ and $\alpha \in {\rm At}({\rm Con}_{\WCL }(L,\cdot ^{\Delta }))\subseteq {\rm Con}_{\WCL }(L,\cdot ^{\Delta })\setminus \{\mbox{\bf =}_L\}$, so that $|L/\alpha |\leq n-1$ and, according to Remark \ref{onWCLcg}, $\alpha =Cg_{\WCL ,L}(a,b)$ for some $a,b\in L$ with $a\prec b$.

\noindent (\ref{cgwclfin1}) By the induction hypothesis, $|{\rm Con}_{\WCL }(L/\alpha ,\cdot ^{\Delta [\alpha ]})|\leq 2^{n-3}+1$ and thus $|{\rm Con}_{\WCL }(L,\cdot ^{\Delta })|\leq 2^{n-2}+2$ according to Lemma \ref{atcg}.

Assume by absurdum that $|{\rm Con}_{\WCL }(L)|=2^{n-2}+2$. Then, by Lemma \ref{atcg}, it follows that $|{\rm Con}_{\WCL }(L/\alpha ,\cdot ^{\Delta [\alpha ]})|\linebreak \geq 2^{n-3}+1$, hence $|{\rm Con}_{\WCL }(L/\alpha,\cdot ^{\Delta [\alpha ]})|=2^{n-3}+1$ and thus $L/\alpha \cong {\cal C}_{n-1}$ by the induction hypothesis, in particular $|L/\alpha |=n-1=|L|-1$, thus, by Lemma \ref{pWCLcg}, (\ref{pWCLcg1}), $a\in {\rm Mi}(L)$, $b\in {\rm Ji}(L)$, $a^{\Delta }=b^{\Delta }$ and $\alpha =Cg_L(a,b)=\varepsilon _L(a,b)$.

Since $L/\alpha $ is a chain, it follows that, for any $x,y\in L\setminus \{a,b\}$, $x/\alpha =\{x\}$ is comparable to $y/\alpha =\{y\}$ and to $a/\alpha =b/\alpha =\{a,b\}$, hence $x$ is comparable to $y$ and to at least one of $a$ and $b$. Assume by absurdum that $x||a$ or $x||b$. W.l.g. we may assume that $x||b$. Then, by the above, $x<b$, $x>a\wedge x<a$ and $a<a\vee x\leq b$, thus $a\vee x=b$, hence $\{a\wedge x,a,x,b\}$ is a sublattice of $L$ isomorphic to ${\cal C}_2^2$, therefore $(a\wedge x,x)\in L/\alpha $, which contradicts the fact $\alpha $ only collapses $a$ with $b$. Hence $x$ is comparable to each of $a$ and $b$.

Therefore $L\cong {\cal C}_n$, thus $\cdot ^{\Delta }=\cdot ^{\Delta L}$ and ${\rm Con}_{\WCL }(L,\cdot ^{\Delta })={\rm Con}_{\WCL }(L,\cdot ^{\Delta L})\cong {\rm Con}_{\WCL }({\cal C}_n,\cdot ^{\Delta {\cal C}_n})\cong {\cal C}_2^{n-2}\oplus {\cal C}_2$, thus $|{\rm Con}_{\WCL }(L,\cdot ^{\Delta L})|=2^{n-2}+1$, contradicting the above.

Hence $|{\rm Con}_{\WCL }(L,\cdot ^{\Delta })|\leq 2^{n-2}+1$.

In what follows, one can reason as above on the lattice structure of $L$ based on that of $L/\alpha $ and the form of $\alpha $; we will skip such details in the rest of this proof.

\noindent (\ref{cgwclfin2}) Assume that $|{\rm Con}_{\WCL }(L,\cdot ^{\Delta })|=2^{n-2}+1$. By Theorem \ref{cglatfin}, if $L$ were not a chain, then $|{\rm Con}(L)|\leq 2^{n-2}<|{\rm Con}_{\WCL }(L,\cdot ^{\Delta })|$, which would contradict the fact that ${\rm Con}_{\WCL }(L,\cdot ^{\Delta })\subseteq {\rm Con}(L)$.

Hence, by the proof of (\ref{cgwclfin1}), $|{\rm Con}_{\WCL }(L,\cdot ^{\Delta })|=2^{n-2}+1$ exactly when $L\cong {\cal C}_n$.

\noindent (\ref{cgwclfin3}) Assume by absurdum that $|{\rm Con}_{\WCL }(L,\cdot ^{\Delta })|=2^{n-2}<2^{n-2}+1$, so that $L\ncong {\cal C}_n$ by (\ref{cgwclfin2}), thus $|{\rm Con}(L)|\leq 2^{n-2}$ by Theorem \ref{cglatfin}, hence ${\rm Con}(L)={\rm Con}_{\WCL }(L,\cdot ^{\Delta })$, in particular $|{\rm Con}(L)|=|{\rm Con}_{\WCL }(L,\cdot ^{\Delta })|=2^{n-2}$, therefore, again by Theorem \ref{cglatfin}, $L\cong {\cal C}_{n-k-2}\oplus {\cal C}_2^2\oplus {\cal C}_k$ for some $k\in [1,n-3]$.

If $\cdot ^{\Delta }=\cdot ^{\Delta L}$, then ${\rm Con}_{\WCL }(L,\cdot ^{\Delta })={\rm Con}_{\WCL }(L,\cdot ^{\Delta L})={\rm Con}_1(L)\cup \{L^2\}\neq {\rm Con}(L)$, which contradicts the above. Hence $\cdot ^{\Delta }\neq \cdot ^{\Delta L}$, thus $k=1$, so $L\cong {\cal C}_{n-3}\oplus {\cal C}_2^2$, and $\cdot ^{\Delta }=\cdot ^{\Delta _{{\rm CoAt}({\cal C}_{n-3}\oplus {\cal C}_2^2)}}$. Then, since $n\geq 9>4$ and thus ${\cal C}_{n-3}$ is nontrivial, ${\rm Con}_{\WCL }(L,\cdot ^{\Delta })={\rm Con}_{\WCL }(L,\cdot ^{\Delta _{{\rm CoAt}({\cal C}_{n-3}\oplus {\cal C}_2^2)}})\cong ({\rm Con}({\cal C}_{n-3})\times {\rm Con}_{\WCL 1}({\cal C}_2^2,\cdot ^{\Delta _{{\rm CoAt}({\cal C}_2^2)}}))\oplus {\cal C}_2\cong ({\cal C}_2^{n-4}\times {\cal C}_1)\oplus {\cal C}_2\cong {\cal C}_2^{n-4}\oplus {\cal C}_2$, so $|{\rm Con}_{\WCL }(L,\cdot ^{\Delta })|=2^{n-4}+1\neq 2^{n-2}$ and we have another contradiction.

Therefore $|{\rm Con}_{\WCL }(L,\cdot ^{\Delta })|\neq 2^{n-2}$.

\noindent (\ref{cgwclfin4}) Assume that $|{\rm Con}_{\WCL }(L,\cdot ^{\Delta })|<2^{n-2}+1$, so that $|{\rm Con}_{\WCL }(L,\cdot ^{\Delta })|<2^{n-2}$ by (\ref{cgwclfin3}) and $L\ncong {\cal C}_n$ by (\ref{cgwclfin2}).

Assume by absurdum that $|{\rm Con}_{\WCL }(L,\cdot ^{\Delta })|\geq 2^{n-3}+2$, so that, by Lemma \ref{atcg}, $|{\rm Con}_{\WCL }(L/\alpha ,\cdot ^{\Delta [\alpha ]})|\geq 2^{n-4}+1>4=2^{4-2}$ since $n\geq 9$, thus, by the induction hypothesis, we are in one of the following cases:

\noindent $\blacksquare \ $ $|{\rm Con}_{\WCL }(L/\alpha ,\cdot ^{\Delta [\alpha ]})|=2^{n-3}+1$, but then, as in (\ref{cgwclfin1}), we would obtain that $L\cong {\cal C}_n$, contradicting the above;

\noindent $\blacksquare \ $ $|{\rm Con}_{\WCL }(L/\alpha ,\cdot ^{\Delta [\alpha ]})|=2^{n-4}+1$, hence we are in one of the following subcases:

$\bullet \ $ $|L/\alpha |=n-1$, so that $\alpha $ is as in Lemma \ref{pWCLcg}, (\ref{pWCLcg1}), and $L/\alpha \cong {\cal C}_{n-k-3}\oplus {\cal C}_2^2\oplus {\cal C}_k$ for some $k\in [2,n-4]$, hence $L\cong {\cal C}_{n-k-2}\oplus {\cal C}_2^2\oplus {\cal C}_k$ or $L\cong {\cal C}_{n-k-3}\oplus {\cal C}_2^2\oplus {\cal C}_{k+1}$ or $L\cong {\cal C}_{n-k-3}\oplus {\cal N}_5\oplus {\cal C}_k$, so $\cdot ^{\Delta L}$ is the only weak complementation on $L$ and, in the first two of these cases, ${\rm Con}_{\WCL }(L,\cdot ^{\Delta L})\cong {\cal C}_2^{n-3}\oplus {\cal C}_2$, thus $|{\rm Con}_{\WCL }(L,\cdot ^{\Delta L})|=2^{n-3}+1$, while, in the latter case, ${\rm Con}_{\WCL }(L,\cdot ^{\Delta L})\cong {\rm Con}({\cal C}_{n-k-3}\oplus {\cal N}_5\oplus {\cal C}_{k-1})\oplus {\cal C}_2\cong ({\cal C}_2^{n-6}\times ({\cal C}_2\oplus {\cal C}_2^2))\oplus {\cal C}_2$, thus $|{\rm Con}_{\WCL }(L,\cdot ^{\Delta L})|=5\cdot 2^{n-6}+1$;

$\bullet \ $ $|L/\alpha |=n-2>4$ and $L/\alpha \cong {\cal C}_{n-2}$, so that $\alpha $ is as in Lemma \ref{pWCLcg}, (\ref{pWCLcg2}), case \textcircled{$\beta\! _1$}, thus $L\cong {\cal C}_{n-k-2}\oplus {\cal C}_2^2\oplus {\cal C}_k$ for some $k\in [1,n-3]$ and $\cdot ^{\Delta }=\cdot ^{\Delta (L)}$, thus, as above, $|{\rm Con}_{\WCL }(L,\cdot ^{\Delta (L)})|=2^{n-3}+1$ if $k\geq 2$, while, if $k=1$, so that $L\cong {\cal C}_{n-3}\oplus {\cal C}_2^2$, then ${\rm Con}_{\WCL }(L,\cdot ^{\Delta (L)})\cong ({\cal C}_2^{n-4}\times {\cal C}_1)\oplus {\cal C}_2\cong {\cal C}_2^{n-4}\oplus {\cal C}_2$, so that $|{\rm Con}_{\WCL }(L,\cdot ^{\Delta (L)})|=2^{n-4}+1$.

Every case above contradicts the assumption that $|{\rm Con}_{\WCL }(L,\cdot ^{\Delta })|=2^{n-3}+2$; hence $|{\rm Con}_{\WCL }(L,\cdot ^{\Delta })|\leq 2^{n-3}+1$.

\noindent (\ref{cgwclfin5}) Assume that $|{\rm Con}_{\WCL }(L,\cdot ^{\Delta })|=2^{n-3}+1$, so that $|{\rm Con}_{\WCL }(L/\alpha ,\cdot ^{\Delta [\alpha ]})|\geq 2^{n-4}+1/2$ by Lemma \ref{atcg}, thus $|{\rm Con}_{\WCL }(L/\alpha ,\cdot ^{\Delta [\alpha ]})|\geq 2^{n-4}+1$ and hence, by the proof of (\ref{cgwclfin4}), $L\cong {\cal C}_{n-k-2}\oplus {\cal C}_2^2\oplus {\cal C}_k$ for some $k\in [1,n-3]$ or $L\cong {\cal C}_{n-k-3}\oplus {\cal N}_5\oplus {\cal C}_k$ for some $k\in [2,n-4]$, out of which only the first form for $k\geq 2$, with its unique weak complementation $\cdot ^{\Delta L}$, has exactly $2^{n-3}+1$ congruences. 

\noindent (\ref{cgwclfin6}),(\ref{cgwclfin7}) Assume that $|{\rm Con}_{\WCL }(L,\cdot ^{\Delta })|\leq 2^{n-3}$, so that, by (\ref{cgwclfin1})--(\ref{cgwclfin5}), $L$ is not isomorphic to either of the lattices ${\cal C}_n$ and ${\cal C}_{n-k-2}\oplus {\cal C}_2^2\oplus {\cal C}_k$ for any $k\in [2,n-4]$.

Assume by absurdum that $|{\rm Con}_{\WCL }(L,\cdot ^{\Delta })|\geq 5\cdot 2^{n-6}+2$, so that $|{\rm Con}_{\WCL }(L/\alpha ,\cdot ^{\Delta [\alpha ]})|\geq 5\cdot 2^{n-7}+1$ by Lemma \ref{atcg}, so that, by the induction hypothesis, we are in one of the following cases:

\noindent $\blacksquare \ $ $|{\rm Con}_{\WCL }(L/\alpha ,\cdot ^{\Delta [\alpha ]})|=2^{n-3}+1$; but then, as in (\ref{cgwclfin1}), it would follow that $L\cong {\cal C}_n$, contradicting the above;

\noindent $\blacksquare \ $ $|{\rm Con}_{\WCL }(L/\alpha ,\cdot ^{\Delta [\alpha ]})|=2^{n-4}+1$; but then, as in (\ref{cgwclfin4}), it would follow that $L\cong {\cal C}_{n-k-2}\oplus {\cal C}_2^2\oplus {\cal C}_k$ for some $k\in [1,n-3]$ or $L\cong {\cal C}_{n-k-3}\oplus {\cal N}_5\oplus {\cal C}_k$ for some $k\in [2,n-4]$, hence, by the cases eliminated at the beginning of this proof of (\ref{cgwclfin9}), either $L\cong {\cal C}_{n-3}\oplus {\cal C}_2^2$, which can be endowed with $\cdot ^{\Delta L}$ and with $\cdot ^{\Delta _{{\rm CoAt}(L)}}$, or $L\cong {\cal C}_{n-k-3}\oplus {\cal N}_5\oplus {\cal C}_k$ for some $k\in [2,n-4]$, which can only be endowed with $\cdot ^{\Delta L}$, and neither of these weakly complemented lattices has $5\cdot 2^{n-6}+2$ or more congruences; indeed:

\noindent ${\rm Con}_{\WCL }({\cal C}_{n-3}\oplus {\cal C}_2^2,\cdot ^{\Delta {\cal C}_{n-3}\oplus {\cal C}_2^2})\cong {\rm Con}_{\WCL }({\cal C}_{n-3}\oplus {\cal C}_2^2,\cdot ^{\Delta _{{\rm CoAt}({\cal C}_{n-3}\oplus {\cal C}_2^2)}})\cong ({\rm Con}({\cal C}_{n-3})\times {\cal C}_1)\oplus {\cal C}_2\cong {\cal C}_2^{n-4}\oplus {\cal C}_2$, hence $|{\rm Con}_{\WCL }({\cal C}_{n-3}\oplus {\cal C}_2^2,\cdot ^{\Delta {\cal C}_{n-3}\oplus {\cal C}_2^2})|=|{\rm Con}_{\WCL }({\cal C}_{n-3}\oplus {\cal C}_2^2,\cdot ^{\Delta _{{\rm CoAt}({\cal C}_{n-3}\oplus {\cal C}_2^2)}})|=2^{n-4}+1$;

\noindent ${\rm Con}_{\WCL }({\cal C}_{n-k-3}\oplus {\cal N}_5\oplus {\cal C}_k,\cdot ^{\Delta {\cal C}_{n-k-3}\oplus {\cal N}_5\oplus {\cal C}_k})\cong {\rm Con}({\cal C}_{n-k-3}\oplus {\cal N}_5\oplus {\cal C}_{k-1})\oplus {\cal C}_2\cong ({\cal C}_2^{n-6}\times ({\cal C}_2\oplus {\cal C}_2^2))\oplus {\cal C}_2$, thus $|{\rm Con}_{\WCL }({\cal C}_{n-k-3}\oplus {\cal N}_5\oplus {\cal C}_k,\cdot ^{\Delta {\cal C}_{n-k-3}\oplus {\cal N}_5\oplus {\cal C}_k})|=5\cdot 2^{n-6}+1$;

\noindent $\blacksquare \ $ $|{\rm Con}_{\WCL }(L/\alpha ,\cdot ^{\Delta [\alpha ]})|=5\cdot 2^{n-7}+1$; then, by the induction hypothesis, it follows that $|L/\alpha |=n-1$ and $L\cong {\cal C}_{n-k-4}\oplus {\cal N}_5\oplus {\cal C}_k$ for some $k\in [2,n-5]$, so that, by Lemma \ref{pWCLcg}, (\ref{pWCLcg1}), $L\cong {\cal C}_{n-k-4}\oplus {\cal N}_5\oplus {\cal C}_{k+1}$ or $L\cong {\cal C}_{n-k-3}\oplus {\cal N}_5\oplus {\cal C}_k$ or $L\cong {\cal C}_{n-k-4}\oplus ({\cal C}_3\boxplus {\cal C}_5)\oplus {\cal C}_k$ or $L\cong {\cal C}_{n-k-4}\oplus ({\cal C}_4\boxplus {\cal C}_4)\oplus {\cal C}_k$, each of which can only be endowed with the trivial weak complementation $\cdot ^{\Delta L}$, w.r.t. which neither has $5\cdot 2^{n-6}+2$ or more congruences; indeed:

\noindent ${\rm Con}_{\WCL }({\cal C}_{n-k-4}\oplus {\cal N}_5\oplus {\cal C}_{k+1},\cdot ^{\Delta {\cal C}_{n-k-4}\oplus {\cal N}_5\oplus {\cal C}_{k+1}})\cong {\rm Con}_{\WCL }({\cal C}_{n-k-3}\oplus {\cal N}_5\oplus {\cal C}_k,\cdot ^{\Delta {\cal C}_{n-k-3}\oplus {\cal N}_5\oplus {\cal C}_k})\cong ({\cal C}_2^{n-6}\times ({\cal C}_2\oplus {\cal C}_2^2))\oplus {\cal C}_2$, thus $|{\rm Con}_{\WCL }({\cal C}_{n-k-4}\oplus {\cal N}_5\oplus {\cal C}_{k+1},\cdot ^{\Delta {\cal C}_{n-k-4}\oplus {\cal N}_5\oplus {\cal C}_{k+1}})|=|{\rm Con}_{\WCL }({\cal C}_{n-k-3}\oplus {\cal N}_5\oplus {\cal C}_k,\cdot ^{\Delta {\cal C}_{n-k-3}\oplus {\cal N}_5\oplus {\cal C}_k})|=5\cdot 2^{n-6}+1$;

\noindent it is easy to check that ${\rm Con}({\cal C}_3\boxplus {\cal C}_5)\cong {\rm Con}({\cal C}_4\boxplus {\cal C}_4)\cong {\cal C}_2^2\oplus {\cal C}_2^2$ \cite{eucfifin}, so that: ${\rm Con}_{\WCL }({\cal C}_{n-k-4}\oplus ({\cal C}_3\boxplus {\cal C}_5)\oplus {\cal C}_k,\cdot ^{\Delta {\cal C}_{n-k-4}\oplus ({\cal C}_3\boxplus {\cal C}_5)\oplus {\cal C}_k})\cong {\rm Con}_{\WCL }({\cal C}_{n-k-4}\oplus ({\cal C}_4\boxplus {\cal C}_4)\oplus {\cal C}_k,\cdot ^{\Delta {\cal C}_{n-k-4}\oplus ({\cal C}_4\boxplus {\cal C}_4)\oplus {\cal C}_k})\cong ({\cal C}_2^{n-7}\times ({\cal C}_2^2\oplus {\cal C}_2^2))\oplus {\cal C}_2$, thus $|{\rm Con}_{\WCL }({\cal C}_{n-k-4}\oplus ({\cal C}_3\boxplus {\cal C}_5)\oplus {\cal C}_k,\cdot ^{\Delta {\cal C}_{n-k-4}\oplus ({\cal C}_3\boxplus {\cal C}_5)\oplus {\cal C}_k})|=|{\rm Con}_{\WCL }({\cal C}_{n-k-4}\oplus ({\cal C}_4\boxplus {\cal C}_4)\oplus {\cal C}_k,\cdot ^{\Delta {\cal C}_{n-k-4}\oplus ({\cal C}_4\boxplus {\cal C}_4)\oplus {\cal C}_k})|=7\cdot 2^{n-7}+1$.

Therefore $|{\rm Con}_{\WCL }(L,\cdot ^{\Delta })|\leq 5\cdot 2^{n-6}+1$.

\noindent (\ref{cgwclfin8}) Assume that $|{\rm Con}_{\WCL }(L,\cdot ^{\Delta })|=5\cdot 2^{n-6}+1$, so that $|{\rm Con}_{\WCL }(L/\alpha ,\cdot ^{\Delta [\alpha ]})|\geq 5\cdot 2^{n-7}+1/2$ by Lemma \ref{atcg}, and thus $|{\rm Con}_{\WCL }(L/\alpha ,\cdot ^{\Delta [\alpha ]})|\geq 5\cdot 2^{n-7}+1$, hence, as in the proof of (\ref{cgwclfin7}), it follows that $L\cong {\cal C}_{n-k-3}\oplus {\cal N}_5\oplus {\cal C}_k$ for some $k\in [2,n-4]$.

\noindent (\ref{cgwclfin9}),(\ref{cgwclfin10}) Assume that $|{\rm Con}_{\WCL }(L,\cdot ^{\Delta })|<5\cdot 2^{n-6}$ and assume by absurdum that $|{\rm Con}_{\WCL }(L,\cdot ^{\Delta })|\geq 2^{n-4}+2$, so that $|{\rm Con}_{\WCL }(L/\alpha ,\cdot ^{\Delta [\alpha ]})|\geq 2^{n-5}+1$ by Lemma \ref{atcg}.

Then, by the proof of (\ref{cgwclfin8}), we can not have $|{\rm Con}_{\WCL }(L/\alpha ,\cdot ^{\Delta [\alpha ]})|\in \{2^{n-3}+1,2^{n-4}+1,5\cdot 2^{n-7}+1\}$, hence, by the induction hypothesis, it follows that $|{\rm Con}_{\WCL }(L/\alpha ,\cdot ^{\Delta [\alpha ]})|=2^{n-5}+1$.

Let us note that, in the case when $|{\rm Con}_{\WCL }(L/\alpha ,\cdot ^{\Delta [\alpha ]})|=2^{n-4}+1$, one of the possible structures of $L$ is $L\cong {\cal C}_{n-3}\oplus {\cal C}_2^2$, endowed with any of its two weak complementations, which is the only situation where we have $|{\rm Con}_{\WCL }(L/\alpha ,\cdot ^{\Delta [\alpha ]})|\in \{2^{n-3}+1,2^{n-4}+1,5\cdot 2^{n-7}+1\}$ and $|{\rm Con}_{\WCL }(L,\cdot ^{\Delta })|=2^{n-4}+1$.

Now we consider the case $|{\rm Con}_{\WCL }(L/\alpha ,\cdot ^{\Delta [\alpha ]})|=2^{n-5}+1$, in which, by the induction hypothesis, we can be in one of these situations:

\noindent $\blacksquare \ $ $|L/\alpha |=n-1$, so that, by the induction hypothesis, Lemma \ref{pWCLcg}, (\ref{pWCLcg1}), the fact that, for any bounded lattice $L$, if $b=1\succ a\in L$, so that $a^{\Delta L}=1$ and $b^{\Delta L}=1^{\Delta L}=0$, then $Cg_{\WCL ,L}(a,b)=L^2$ in $(L,\cdot ^{\Delta L})$ and the calculations in Example \ref{exhsums}, one of the following holds:

$\bullet \ $ $L/\alpha \cong {\cal C}_{n-r-s+2}\oplus ({\cal C}_r\boxplus {\cal C}_s)$ for some $r,s\in \N \setminus \{0,1,2\}$ such that $r+s\leq n$, so that $L\cong {\cal C}_{n-r-s+3}\oplus ({\cal C}_r\boxplus {\cal C}_s)$ or $L\cong {\cal C}_{n-r-s+2}\oplus ({\cal C}_{r+1}\boxplus {\cal C}_s)$ or $L\cong {\cal C}_{n-r-s+2}\oplus ({\cal C}_r\boxplus {\cal C}_{s+1})$; each of these three bounded lattices has exactly $2^{n-4}+1$ congruences w.r.t. its trivial weak complementation, and, if it not isomorphic to ${\cal C}_{n-3}\oplus {\cal C}_2^2$, then it has strictly less than $2^{n-4}+1$ congruences w.r.t. its nontrivial weak complementation $\cdot ^{\Delta _{{\rm CoAt}(L)}}$;

$\bullet \ $ $L/\alpha \cong {\cal C}_{n-k-5}\oplus ({\cal C}_2\times {\cal C}_3)\oplus {\cal C}_k$ for some $k\in [2,n-6]$, so that $L\cong {\cal C}_{n-k-4}\oplus ({\cal C}_2\times {\cal C}_3)\oplus {\cal C}_k$ or $L\cong {\cal C}_{n-k-5}\oplus ({\cal C}_2\times {\cal C}_3)\oplus {\cal C}_{k+1}$ or $L\cong {\cal C}_{n-k-5}\oplus M\oplus {\cal C}_k$, where $M$ is one of the following bounded lattices, each of which can be easily proven to have exactly $9$ lattice congruences, in whose Hasse diagrams we indicate the pair $a,b$ of elements of $L$ that generates $\alpha $ as in Lemma \ref{pWCLcg}, (\ref{pWCLcg1}):\vspace*{-2pt}

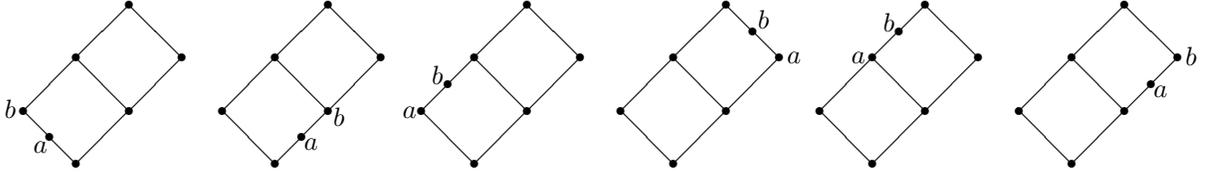
\begin{center}\begin{tabular}{cccccc}
\hspace*{-15pt}
\begin{picture}(40,60)(0,0)
\put(20,0){\circle*{3}}
\put(0,20){\circle*{3}}
\put(40,20){\circle*{3}}
\put(20,40){\circle*{3}}
\put(60,40){\circle*{3}}
\put(40,60){\circle*{3}}
\put(10,10){\circle*{3}}
\put(4,4){$a$}
\put(-7,17){$b$}
\put(20,0){\line(1,1){40}}
\put(20,0){\line(-1,1){20}}
\put(40,60){\line(-1,-1){40}}
\put(40,60){\line(1,-1){20}}
\put(20,40){\line(1,-1){20}}
\end{picture}
&\hspace*{20pt}
\begin{picture}(40,60)(0,0)
\put(20,0){\circle*{3}}
\put(0,20){\circle*{3}}
\put(40,20){\circle*{3}}
\put(20,40){\circle*{3}}
\put(60,40){\circle*{3}}
\put(40,60){\circle*{3}}
\put(30,10){\circle*{3}}
\put(31,5){$a$}
\put(42,14){$b$}
\put(20,0){\line(1,1){40}}
\put(20,0){\line(-1,1){20}}
\put(40,60){\line(-1,-1){40}}
\put(40,60){\line(1,-1){20}}
\put(20,40){\line(1,-1){20}}
\end{picture}
&\hspace*{20pt}
\begin{picture}(40,60)(0,0)
\put(20,0){\circle*{3}}
\put(0,20){\circle*{3}}
\put(40,20){\circle*{3}}
\put(20,40){\circle*{3}}
\put(60,40){\circle*{3}}
\put(40,60){\circle*{3}}
\put(10,30){\circle*{3}}
\put(4,30){$b$}
\put(-7,17){$a$}
\put(20,0){\line(1,1){40}}
\put(20,0){\line(-1,1){20}}
\put(40,60){\line(-1,-1){40}}
\put(40,60){\line(1,-1){20}}
\put(20,40){\line(1,-1){20}}
\end{picture}
&\hspace*{20pt}
\begin{picture}(40,60)(0,0)
\put(20,0){\circle*{3}}
\put(0,20){\circle*{3}}
\put(40,20){\circle*{3}}
\put(20,40){\circle*{3}}
\put(60,40){\circle*{3}}
\put(40,60){\circle*{3}}
\put(50,50){\circle*{3}}
\put(52,51){$b$}
\put(63,38){$a$}
\put(20,0){\line(1,1){40}}
\put(20,0){\line(-1,1){20}}
\put(40,60){\line(-1,-1){40}}
\put(40,60){\line(1,-1){20}}
\put(20,40){\line(1,-1){20}}
\end{picture}
&\hspace*{20pt}
\begin{picture}(40,60)(0,0)
\put(20,0){\circle*{3}}
\put(0,20){\circle*{3}}
\put(40,20){\circle*{3}}
\put(20,40){\circle*{3}}
\put(60,40){\circle*{3}}
\put(40,60){\circle*{3}}
\put(30,50){\circle*{3}}
\put(24,50){$b$}
\put(12,38){$a$}
\put(20,0){\line(1,1){40}}
\put(20,0){\line(-1,1){20}}
\put(40,60){\line(-1,-1){40}}
\put(40,60){\line(1,-1){20}}
\put(20,40){\line(1,-1){20}}
\end{picture}
&\hspace*{20pt}
\begin{picture}(40,60)(0,0)
\put(20,0){\circle*{3}}
\put(0,20){\circle*{3}}
\put(40,20){\circle*{3}}
\put(20,40){\circle*{3}}
\put(60,40){\circle*{3}}
\put(40,60){\circle*{3}}
\put(50,30){\circle*{3}}
\put(51,25){$a$}
\put(63,37){$b$}
\put(20,0){\line(1,1){40}}
\put(20,0){\line(-1,1){20}}
\put(40,60){\line(-1,-1){40}}
\put(40,60){\line(1,-1){20}}
\put(20,40){\line(1,-1){20}}
\end{picture}\end{tabular}\end{center}\vspace*{-3pt}

\noindent out of the eight lattices enumerated above, w.r.t. its their unique trivial weak complementations, the first two have exactly $2^{n-4}+1$ congruences, while the other six have exactly $9\cdot 2^{n-8}+1$ congruences;

$\bullet \ $ $L/\alpha \cong {\cal C}_{n-r-s-5}\oplus {\cal C}_2^2\oplus {\cal C}_r\oplus {\cal C}_2^2\oplus {\cal C}_s$ for some $r,s\in \N ^*$ such that $s>1$ and $r+s\leq n-6$, so that $L\cong {\cal C}_{n-r-s-4}\oplus {\cal C}_2^2\oplus {\cal C}_r\oplus {\cal C}_2^2\oplus {\cal C}_s$ or $L\cong {\cal C}_{n-r-s-5}\oplus {\cal C}_2^2\oplus {\cal C}_{r+1}\oplus {\cal C}_2^2\oplus {\cal C}_s$ or $L\cong {\cal C}_{n-r-s-5}\oplus {\cal C}_2^2\oplus {\cal C}_r\oplus {\cal C}_2^2\oplus {\cal C}_{s+1}$ or $L\cong {\cal C}_{n-r-s-5}\oplus {\cal N}_5\oplus {\cal C}_r\oplus {\cal C}_2^2\oplus {\cal C}_s$ or $L\cong {\cal C}_{n-r-s-5}\oplus {\cal C}_2^2\oplus {\cal C}_r\oplus {\cal N}_5\oplus {\cal C}_s$, each of which can only be endowed with the trivial weak complementation, w.r.t. the first three have exactly $2^{n-4}+1$ congruences, while the latter two have exactly $5\cdot 2^{n-7}+1$ congruences;

\noindent $\blacksquare \ $ $|L/\alpha |=n-2$, so that, by the induction hypothesis, $L/\alpha \cong {\cal C}_{n-k-4}\oplus {\cal C}_2^2\oplus {\cal C}_k$ for some $k\in [2,n-5]$ and, since $n\geq 9$, we can be in neither of the cases \textcircled{$\alpha $} or \textcircled{$\beta\! _2$} in Lemma \ref{pWCLcg}, (\ref{pWCLcg2}), thus we are in case \textcircled{$\beta\! _1$}, hence $L\cong {\cal C}_{n-k-h-4}\oplus {\cal C}_2^2\oplus {\cal C}_h\oplus {\cal C}_2^2\oplus {\cal C}_k$ for some $h\in [1,n-k-5]$ or $L\cong {\cal C}_{n-k-4}\oplus {\cal C}_2^2\oplus {\cal C}_h\oplus {\cal C}_2^2\oplus {\cal C}_k$ for some $h\in [1,k-3]$ or $L\cong {\cal C}_{n-k-4}\oplus ({\cal C}_2\times {\cal C}_3)\oplus {\cal C}_k$, each of which can only be endowed with the trivial weak complementation, w.r.t. which it has exactly $2^{n-4}+1$ congruences;

\noindent $\blacksquare \ $ $|L/\alpha |=n-3$, so that $L/\alpha \cong {\cal C}_{n-3}$ by the induction hypothesis and, since $n\geq 9$, we can be in neither of the cases \textcircled{$\gamma $}, \textcircled{$\delta _1$}, \textcircled{$\delta _2$}, \textcircled{$\epsilon _2$}, \textcircled{$\varphi\! _2$} or \textcircled{$\psi\! _2$} in Lemma \ref{pWCLcg}, (\ref{pWCLcg3}), hence we are in one of the following cases:

\textcircled{$\delta _3$}, \textcircled{$\delta _4$}, \textcircled{$\epsilon _1$}: in each of these cases $L\cong {\cal C}_{n-k-2}\oplus {\cal C}_2^2\oplus {\cal C}_k$ for some $k\in [1,n-3]$, which contradicts the fact that, with the notations in Lemma \ref{pWCLcg}, (\ref{pWCLcg3}), $a/Cg_L(a,b)=[a,b\vee c]_L\cong {\cal C}_2^2$;

\textcircled{$\varphi\! _1$} or \textcircled{$\psi\! _1$}: $L\cong {\cal C}_{n-k-3}\oplus {\cal N}_5\oplus {\cal C}_k$ for some $k\in [2,n-4]$, which has $5\cdot 2^{n-6}+1$ congruences w.r.t. its unique weak complementation, contradicting the hypothesis;

\textcircled{$\chi $}: since $L/\alpha $ is a chain and thus $(b\vee c)/\alpha $ is comparable to $e/\alpha $, we can not be in case \textcircled{$\chi\! _2$} and we have $L\cong {\cal C}_{n-k-4}\oplus ({\cal C}_2\times {\cal C}_3)\oplus {\cal C}_k$ for some $k\in [1,n-5]$, which has $2^{n-4}+1$ congruences w.r.t. its unique weak complementation if $k>1$ and $2^{n-5}+1$ congruences w.r.t. any of its weak complementations if $k=1$, as noticed in Example \ref{C2xC3}.

Therefore $|{\rm Con}_{\WCL }(L,\cdot ^{\Delta })|\leq 2^{n-4}+1$.

\noindent (\ref{cgwclfin11}) Assume that $|{\rm Con}_{\WCL }(L,\cdot ^{\Delta })|=2^{n-4}+1$, so that $|{\rm Con}_{\WCL }(L/\alpha ,\cdot ^{\Delta [\alpha ]})|\geq 2^{n-5}+1/2$ by Lemma \ref{atcg} and thus $|{\rm Con}_{\WCL }(L/\alpha ,\cdot ^{\Delta [\alpha ]})|\geq 2^{n-5}+1$.

Then, by the proof of (\ref{cgwclfin10}), it follows that: $L\cong {\cal C}_{n-r-s+3}\oplus ({\cal C}_r\boxplus {\cal C}_s)$ for some $r,s\in \N \setminus \{0,1,2\}$ such that $r+s\leq n+1$ and $\cdot ^{\Delta }$ is trivial if $r+s>6$ or $L\cong {\cal C}_{n-k-4}\oplus ({\cal C}_2\times {\cal C}_3)\oplus {\cal C}_k$ for some $k\in [2,n-5]$ or $L\cong {\cal C}_{n-r-s-4}\oplus {\cal C}_2^2\oplus {\cal C}_r\oplus {\cal C}_2^2\oplus {\cal C}_s$ for some $r,s\in \N ^*$ such that $s>1$ and $r+s\leq n-5$.\end{proof}

\begin{corollary} For any $n\in \N ^*$, any lattice $L$ with $|L|=n$ and any dual weak complementation $\cdot ^{\nabla }$ on $L$, we have:\begin{enumerate}
\item\label{cgwdclfin1} $|{\rm Con}_{\WDCL }(L,\cdot ^{\nabla })|\leq 2^{n-2}+1$;
\item\label{cgwdclfin2} $|{\rm Con}_{\WDCL }(L,\cdot ^{\nabla })|=2^{n-2}+1$ iff ${\rm Con}_{\WDCL }(L,\cdot ^{\nabla })\cong {\cal C}_2^{n-2}\oplus {\cal C}_2$ iff $n\geq 2$ and $L\cong {\cal C}_n$;
\item\label{cgwdclfin3} $|{\rm Con}_{\WDCL }(L,\cdot ^{\nabla })|=2^{n-2}$ iff $n=4$ and ${\rm Con}_{\WDCL }(L,\cdot ^{\nabla })\cong {\cal C}_2^2$ iff $L\cong {\cal C}_2^2$ and $\cdot ^{\nabla }$ is the Boolean complementation;
\item\label{cgwdclfin4} $|{\rm Con}_{\WDCL }(L,\cdot ^{\nabla })|<2^{n-2}$ iff $|{\rm Con}_{\WDCL }(L,\cdot ^{\nabla })|\leq 2^{n-3}+1$;
\item\label{cgwdclfin5} $|{\rm Con}_{\WDCL }(L,\cdot ^{\nabla })|=2^{n-3}+1$ iff ${\rm Con}_{\WDCL }(L,\cdot ^{\nabla })\cong {\cal C}_2^{n-3}\oplus {\cal C}_2$ iff $n\geq 5$ and $L\cong {\cal C}_k\oplus {\cal C}_2^2\oplus {\cal C}_{n-k-2}$ for some $k\in [2,n-3]$;
\item\label{cgwdclfin6} $|{\rm Con}_{\WDCL }(L,\cdot ^{\nabla })|=3\cdot 2^{n-5}$ iff $n=5$ and ${\rm Con}_{\WDCL }(L,\cdot ^{\nabla })\cong {\cal C}_3$ or $n=6$ and ${\rm Con}_{\WCL }(L,\cdot ^{\nabla })\cong {\cal C}_2\times {\cal C}_3$ iff $L\cong {\cal N}_5$ or $L\cong {\cal C}_2\times {\cal C}_3$ and $\cdot ^{\nabla }$ is the direct product of the trivial dual weak complementations on the chains ${\cal C}_2$ and ${\cal C}_3$;
\item\label{cgwdclfin7} if $n\neq 5$ and $(L,\cdot ^{\nabla })\ncong _{\WDCL }({\cal C}_2,\cdot ^{\nabla {\cal C}_2})\times ({\cal C}_3,\cdot ^{\nabla {\cal C}_3})$, then: $|{\rm Con}_{\WDCL }(L,\cdot ^{\nabla })|\leq 2^{n-3}$ iff $|{\rm Con}_{\WDCL }(L,\cdot ^{\nabla })|\leq 5\cdot 2^{n-6}+1$;
\item\label{cgwdclfin8} $|{\rm Con}_{\WDCL }(L,\cdot ^{\nabla })|=5\cdot 2^{n-6}+1$ iff ${\rm Con}_{\WDCL }(L,\cdot ^{\nabla })\cong ({\cal C}_2^{n-6}\times ({\cal C}_2\oplus {\cal C}_2^2))\oplus {\cal C}_2$ iff $n\geq 6$ and $L\cong {\cal C}_k\oplus {\cal N}_5\oplus {\cal C}_{n-k-3}$ for some $k\in [2,n-4]$;
\item\label{cgwdclfin9} $|{\rm Con}_{\WDCL }(L,\cdot ^{\nabla })|=5\cdot 2^{n-6}$ iff $n=6$ and either ${\rm Con}_{\WDCL }(L,\cdot ^{\nabla })\cong {\cal C}_2\oplus {\cal C}_2^2$ or ${\rm Con}_{\WDCL }(L,\cdot ^{\nabla })\cong {\cal C}_2^2\oplus {\cal C}_2$ iff one of the following holds:

$\bullet \ $ $L\cong {\cal C}_2\times {\cal C}_3$ and $\cdot ^{\nabla }$ is neither trivial nor equal to the direct product of the dual weak complementations on the chains ${\cal C}_2$ and ${\cal C}_3$ (case in which ${\rm Con}_{\WDCL }(L,\cdot ^{\nabla })\cong {\cal C}_2\oplus {\cal C}_2^2$);

$\bullet \ $  $L\cong {\cal C}_3\boxplus {\cal C}_5$ or $L\cong {\cal C}_4\boxplus {\cal C}_4$ and $\cdot ^{\nabla }$ is trivial (case in which ${\rm Con}_{\WDCL }(L,\cdot ^{\nabla })\cong {\cal C}_2^2\oplus {\cal C}_2$);
\item\label{cgwdclfin10} $|{\rm Con}_{\WDCL }(L,\cdot ^{\nabla })|<5\cdot 2^{n-6}$ iff $|{\rm Con}_{\WDCL }(L,\cdot ^{\nabla })|\leq 2^{n-4}+1$;
\item\label{cgwdclfin11} $|{\rm Con}_{\WDCL }(L,\cdot ^{\nabla })|=2^{n-4}+1$ iff $n\geq 5$ and ${\rm Con}_{\WDCL }(L,\cdot ^{\nabla })\cong {\cal C}_2^{n-4}\oplus {\cal C}_2$ iff one of the following holds:

$\bullet \ $ $n\geq 5$ and $L\cong ({\cal C}_r\boxplus {\cal C}_s)\oplus {\cal C}_{n-r-s+3}$ for some $r,s\in \N \setminus \{0,1,2\}$ such that $r+s\leq n+1$ and, if $r+s>6$ (that is if $L\ncong {\cal C}_2^2\oplus {\cal C}_{n-3}$), then $\cdot ^{\nabla }$ is trivial;

$\bullet \ $ $n\geq 7$ and $L\cong {\cal C}_k\oplus ({\cal C}_2\times {\cal C}_3)\oplus {\cal C}_{n-k-4}$ for some $k\in [2,n-5]$;

$\bullet \ $ $n\geq 8$ and $L\cong {\cal C}_s\oplus {\cal C}_2^2\oplus {\cal C}_r\oplus {\cal C}_2^2\oplus {\cal C}_{n-r-s-4}$ for some $r,s\in \N ^*$ such that $s>1$ and $r+s\leq n-5$.\end{enumerate}\label{cgwdclfin}\end{corollary}

\begin{proof} By duality, from Theorem \ref{cgwclfin}.\end{proof}

\begin{corollary} For any $n\in \N ^*$, any lattice $L$ with $|L|=n$ and any weak dicomplementation $(\cdot ^{\Delta },\cdot ^{\nabla })$ on $L$, we have:\begin{enumerate}
\item\label{cgwdlfin0} $|{\rm Con}_{\WDL }(L,\cdot ^{\Delta },\cdot ^{\nabla })|\leq 2^{n-1}$;
\item\label{cgwdlfin0n} $|{\rm Con}_{\WDL }(L,\cdot ^{\Delta },\cdot ^{\nabla })|=2^{n-1}$ iff $n\in \{1,2\}$;
\item\label{cgwdlfin3} $|{\rm Con}_{\WDL }(L,\cdot ^{\Delta },\cdot ^{\nabla })|=2^{n-2}$ iff $n=4$ and ${\rm Con}_{\WDL }(L,\cdot ^{\Delta },\cdot ^{\nabla })\cong {\cal C}_2^2$ iff $L\cong {\cal C}_2^2$ and $\cdot ^{\Delta }=\cdot ^{\nabla }$ is the Boolean complementation;
\item\label{cgwdlfin1} if $L\ncong {\cal C}_2^2$ or its weak dicomplementation is not Boolean, then: $|{\rm Con}_{\WDL }(L,\cdot ^{\Delta },\cdot ^{\nabla })|<2^{n-1}$ iff $|{\rm Con}_{\WDL }(L,\linebreak \cdot ^{\Delta },\cdot ^{\nabla })|\leq 2^{n-3}+1$;
\item\label{cgwdlfin2} $|{\rm Con}_{\WDL }(L,\cdot ^{\Delta },\cdot ^{\nabla })|=2^{n-3}+1$ iff ${\rm Con}_{\WDL }(L,\cdot ^{\Delta },\cdot ^{\nabla })\cong {\cal C}_2^{n-3}\oplus {\cal C}_2$ iff $n\geq 3$ and $L\cong {\cal C}_n$;
\item\label{cgwdlfin4} $|{\rm Con}_{\WDL }(L,\cdot ^{\Delta },\cdot ^{\nabla })|\leq 2^{n-3}$ iff $|{\rm Con}_{\WDL }(L,\cdot ^{\Delta },\cdot ^{\nabla })|\leq 2^{n-4}+1$;
\item\label{cgwdlfin5} $|{\rm Con}_{\WDL }(L,\cdot ^{\Delta },\cdot ^{\nabla })|=2^{n-4}+1$ iff ${\rm Con}_{\WDL }(L,\cdot ^{\Delta },\cdot ^{\nabla })\cong {\cal C}_2^{n-4}\oplus {\cal C}_2$ iff one of the following holds:

$\bullet \ $ $n\geq 5$, $L\cong {\cal C}_k\boxplus {\cal C}_{n-k-2}$ for some $k\in [3,n-2]$ and $(\cdot ^{\Delta },\cdot ^{\nabla })$ is the trivial weak dicomplementation on $L$;

$\bullet \ $ $n\geq 6$ and $L\cong {\cal C}_k\oplus {\cal C}_2^2\oplus {\cal C}_{n-k-2}$ for some $k\in [2,n-4]$.\end{enumerate}\label{cgwdlfin}\end{corollary}

\begin{proof} Since ${\rm Con}_{\WDL }(L,\cdot ^{\Delta },\cdot ^{\nabla })={\rm Con}_{\WCL }(L,\cdot ^{\Delta })\cap {\rm Con}_{\WDCL }(L,\cdot ^{\nabla })$ and thus $|{\rm Con}_{\WDL }(L,\cdot ^{\Delta },\cdot ^{\nabla })|\leq \linebreak |{\rm Con}_{\WCL }(L,\cdot ^{\Delta })|,|{\rm Con}_{\WDCL }(L,\cdot ^{\nabla })|$, it suffices to calculate the numbers of lattice congruences of the lattices in Theorem \ref{cgwclfin} and Corollary \ref{cgwdclfin} that preserve the weak complementations in Theorem \ref{cgwclfin} and the dual weak complementations in Corollary \ref{cgwdclfin}. Actually, we only need to look at the lattices that appear both in Theorem \ref{cgwclfin} and in Corollary \ref{cgwdclfin}; for instance, if $n\geq 6$, then, by Theorem \ref{cgwclfin}, (\ref{cgwclfin11}), and Corollary \ref{cgwdclfin}, ${\cal C}_{n-5}\oplus ({\cal C}_3\boxplus {\cal C}_5)$ can be organized as a weakly complemented lattice with $2^{n-4}+1$ congruences, but not as a dual weakly complemented lattice with at least $2^{n-4}+1$ congruences.

As in the proof of Theorem \ref{cgwclfin} for weak complementations, we use the remarks at the beginning of Section \ref{ordsum} and Theorem \ref{mainth} to determine the weak dicomplementations. We determine the congruence lattices w.r.t. these weak dicomplementations by using Corollary \ref{cgordsumwdl} along with Proposition \ref{cghsumnontriv} and Corollaries \ref{cghsumdualnontriv} and \ref{cghsumdinontriv}.

${\cal C}_n$ can only be endowed with the trivial weak dicomplementation, w.r.t. which ${\cal C}_1$ has $1=2^{1-1}$ congruences, ${\cal C}_2$ has $2=2^{2-1}$ congruences and, if $n\geq 3$, then ${\rm Con}_{\WDL }({\cal C}_n,\cdot ^{\Delta {\cal C}_n},\cdot ^{\nabla {\cal C}_n})\cong {\rm Con}({\cal C}_{n-2})\oplus {\cal C}_2\cong {\cal C}_2^{n-3}\oplus {\cal C}_2$, thus $|{\rm Con}_{\WDL }({\cal C}_n,\cdot ^{\Delta {\cal C}_n},\cdot ^{\nabla {\cal C}_n})|=2^{n-3}+1$.

We use the notations for weak complementations in the proof of Theorem \ref{cgwclfin} and we also denote, in the case when $L$ has exactly two distinct atoms, by $\cdot ^{\nabla _{{\rm At}(L)}}$ the unique nontrivial dual weak complementation on $L$, as in Proposition \ref{2coat}.

By Corollary \ref{cghsumdinontriv}, ${\cal C}_2^2$ has four weak dicomplementations, w.r.t. which: $|{\rm Con}_{\WDL }({\cal C}_2^2,\cdot ^{\Delta _{{\rm CoAt}({\cal C}_2^2)}},\cdot ^{\nabla _{{\rm At}({\cal C}_2^2)}})|=4=2^{|{\cal C}_2^2|-2}$, while $|{\rm Con}_{\WDL }({\cal C}_2^2,\cdot ^{\Delta {\cal C}_2^2},\cdot ^{\nabla {\cal C}_2^2})|=|{\rm Con}_{\WDL }({\cal C}_2^2,\cdot ^{\Delta _{{\rm CoAt}({\cal C}_2^2)}},\cdot ^{\nabla {\cal C}_2^2})|=|{\rm Con}_{\WDL }({\cal C}_2^2,\cdot ^{\Delta {\cal C}_2^2},\cdot ^{\nabla _{{\rm At}({\cal C}_2^2)}})|=2=2^{|{\cal C}_2^2|-3}$.

By Example \ref{C2xC3}, w.r.t. any of its $16$ weak dicomplementations, ${\cal C}_2\times {\cal C}_3$ has at most $4=2^{|{\cal C}_2\times {\cal C}_3|-4}$ congruences.

If $n\geq 5$, then ${\cal C}_{n-3}\oplus {\cal C}_2^2$ can only be endowed with two weak dicomplementations: the trivial one, $(\cdot ^{\Delta {\cal C}_{n-3}\oplus {\cal C}_2^2},\cdot ^{\nabla {\cal C}_{n-3}\oplus {\cal C}_2^2})$, and $(\cdot ^{\Delta _{{\rm CoAt}({\cal C}_{n-3}\oplus {\cal C}_2^2)}},\cdot ^{\nabla {\cal C}_{n-3}\oplus {\cal C}_2^2})$, w.r.t. which ${\rm Con}_{\WDL }({\cal C}_{n-3}\oplus {\cal C}_2^2,\cdot ^{\Delta {\cal C}_{n-3}\oplus {\cal C}_2^2},\cdot ^{\nabla {\cal C}_{n-3}\oplus {\cal C}_2^2})={\rm Con}_{\WDL }({\cal C}_{n-3}\oplus {\cal C}_2^2,\cdot ^{\Delta _{{\rm CoAt}({\cal C}_{n-3}\oplus {\cal C}_2^2)}},\cdot ^{\nabla {\cal C}_{n-3}\oplus {\cal C}_2^2})\cong {\rm Con}({\cal C}_{n-4})\oplus {\cal C}_2\cong {\cal C}_2^{n-5}\oplus {\cal C}_2$, thus $|{\rm Con}_{\WDL }({\cal C}_{n-3}\oplus {\cal C}_2^2,\cdot ^{\Delta {\cal C}_{n-3}\oplus {\cal C}_2^2},\linebreak \cdot ^{\nabla {\cal C}_{n-3}\oplus {\cal C}_2^2})|=|{\rm Con}_{\WDL }({\cal C}_{n-3}\oplus {\cal C}_2^2,\cdot ^{\Delta _{{\rm CoAt}({\cal C}_{n-3}\oplus {\cal C}_2^2)}},\cdot ^{\nabla {\cal C}_{n-3}\oplus {\cal C}_2^2})|=2^{n-5}+1$. Dually, if $n\geq 5$, then ${\cal C}_2^2\oplus {\cal C}_{n-3}$ can only be endowed with two weak dicomplementations, w.r.t. which $|{\rm Con}_{\WDL }({\cal C}_2^2\oplus {\cal C}_{n-3},\cdot ^{\Delta {\cal C}_2^2\oplus {\cal C}_{n-3}},\cdot ^{\nabla {\cal C}_2^2\oplus {\cal C}_{n-3}})|=|{\rm Con}_{\WDL }({\cal C}_2^2\oplus {\cal C}_{n-3},\cdot ^{\Delta {\cal C}_2^2\oplus {\cal C}_{n-3}},\cdot ^{\nabla _{{\rm At}({\cal C}_2^2\oplus {\cal C}_{n-3})}})|=2^{n-5}+1$.

If $n\geq 6$, then, for any $k\in [2,n-4]$, ${\cal C}_k\oplus {\cal C}_2^2\oplus {\cal C}_{n-k-2}$ can only be endowed with the trivial weak dicomplementation, w.r.t. which ${\rm Con}_{\WDL }({\cal C}_k\oplus {\cal C}_2^2\oplus {\cal C}_{n-k-2},\cdot ^{\Delta {\cal C}_k\oplus {\cal C}_2^2\oplus {\cal C}_{n-k-2}},\cdot ^{\nabla {\cal C}_k\oplus {\cal C}_2^2\oplus {\cal C}_{n-k-2}})\cong {\rm Con}({\cal C}_{k-1}\oplus {\cal C}_2^2\oplus {\cal C}_{n-k-3})\oplus {\cal C}_2\cong {\cal C}_{k-2+2+n-k-4}\oplus {\cal C}_2={\cal C}_2^{n-4}\oplus {\cal C}_2$, thus $|{\rm Con}_{\WDL }({\cal C}_k\oplus {\cal C}_2^2\oplus {\cal C}_{n-k-2},\cdot ^{\Delta {\cal C}_k\oplus {\cal C}_2^2\oplus {\cal C}_{n-k-2}},\cdot ^{\nabla {\cal C}_k\oplus {\cal C}_2^2\oplus {\cal C}_{n-k-2}})|=2^{n-4}+1$.

By Example \ref{exhsums}, if $n\geq 5$, then, for any $k\in [3,n-2]$, ${\cal C}_k\boxplus {\cal C}_{n-k+2}$ has four weak dicomplementations, w.r.t. which:

$|{\rm Con}_{\WDL }({\cal C}_k\boxplus {\cal C}_{n-k+2},\cdot ^{\Delta {\cal C}_k\boxplus {\cal C}_{n-k+2}},\cdot ^{\nabla {\cal C}_k\boxplus {\cal C}_{n-k+2}})|=2^{n+2-6}+1=2^{n-4}+1$;

$|{\rm Con}_{\WDL }({\cal C}_k\boxplus {\cal C}_{n-k+2},\cdot ^{\Delta _{{\rm CoAt}({\cal C}_k\boxplus {\cal C}_{n-k+2})}},\cdot ^{\nabla {\cal C}_k\boxplus {\cal C}_{n-k+2}})|=|{\rm Con}_{\WDL }({\cal C}_k\boxplus {\cal C}_{n-k+2},\cdot ^{\Delta {\cal C}_k\boxplus {\cal C}_{n-k+2}},\cdot ^{\nabla _{{\rm At}({\cal C}_k\boxplus {\cal C}_{n-k+2})}})|\linebreak =\begin{cases}2^{n-k+2-4}+1=2^{n-5}+1,& k=3,\\ 2^{n-6}+1,& k\geq 4;\end{cases}$

$|{\rm Con}_{\WDL }({\cal C}_k\boxplus {\cal C}_{n-k+2},\cdot ^{\Delta _{{\rm CoAt}({\cal C}_k\boxplus {\cal C}_{n-k+2})}},\cdot ^{\nabla _{{\rm At}({\cal C}_k\boxplus {\cal C}_{n-k+2})}})|=\begin{cases}2\leq 2^{n-4},& k\leq 4,n-k\leq 2;\\ 2^{n-k-3}+1\leq 2^{n-6}+1,& k\leq 4,n-k\geq 3;\\ 2^{n-8}+1,& k\geq 5,n-k\geq 3.\end{cases}$

If $n\geq 8$ and $L\cong {\cal C}_k\oplus ({\cal C}_2\times {\cal C}_3)\oplus {\cal C}_{n-k-4}$ for some $k\in [2,n-6]$ or $n\geq 9$ and $L\cong {\cal C}_{n-r-s-4}\oplus {\cal C}_2^2\oplus {\cal C}_r\oplus {\cal C}_2^2\oplus {\cal C}_s$ for some $r,s\in \N ^*$ such that $s>1$ and $r+s\leq n-6$, then $L$ can only be endowed with the trivial weak dicomplementation and ${\rm Con}_{\WDL }(L,\cdot ^{\Delta L},\cdot ^{\nabla L})\cong {\cal C}_2^{n-5}\oplus {\cal C}_2$, thus $|{\rm Con}_{\WDL }(L,\cdot ^{\Delta L},\cdot ^{\nabla L})|=2^{n-5}+1$.\end{proof}

\section*{Acknowledgements}

This work was supported by the research grant number IZSEZO\_186586/1, awarded to the project {\em Re\-ti\-cu\-la\-tions of Concept Algebras} by the Swiss National Science Foundation, within the programme Scientific Exchanges.

\end{document}